\documentclass[smallcondensed]{svjour3}

\usepackage{etoolbox}
\usepackage{amsmath,amssymb}
\usepackage{subcaption}
\usepackage{tikz}
\usepackage{graphicx}
\usepackage{xcolor}
\usepackage[colorlinks=true, allcolors=blue]{hyperref}
\usepackage{mathtools}
\usepackage{cleveref}   
\usepackage{booktabs}
\usepackage{hyperref}

\usepackage{algorithm} 
\usepackage[noend]{algpseudocode}

\makeatletter
\newcommand*{\declarecommand}{%
  \@star@or@long\declare@command
}
\newcommand*{\declare@command}[1]{%
  \provide@command{#1}{}%
  \renew@command{#1}%
}
\makeatother
\declarecommand{\x}{{\mathbf{x}}}
\declarecommand{\y}{{\mathbf{y}}}
\declarecommand{\z}{{\mathbf{z}}}
\declarecommand{\r}{{\mathbf{r}}}
\declarecommand{\n}{{\mathbf{n}}}
\declarecommand{\u}{{\mathbf{u}}}
\declarecommand{\v}{{\mathbf{v}}}
\declarecommand{\w}{{\mathbf{w}}}
\declarecommand{\f}{{\mathbf{f}}}
\declarecommand{\c}{{\mathbf{c}}}
\declarecommand{\bom}{{\boldsymbol{\omega}}}
\declarecommand{\X}{{\mathbf{X}}}
\declarecommand{\XE}{\X_\text{ext}}
\declarecommand{\XI}{\X_\text{int}}
\declarecommand{\tX}{\tilde\X}
\declarecommand{\tx}{\tilde x}
\declarecommand{\ty}{\tilde y}
\declarecommand{\btau}{{\boldsymbol{\tau}}}
\declarecommand{\bsig}{{\boldsymbol{\sigma}}}
\declarecommand{\bSig}{{\boldsymbol{\Sigma}}}
\declarecommand{\bgamma}{{\boldsymbol{\gamma}}}
\declarecommand{\S}{\mathcal{S}}
\declarecommand{\D}{\mathcal{D}}
\declarecommand{\SG}{\S_\Gamma}
\declarecommand{\DG}{\D_\Gamma}
\declarecommand{\G}{\mathcal{G}}
\declarecommand{\grad}{\nabla}
\declarecommand{\Id}{\mathbb{I}}

\newcommand{\bi}{\begin{itemize}}
\newcommand{\ei}{\end{itemize}}
\newcommand{\ben}{\begin{enumerate}}
\newcommand{\een}{\end{enumerate}}
\newcommand{\be}{\begin{equation}}
\newcommand{\ee}{\end{equation}}
\newcommand{\bea}{\begin{eqnarray}} 
\newcommand{\eea}{\end{eqnarray}}
\newcommand{\ba}{\begin{align}} 
\newcommand{\ea}{\end{align}}
\newcommand{\bse}{\begin{subequations}} 
\newcommand{\ese}{\end{subequations}}
\newcommand{\bc}{\begin{center}}
\newcommand{\ec}{\end{center}}
\newcommand{\bfi}{\begin{figure}}
\newcommand{\efi}{\end{figure}}
\newcommand{\ca}[2]{\caption{#1 \label{#2}}}
\newcommand{\ig}[2]{\includegraphics[#1]{#2}}
\newcommand{\bmp}[1]{\begin{minipage}{#1}}
\newcommand{\emp}{\end{minipage}}

\newcommand{\tbox}[1]{{\mbox{\scriptsize #1}}}
\newcommand{\mbf}[1]{{\mathbf #1}}
\newcommand{\half}{\mbox{\small $\frac{1}{2}$}}
\newcommand{\vt}[2]{\left[\begin{array}{r}#1\\#2\end{array}\right]} 
\newcommand{\R}{\mathbb{R}}
\newcommand{\Z}{\mathbb{Z}}
\newcommand{\C}{\mathbb{C}}
\newcommand{\ep}{\epsilon}            
\newcommand{\bigO}{{\mathcal O}}
\DeclareMathOperator{\re}{Re}
\DeclareMathOperator{\im}{Im}

\newtheorem{thm}{Theorem}

\newtheorem{lem}[thm]{Lemma}

\newtheorem{pro}[thm]{Proposition}

\newtheorem{rmk}[thm]{Remark}
\newcommand{\pO}{{\partial\Omega}}
\newcommand{\Oc}{{\R^d\backslash\overline{\Omega}}}   
\newcommand{\emach}{\epsilon_\tbox{\rm mach}}
\newcommand{\gim}{\gamma^\tbox{\rm imag}}
\newcommand{\ta}{\tilde\alpha}           
\newcommand{\tb}{\tilde\beta}
\newcommand{\Nc}{M}                  
\declarecommand{\T}{{\mathbf{T}}}
\declarecommand{\bphi}{{\boldsymbol{\phi}}}
\declarecommand{\bpsi}{{\boldsymbol{\psi}}}
\declarecommand{\ddist}{{d_{\tbox{\rm min}}}}        
\declarecommand{\cnoise}{{\boldsymbol{\eta}}}

\begin{document}

\author{David~B.~Stein \and Alex~H.~Barnett}

\institute{David B. Stein \at
  Center for Computational Biology, Flatiron Institute, New York, NY 10010, USA
  \email{dstein@flatironinstitute.org}
\and
Alex H. Barnett \at
Center for Computational Mathematics, Flatiron Institute, New York, NY 10010, USA}

\title{Quadrature by fundamental solutions: kernel-independent layer potential evaluation for large collections of simple objects}

\maketitle

\begin{abstract}
  Well-conditioned boundary integral methods for the solution of elliptic
  boundary value problems (BVPs) are powerful tools for static and dynamic physical simulations.
  %
  When there are many close-to-touching boundaries (eg, in complex fluids) or when the solution is needed in the bulk, nearly-singular integrals must be evaluated at many targets.
  %
  We show that precomputing a linear map from surface density to an
  {\em effective source} representation
  renders this task highly efficient,
  in the common case where each object is ``simple'', ie, its smooth boundary needs only moderately many nodes. 
  %
  We present a kernel-independent method needing
  only an {\em upsampled smooth} surface quadrature,
  and one dense factorization, for each distinct shape.
  No (near-)singular quadrature rules are needed.
  The resulting effective sources are drop-in compatible with
  fast algorithms, with no local corrections nor bookkeeping.
  Our extensive numerical tests include
  2D FMM-based Helmholtz and Stokes BVPs
  with up to 1000 objects (281000 unknowns),
  and a 3D Laplace BVP with 10 ellipsoids separated by $1/30$ of a
  diameter.
  We include a rigorous analysis for analytic data in 2D and 3D.
\end{abstract}



\section{Introduction}
\label{section:introduction}

Boundary integral equations (BIEs)
are advantageous for the numerical solution of a wide variety of
linear
boundary-value problems (BVPs) in science and engineering
\cite{LIE,HW}.
They include electro/magnetostatics \cite{Moura94,yingbeale},
acoustics \cite{rokh83,kress91},
electromagnetics/optics \cite{coltonkress,CMS,laiaxi},
elastostatics/dynamics \cite{helsingbigelasto,chaillat08},
viscous fluid flow \cite{yanplatform,quaife2021hydrodynamics,sinha2016shape,nazockdast2017fast,nazockdast2017cytoplasmic},
electrohydrodynamics \cite{sorgentone2021numerical},
and many others.
BIEs also form a component in solvers for BVPs
with volume driving and/or nonlinearities
by solving for a homogeneous PDE solution which corrects
the boundary conditions \cite{mayo84,biros04,fryklund2018partition,fryklund2020integral,ludvig_nufft,young2021many}.
``Fast'' (quasi-linear scaling)
algorithms to apply the resulting discretized operators,
such as the fast multipole method (FMM) \cite{lapFMM,fmm1,CMS},
have revolutionized the size of problems that can be tackled \cite{pvfmm}.
More recently, fast direct solvers have enabled large gains when iterative
solution is inefficient \cite{hackbusch,gunnarbook,qpfds}.
Despite this progress, the issue of efficient and accurate discretization
of BIEs in complex geometries persists.
In this work we present a new tool to address this in the
common case of a large number of simple, possibly close-to-touching objects,
as can arise in numerical homogenization, porous media,
and complex fluids.

For example, and to fix notation,
let $\Omega$ be either one bounded obstacle or the union of many such
obstacles in $\R^d$,
let $L$ be a linear constant-coefficient 2nd-order
elliptic differential operator, and consider solving the BVP
\bea
{\cal L} u &=& 0 \qquad \mbox{ in } \R^d \backslash \overline{\Omega}
\label{pde}
\\
u&=& f \qquad \mbox{ on } \pO
\label{bc}
\eea
with an appropriate decay or radiation condition imposed on
$u(\x)$ as $\|\x\|\to\infty$.
In terms of the translationally invariant fundamental solution
(free space Green's function) for $\cal L$, denoted by
$G(\x,\y) = G(\x-\y)$, for $\x,\y \in \R^d$,
a common {\em layer potential} representation for the solution is,
in the scalar case,
\be
u(\x) = [(\alpha \S + \beta \D) \tau](\x) :=
\int_\pO \left( \alpha G(\x,\y) + \beta \frac{\partial G(\x,\y)}{\partial \n_\y}
\right) \tau(\y) ds_\y,
\quad \x \in \Oc.
\label{rep}
\ee
Here $ds_\y$ is the arc or surface element, $\n_\y$ the unit outward normal
at $\y\in\pO$, and the formula
defines the single-layer $\S$, and (again in the scalar case only)
double-layer $\D$ potentials.
The constants $\alpha$ and $\beta$ are given.
The unknown {\em density} function $\tau$ lives on $\pO$ and is found by solving
a so-called {\em indirect} BIE derived from \eqref{rep} by taking the exterior limit
$\x\to\pO$, using jump relations \cite{HW,coltonkress}.
In the case of Dirichlet boundary conditions, this BIE is
\be
   [\alpha S + \beta(\half I + D)] \tau = f
   ~,
   \label{bie}
\ee
where $f$ is given boundary data, and $S$ and $D$ are the principal
value {\em boundary integral operators} resulting by restricting $\S$ and $\D$
to $\pO$.
Note that, since (at least for $\pO$ smooth) $S$ and $D$ are compact,
for $\beta \neq 0$ the BIE is of Fredholm 2nd-kind;
in general $\alpha$ and $\beta$ are chosen to give this property
and to give a unique solution \cite{HW,atkinson}.

Despite this elegant framework,
in practice there remain two challenging tasks:
\ben
\item 
high-order accurate discretization of \eqref{bie},
meaning filling (or, for large problems, merely applying)
the $N \times N$ Nystr\"om matrix in a linear system
\be
A \btau = \mbf{f} ~,
\label{linsys}
\ee
which approximates \eqref{bie};
and
\item 
numerical evaluation of \eqref{rep} at target points
$\x$ including those arbitrarily close to $\pO$,
given the solution $\btau \in \C^N$ to \eqref{linsys}.
\een
Much of the difficulty of both tasks originates in the {\em singularity}
in $G(\x-\y)$ as $\x\to\y$.
Their troublesome nature for various kernels,
especially in 3D ($d=3$) and/or complex geometries,
is indicated by the large number of methods,
and its active growth as a research area
(briefly reviewed in Section~\ref{s:prior}).

Task 2 arises especially frequently in fluid simulations containing many interacting bodies, e.g. in blood and vesicular flow or sedimentation problems.
For either rigid or deformable particles
BIE solutions (or simpler hydrodynamic layer-potential evaluations \cite{ves2d})
are typically needed at every time-step, and regularizations \cite{sinha2016shape} or special near-field quadrature schemes must be used to maintain fidelity (see \Cref{s:prior}).
When non-Newtonian rheology arises, as in complex and active fluids, continuum models typically track extra stress or orientation fields \cite{saintillan2018rheology} in the bulk whose evolution requires knowledge of both hydrodynamic velocities and stresses.
Due to the difficulty of BIEs, simulations of such complex fluid
and active matter systems have instead primarily been done using finite/spectral element methods and cut cell methods in stationary geometries \cite{owens1996steady,theillard2017geometric}, and 
regularized methods \cite{cortez2001method,peskin2002immersed,li2019orientation} in moving geometries.
In certain cases such regularized methods are known to give inaccurate results, with nontrivial corrections required to ensure convergence \cite{krishnan2017fully,stein2019convergent}.
Thus robust methods for BIE with many near-boundary targets can enable complex fluid simulations in regimes that are currently hard to access.

In many applications the number of nodes needed on each distinct boundary
is ``small'' (at most a few thousand, in either 2D or 3D).
In this case dense, linear algebraic methods with $\bigO(N^3)$ cost
are practical for per-object precomputations.
This handles the diagonal (self-interaction) blocks of $A$;
an FMM, followed by local corrections,
may then apply its off-diagonal blocks. 
Focusing on Nystr\"om discretizations \cite[Ch.~12]{LIE} \cite{coltonkress}
for indirect BIEs, we exploit this idea to propose a simple but
efficient new approach to both tasks 1 and 2,
that is in large part kernel- and dimension-independent,
and furthermore is already in use \cite{young2021many}.
Its kernel-independence allows easy switching between PDEs,
or to axisymmetric, periodic, or multilayer Green's functions.
It is essentially automated in 2D, but requires parameter adjustment in 3D.
A key advantage at the implementation level is
that a {\em single} FMM-compatible representation covers
on-surface, near-surface, and far-field,
bypassing the bookkeeping that complicates high-performance codes
\cite{yanplatform}.

\bfi 
\ig{width=\textwidth}{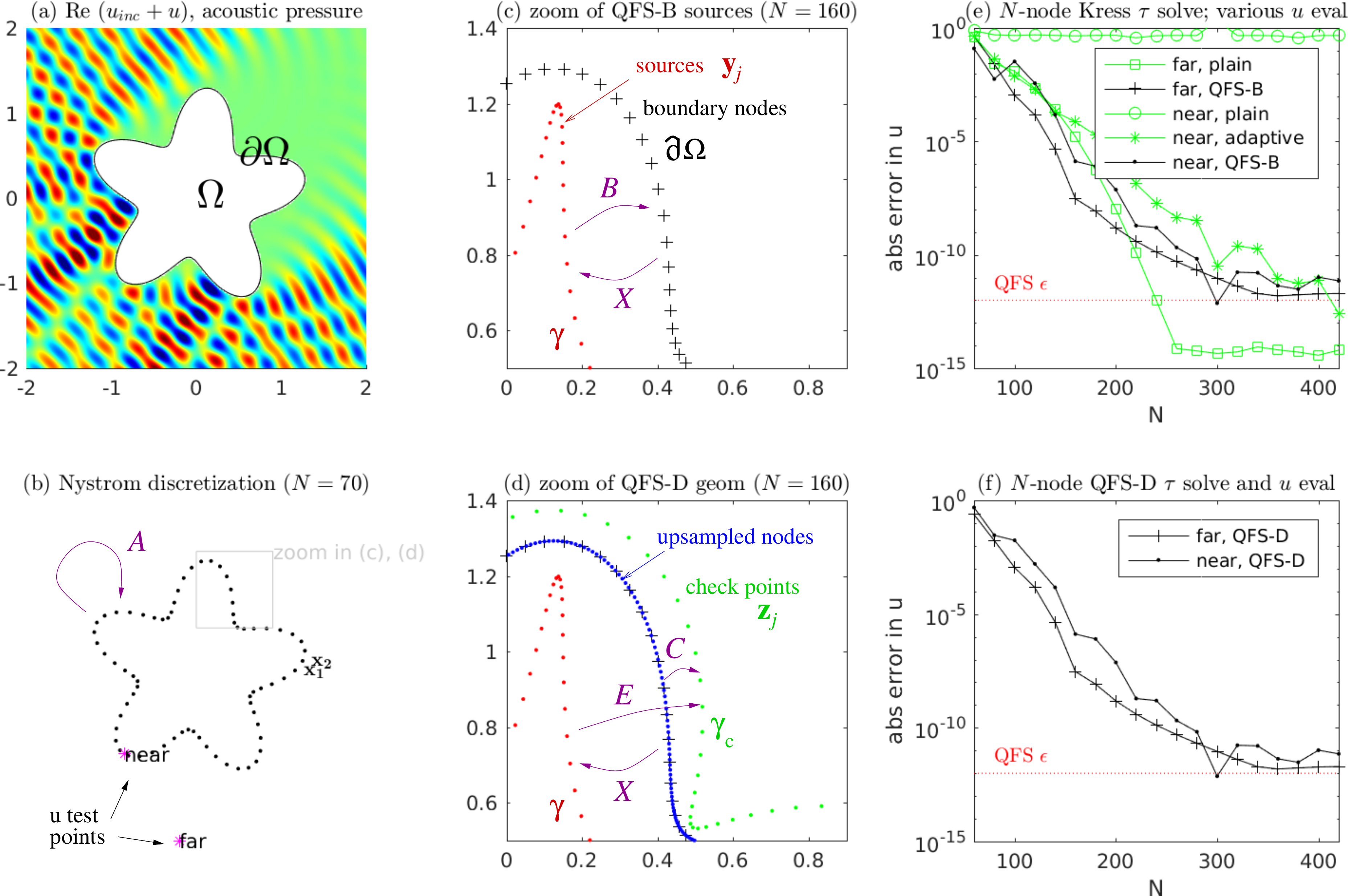}
\ca{Overview of proposal applied to an exterior Dirichlet Helmholtz 
  (sound-hard) scattering BVP from $\pO$
  parameterized in polars by $r(t) = 1 + 0.3\,\cos(5t + 0.2)$,
  that is, $\x(t) = (r(t) \cos t, r(t) \sin t)$.
  (a) Incident wave (from bottom left at $\theta=\pi/5$, wavenumber $k=20$) plus scattered wave (BVP solution $u$).
  (b) Discretization nodes (indicating action of Nystr\"om matrix $A$), and ``far'' and ``near'' test targets.
  (c) QFS-B proxy sources (red dots) and boundary nodes ($+$ symbols)
  (d) QFS-D proxy sources (red dots), check points (green dots), and
  upsampled boundary nodes (blue dots).
  (e) Errors in $u$ for various evaluation methods, having
  solved the density vector using Kress quadrature for $A$.
  (f) The full ``desingularized'' scheme: errors in $u$, having first
  solved the density using $A$ as filled via QFS-D.
  Note that the black curves converge at comparable rates to the green ones
  down to the requested tolerance $\ep$.
}{f:setup}
\efi 

Let us sketch the basic proposal in a simple 2D exterior acoustic frequency-domain scattering (Helmholtz) Dirichlet BVP, with a single boundary curve $\pO$; see Fig.~\ref{f:setup}.
The incident plane wave has wavenumber $k$, and the resulting
scattered wave $u$ solves the BVP with data the negative of
this incident wave on $\pO$.
This ensures that their sum (the physical solution shown in panel (a))
has zero Dirichlet boundary data.
In the representation \eqref{rep}, $\alpha = -ik$, $\beta=1$
(the usual ``combined field'' or CFIE \cite{coltonkress}), and $G(\r) = (i/4)H^{(1)}_0(k\|\r\|)$,
where $H_0^{(1)}$ is the Hankel function of the first kind.
Panel (b) shows the plain $N$-node periodic trapezoid rule (PTR)
quadrature used on $\pO$.

The main idea---which we call quadrature by fundamental solutions (QFS)---is to place roughly $N$ effective or proxy sources
a controlled distance from $\pO$ on its {\em non-physical}
(interior) side,
whose strengths are chosen to approximate the desired potential
\eqref{rep}, both on $\pO$
and throughout the solution domain $\Oc$.
We precompute a ``source-from-density'' matrix $X$ mapping any
smooth boundary
density sample vector $\btau$ to an equivalent proxy strength vector.
This matrix equation for $X$ is solved densely in a backward-stable fashion
by {\em collocation} (matching) of the potential,
either on the surface $\pO$ (as in panel (c)),
or on a nearby set of ``check points''
a controlled distance from $\pO$ but on the {\em physical} side (panel (d)).
Armed with $X$, given any density $\btau$
the desired potential \eqref{rep}
is well approximated by a sum over the proxy sources with strengths $\bsig= X\btau$.
This applies for targets $\x$ far from $\pO$, arbitrarily near to $\pO$, or on $\pO$ (the exterior surface limit),
and is compatible with the FMM.
This addresses task 2 above.
Furthermore, by filling the evaluation matrix $B$
(see panel (c)) from proxy sources to the desired data type (trace) on $\pO$,
then the product $BX$ is a good approximation to $A$, the Nystr\"om matrix in \eqref{linsys}, completing task 1.

In Fig.~\ref{f:setup}(e) we show convergence of potential
evaluation (task 2) for two target points (plotted in panel (b)), given a density $\btau$ already solved using the
Kress scheme \cite{kress91} generally considered a ``gold standard'' \cite{hao,helsing_helm}.
For the easy case of a far target, our QFS proposal has a similar
convergence rate as the plain PTR, down to the requested tolerance
of $\ep=10^{-12}$.
For a near target (a distance $10^{-4}$ from $\pO$, where the plain PTR of course
fails dismally), the same QFS scheme
has similar convergence to the expensive gold-standard method
of adaptive Gaussian quadrature applied to the trigonometric polynomial
interpolant of the density, again down to $\ep$.
In Fig.~\ref{f:setup}(f) we use QFS both to
fill the Nystr\"om matrix $A$ and for
potential evaluation (combining tasks 1 and 2);
again we see similar convergence.

\begin{rmk}[MFS]
  The idea of representing homogeneous PDE solutions by Green's function
  sources near the boundary has a 50-year history in the engineering community \cite{Kupradze67,doicu},
  being called the {\em method of fundamental solutions} \cite{Bo85,mfs,acper}, method of auxiliary sources \cite{fridon},
  charge simulation method \cite{Ka89,Ka96}, 1st-kind integral equations
  \cite{kangro2d,kangro3d,gonzalez09}, rational approximation \cite{hochmancorner,lightning}, etc.
  It is well known to produce exponentially ill-conditioned linear systems.
  Recently, similar ``proxy point'' ideas flourished in fast direct solvers \cite{gunnarbook},
  kernel-independent FMMs \cite{pvfmm}, and BIE quadrature \cite{qbkix}.
  Our novelty here is to use {\em off-surface} collocation
  to make a black-box general layer-potential evaluator tool,
  which can be inserted, for example, into standard {\em well-conditioned} 2nd-kind BIE frameworks.
\end{rmk}

The method's simplicity and MFS flavor
restricts
the body shapes to which it may be accurately applied.
This arises essentially from the need that the density $\tau$ and data $f$ be
smooth on the local node-spacing scale $h$.
Yet, this is also true for (non-adaptive) BIE quadrature schemes generally.
If smooth objects become extremely close ($\bigO(h^2)$ or closer),
{\em adaptive} surface quadratures are essential to capture $\tau$
\cite{helsingtut,wu2019,hedgehog}, a problem beyond even tasks 1 and 2.
We will not address adaptivity, since many BIE
applications use hand-tuned non-adaptive quadratures.
We target simulations involving simple bodies, but a large number of them,
hence we test only global quadratures,
leaving panel quadratures for the future.
We note that the MFS can also handle 2D corner domains
using a moderate number of clustered sources
\cite{hochmancorner,larrythesis,lightning}.



We structure the rest of the paper as follows.
Section~\ref{s:th} is a general description of QFS for evaluation
in exterior domains, and proves
(along with Appendix~\ref{a:mfs}) robustness criteria
for Laplace, Helmholtz, and Stokes PDEs in 2D and 3D.
The respective BVPs and fundamental solutions are also reviewed.
Section~\ref{s:2d} presents implementations in 2D of the two variants:
QFS-B (Section~\ref{s:qfs-b}) uses $\pO$ as the check curve,
while QFS-D (Section~\ref{s:qfs-d}) uses a displaced check curve.
The other subsections supply 2D convergence theory, and one-body
numerical tests.
Section~\ref{s:big2d} gives performance
tests of FMM-accelerated QFS-D for large-scale 2D Helmholtz and Stokes BVPs
(via geometry generation in Appendix~\ref{a:geometry}).
Section~\ref{s:3d} presents a preliminary 3D Laplace test involving
ellipsoids.
We draw conclusions in Section~\ref{s:conc}.

We host a Python implementation of QFS at
\url{https://github.com/dbstein/qfs}\\
MATLAB codes for some of 2D and 3D tests are also to be found at\\
\url{https://github.com/ahbarnett/QFS}

\subsection{Prior work on singular and near-singular BIE quadratures
for smooth boundaries}
\label{s:prior}

Here we give a brief and incomplete review of the
large literature on high-order Nystr\"om quadratures in 2D and 3D.
For background we suggest \cite{LIE,coltonkress,atkinson,CMS,gunnarbook,hao}.
(We do not address Galerkin discretizations, which have similar challenges.)

We first highlight some methods for task 1: filling $A$.
The density $\tau$ on $\pO$ is represented by an {\em interpolant}
from $\btau$, its samples at nodes.
Nystr\"om's original method \cite[Ch.~12.2]{LIE}
uses the kernel itself as interpolant, but this only applies to smooth kernels
(double-layers for zero-frequency PDEs in 2D).
Other cases need accurate integration of
the product of each interpolatory basis function
(which may be {\em global} or {\em panel-based})
with the weakly-singular kernel.
For 2D Helmholtz, Kress \cite{kress91} proposed a global
product quadrature, which needs analytic insight to split off the
logarithmically-singular part;
an analogous 3D product quadrature uses spherical harmonics \cite{ganesh,sorgentone18}.
Other 2D and 3D approaches include
local weight corrections of the existing grid \cite{kapur,zeta2d,zeta3d},
and interpolating to custom {\em auxiliary quadrature nodes}
\cite{alpert,laiaxi},
where in 3D a local polar transformation can remove the singularity
\cite{bruno01,ying06,bremer3d,gimbutasgrid}.

We turn to task 2: evaluation near the boundary $\pO$.
Here, in 2D and 3D, upsampling of a plain global rule
gets accuracy nearer to $\pO$
\cite{atkinson,ying06}, but cannot approach $\pO$ \cite{ce}.
Per-target {\em local} upsampling can be very efficient in 3D \cite{fmmbie3d}.
In 2D, Cauchy's theorem is a powerful tool,
either globally via barycentric evaluation \cite{helsing_close,lsc2d},
or via panel monomial bases \cite{helsing_close,helsing_helm,wu2019}.
The idea of extrapolation towards $\pO$ from near-surface data
evaluated by an upsampled plain rule underpins
quadrature by expansion \cite{qbx,ce,walaqbx2d,qbkix,klintporous}
and ``hedgehog'' \cite{hedgehog} schemes,
as it does our proposal.
Other approaches include
density interpolation via Green's theorem \cite{perezpw3d},
regularization \cite{beale}, and asymptotics \cite{khatri2d}.
Yet, for the exterior of the sphere, {\em uniform analytic
expansions} of the potential are available \cite{coronasphere,yanplatform};
the wish to extend this to general shapes inspired this work.

Finally, we note the interplay between the two tasks:
on-surface evaluation can aid with task 2,
while
many of the above off-surface methods can be, and are in practice,
applied to task 1.
The latter will also be true for our proposal.



\section{Description of QFS and theoretical background for three PDEs}
\label{s:th}

We present two variants of ``quadrature by fundamental solutions'',
each of which can evaluate layer potentials at targets far from, near to, or on, $\pO$:
\ben
\item
  QFS-B: 
  The boundary $\pO$ itself is used as the check surface, which requires
  the user to supply a Nystr\"om (on-surface self-interaction) matrix $A$. (In the name, ``B'' stands for boundary.)
\item
  QFS-D: A new check surface is used on the opposite side of $\pO$ from the proxy sources, thus the scheme is fully ``desingularized'' (hence ``D''). Only a smooth {\em upsampling} scheme on $\pO$ is needed.
  In addition it provides a method to fill $A$
  (task 1) without singular on-surface quadratures.
  \een
Both schemes have utility in applications;
if the $A$ matrix is already available then QFS-B is more convenient.

Given a density $\tau$ on a boundary $\pO$, and
desired layer potential representation
\eqref{rep} for $u$ in the exterior of $\Omega$,
both variants of QFS use new layer potentials
placed on $\gamma\in\Omega$, an auxiliary
closed curve in $d=2$ or surface in $d=3$,
\be
\tilde u(\x) \;\approx\; [(\ta \S_\gamma + \tb \D_\gamma) \sigma](\x)~,
\qquad \x \in \R^d \backslash \Omega~.
\label{qfsrep}
\ee
Here $\S_\gamma$ and $\D_\gamma$ denote single- and double-layer potentials
on $\gamma$, and the
QFS mixing parameters $(\ta,\tb)$ are generally distinct from
$(\alpha,\beta)$ in \eqref{rep}.
To solve for the QFS source function $\sigma$,
one collocates on a check curve (or surface) $\gamma_c=\pO$
(for QFS-B), or $\gamma_c$ exterior to and enclosing $\overline\Omega$
(for QFS-D; see Fig.~\ref{f:setup}(d)).
The desired Dirichlet data to match, which we call $u_c$, is given simply by evaluating the user-supplied potential,
\be
u_c := [(\alpha \S + \beta \D) \tau]|_{\gamma_c}~.
\label{uc}
\ee
For now we specialize to QFS-B where $\gamma_c=\pO$, so that
care must be taken to use the exterior limit (jump relation), giving
\be
u_c = [\alpha S + \beta(\half I + D)]\tau
\hspace{.5in} \mbox{ (matching data, continuous QFS-B case). }
\label{ucb}
\ee
Equating \eqref{qfsrep} to $u_c$ on $\gamma_c$
then gives the {\em first-kind integral equation} for $\sigma$,
\be
\int_\gamma \left( \ta G(\x,\y) + \tb \frac{\partial G(\x,\y)}{\partial \n_\y}
\right) \sigma(\y) ds_\y
\; = \;
u_c(\x)~, \qquad \x\in\gamma_c
\label{FKIE}
\ee
where for simplicity for now we use notation for the DLP valid only for
scalar PDEs.

We discretize \eqref{FKIE} by
applying quadrature on $\gamma$ with source nodes $\{\y_j\}_{j=1}^P$,
and discrete collocation on $\gamma_c=\pO$ at the user-supplied nodes
$\{\x_i\}_{i=1}^N$, to get the $N\times P$ linear system
\be
\sum_{j=1}^P \biggl[ \ta G(\x_i,\y_j) + \tb \frac{\partial G(\x_i,\y_j)}{\partial \n_{\y_j}} \biggr] \sigma_j
\;=\;
u_c(\x_i)~, \qquad i=1,\dots,N~.
\label{qfssys}
\ee
The right-hand side vector is given by
$\u_c := \{u_c(\x_i)\}_{i=1}^N \approx A\btau$,
where $A$ is the user-supplied exterior-limit Nystr\"om matrix,
and $\btau := \{\tau(\x_i)\}_{i=1}^N$ the user-supplied density vector.
Note that quadrature weights on $\gamma$ could be included;
here for simplicity we left them implicit in $\sigma_j$.
The linear system \eqref{qfssys} needs a direct solution, due to its poor conditioning, to get
$\bsig:=\{\sigma_j\}_{j=1}^P$.
Finally, the discretization of \eqref{qfsrep},
\be
\tilde{u}(\x) := \sum_{j=1}^P \biggl[ \ta G(\x,\y_j) + \tb \frac{\partial G(\x,\y_j)}{\partial \n_{\y_j}} \biggr] \sigma_j~,
\qquad \x \in \R^d \backslash \Omega~,
\label{qfsu}
\ee
defines our approximate QFS evaluation method for $u$ at all target points $\x$.
This completes the simplest mathematical description.

\subsection{Analysis of continuous QFS
  for the exterior Laplace case}
\label{s:thL}

While appealing, the above proposal raises questions:
What source curve/surface $\gamma$
and mixing parameters $(\ta,\tb)$ should be chosen? Is the
choice to match Dirichlet data on $\gamma_c$ robust?
We first give theoretical results
in the continuous case for the Laplace PDE, exterior case,
covering both $d=2$ (which has a curious twist) and $d=3$,
then distill into criteria for more general elliptic PDE.


Recall that the exterior Laplace Dirichlet BVP is,
given $f\in C(\pO)$ and, in $d=2$ also a {\em total charge} $\Sigma\in\R$,
to solve for $u$ obeying
\bea
\Delta u &=& 0 \qquad \mbox{ in } \R^d \backslash \overline{\Omega}
\label{Lpde}
\\
u&=& f \qquad \mbox{ on } \pO
\label{Lbc}
\\
u(\x)&=&
\left\{\begin{array}{ll}
\Sigma \log r + \omega + o(1)~, & d=2~, \\
o(1)~, & d=3~,
\end{array}
\right.
\qquad r:=\|\x\|\to\infty, \mbox{ uniformly in angle}.
\label{Sig}
\eea
This has a unique solution
(\cite[Thm.~6.24]{LIE} when $\Sigma=0$, otherwise see \cite[Sec.~1.4.1]{HW}).
In $d=2$ the {\em constant term} $\omega\in\R$,
which we emphasize is not part of the input data,
may be extracted after solution
as $\omega = \lim_{\|\x\|\to\infty}u(\x)-\Sigma \log \|\x\|$.

The subtlety in $d=2$ is that the desired
Laplace layer potentials \eqref{rep} on $\pO$,
while exterior Laplace solutions,
do {\em not} span the subspace of exterior harmonic
functions obeying \eqref{Sig}: in particular they are restricted to the
subspace with $\omega=0$.
To see this, recall the Laplace fundamental solution
\be
G(\x,\y) =
\left\{\begin{array}{ll}
\frac{1}{2\pi}\log\frac{1}{r},
& d=2, \\
\frac{1}{4\pi r},  & d=3,
\end{array}
\right.
\qquad r:=\|\x-\y\|~.
\label{LG}
\ee
Well known asymptotics \cite[(6.14-15)]{LIE}
as $r:=\|\x\|\to\infty$ mean that
any Laplace SLP with density $\tau$
has the asymptotic $C\log r + \bigO(1/r)$
in $d=2$, where $2\pi C = \int_\pO \tau$ is the total charge.
In $d=3$ the SLP is $\bigO(1/r)$.
The DLP has the bound $\bigO(1/r^{d-1})$ in $d=2,3$.
Thus a mixture \eqref{rep} has asymptotic $\alpha C\log r + o(1)$ in $d=2$,
or $o(1)$ in $d=3$.
A similar asymptotic of course holds for the QFS representation \eqref{qfsrep}.

Does matching Dirichlet data $u_c$ on $\gamma_c=\pO$
proposed in \eqref{uc}--\eqref{FKIE}
lead to a QFS approximation
$\tilde u$ equaling the correct exterior potential $u$?
In $d=3$ the answer must be yes, assuming \eqref{qfsrep} spans
all possible $u_c$, by uniqueness of the exterior Dirichlet BVP.
To handle the $d=2$ case we need
to flip the roles of $\Sigma$ and $\omega$ to consider a modified
BVP where the constant term is given (zero), but not the total charge (logarithmic growth).
The following lemma shows that this is almost always possible.

\begin{lem}[Modified exterior BVP in $d=2$]   
  Let $\Omega\subset\R^2$ be a bounded domain with
  logarithmic capacity $C_\Omega \neq 1$.
  Then
  the ``zero constant term exterior Dirichlet Laplace BVP,''
  where Dirichlet data $f$ on $\pO$
  is specified plus the decay condition $C\log r + o(1)$ as $r\to\infty$
  with $C\in\R$ unknown,
  has a unique solution.
  \label{l:0const}
\end{lem}
\begin{proof}
  Let $v$ solve the standard
  BVP \eqref{Lpde}--\eqref{Sig} with data $f$ and $\Sigma=0$.
  Let $\omega = v_\infty := \lim_{\|\x\|\to\infty} w(\x)$ be its constant term.
  Let $w$ solve the BVP \eqref{Lpde}--\eqref{Sig} with $f\equiv 0$
  and $\Sigma=1$; so $w$ is the Green function for $\Omega$ with a pole
  at infinity, and
  by definition $\log C_\Omega = -w_\infty := -\lim_{\|\x\|\to\infty} w(\x) + \log \|\x\|$ \cite[Sec.~4.2]{Landkofbook}.
  Note that $w_\infty$ is called the {\em Robin constant} for $\Omega$.
  If $C_\Omega\neq 1$, then
  $v - \omega w/(\log C_\Omega)$ solves the modified
  BVP stated in the Lemma, with resulting logarithmic constant
  \be
  C = -\omega/(\log C_\Omega)~.
  \label{Sigfail}
  \ee
\end{proof}
Failure when $C_\Omega=1$ can occur, as illustrated by $\pO$ the unit circle,
for which $f\equiv 0$ gives a 1-dimensional subspace
$c \log r$, $c\in\R$, of solutions to the modified BVP in the lemma.
For $f\equiv 1$, this BVP has no solution.

Armed with the above uniqueness results,
we state our main result for Laplace (proved in Appendix~\ref{a:mfs}).
It shows that: i)
apart from unit logarithmic capacity in $d=2$,
QFS-B as presented above is robust for analytic data and surfaces,
when the surface $\gamma$ is chosen appropriately;
ii) for this a pure SLP $(\ta,\tb) = (1,0)$
is sufficient as the QFS mixture.
The latter has an advantage over the obvious choice
$(\ta,\tb) = (\alpha,\beta)$, both in simplicity and numerical speed.

\begin{thm}[QFS robustness for exterior Laplace]   
  Let $u$ be a Laplace solution in $\Oc$
  with $u=u_c$ on $\pO$,
  and decay conditions
  $u(\x) = C\log r + o(1)$ for some $C$ if $d=2$ (ie, zero constant term),
  or $u(\x) = o(1)$ if $d=3$, for $r:=\|\x\| \to\infty$.
  Let $u$ also continue analytically as a regular Laplace solution throughout the closed
  annulus (or shell) between $\pO$ and a simple smooth interior surface $\gamma\subset\Omega$.
  Then the first kind integral equation
  \be
  \int_\gamma G(\x,\y) \sigma(\y) ds_\y = u_c(\x), \qquad \x\in\pO
  \label{Lmfsie}
  \ee
  has a solution $\sigma\in C^\infty(\gamma)$.
  If $d>2$, or the logarithmic capacity $C_\Omega\neq 1$, the
  solution is unique, and
  \be
  u(\x) = \int_\gamma G(\x,\y) \sigma(\y) ds_\y~, \qquad \x\in\R^d\backslash\Omega~.
  \label{Lmfsrep}
  \ee
  \label{t:qfslap} \end{thm}  

Since layer potentials \eqref{rep} obey the stated decay conditions,
this shows that, at least for densities sufficiently analytic
to allow $u$ to continue as an interior PDE solution up to the source
curve $\gamma$, QFS is robust.
In $d=2$, where complex analysis is available,
it is known (eg \cite[Prop.~3.1]{ce})
that $u$ continues as a regular PDE
solution as least as far into the nonphysical domain as
the density $\tau$ continues analytically from $\pO$.
In $d=3$ results on analytic continuation are uncommon \cite{kangro3d}.

We will show shortly in Remark~\ref{r:logcap}
how numerically to overcome the failure
of Dirichlet matching for the troublesome case $C_\Omega=1$ in $d=2$.

\subsection{Background and robustness results for exterior evaluation for other PDEs}
\label{s:thHS}

From the above Laplace analysis we can distill two criteria
that together guarantee that QFS is a
robust and accurate exterior layer potential evaluator for elliptic PDEs:
\ben
\item[C1)] (Completeness.)
  In the evaluation region $\Oc$, 
  the range of $\tilde u$ generated by densities $\sigma$ in
  the QFS representation \eqref{qfsrep} contains
  the range of $u$ generated by densities $\tau$ in \eqref{rep}.
\item[C2)] (Uniqueness.)
  There exists a linear subspace of exterior PDE solutions in $\Oc$
  that contains the range of QFS representations \eqref{qfsrep},
  and in which imposing the matching data type on $\gamma_c$ leads to
  {\em uniqueness within this subspace}.
  \een
  It is easy to check that C1 plus C2 implies robustness for QFS.

To illustrate, in the above Laplace case,
C1 (the fact that the pure SLP QFS representation spans
the potentials generated by \eqref{rep}) is assured,
at least for sufficiently analytic $\tau$, by Theorem~\ref{t:qfslap}.
Both of these representations lie in the subspace
of exterior harmonic functions with decay as in the hypothesis
of Theorem~\ref{t:qfslap} (ie, zero constant term), which serves as the
subspace in C2.
Apart from when $C_\Omega=1$ in $d=2$, C2 holds, since
{\em within that subspace} Dirichlet data leads to uniqueness
(Lemma~\ref{l:0const}).
The subtlety of the failure for $C_\Omega=1$ is that, while
\eqref{FKIE} is still soluble (shown by construction in the
proof of Theorem~\ref{t:qfslap}),
its lack of uniqueness will lead numerically to
$\tilde u$ values different from $u$ outside $\pO$.


\begin{rmk}
  These issues appear specific to exterior BVPs. Hence we need not
  (and do not) discuss the simpler interior case much in this work. We
  routinely use QFS for interior problems without issue, for example
  the enclosing boundary in Section~\ref{section:large_scale:stokes}.
\end{rmk}

We now apply these criteria to show robustness for QFS in the
examples of Helmholtz and Stokes layer potential evaluation.
We will first need standard background material for these PDEs.
The full theorems are deferred to the Appendix.

{\bf Helmholtz.}
The exterior Dirichlet BVP is,
given any wavenumber $k>0$ and complex function $f\in C(\pO)$, to solve
\bea
(\Delta+k^2) u &=& 0 \qquad \mbox{ in } \Oc
\label{Hpde}
\\
u&=& f \qquad \mbox{ on } \pO
\label{Hbc}
\\
\partial u/\partial r - ik u&=& o(r^{-(d-1)/2})~, \qquad
r:=\|\x\| \to\infty~,
\label{SRC}
\eea
where the last is the Sommerfeld radiation condition.
This has a unique solution
(see \cite[p.~67]{coltonkress} for $d=2$ and \cite[Thm.~3.7]{coltonkress} for $d=3$).
The fundamental solution at wavenumber $k>0$ is \cite[Sec.~2.2, 3.4]{coltonkress}
\be
G(\x,\y) = 
\left\{\begin{array}{ll}
\frac{i}{4} H^{(1)}_0(kr),
& d=2, \\
\frac{e^{ikr}}{4\pi r},  & d=3,
\end{array}
\right.
\qquad r:=\|\x-\y\|~,
\label{HG}
\ee
where $H^{(1)}_0$ is the outgoing Hankel function of order zero.
The resulting SLP and DLP also generate Helmholtz solutions obeying
\eqref{SRC} \cite[Sec.~3.1]{coltonkress},
so that C2 holds for this subspace.
Theorem~\ref{t:qfshelm} then shows that C1 is satisfied when using the
``combined field'' mixture $(\ta,\tb) = (-i\eta,1)$, for $\eta$ any nonzero
real number. Following standard practice we choose $\eta=k$ from now on
\cite{kress85}.
The proof illustrates that a pure SLP or DLP would lead to nonrobustness
for $k^2$ a Neumann or Dirichlet (respectively) eigenvalue of
the Laplacian in the interior of $\gamma$.

{\bf Stokes.}
We refer the reader to Ladyzhenskaya \cite{Ladyzhenskaya} and
Hsiao-Wendland \cite[Sec.~2.3]{HW} for background.
The exterior Dirichlet BVP is,
given constant fluid viscosity $\mu>0$,
velocity data $\f\in C(\pO)^d$, and in $d=2$ a growth condition $\bSig\in\R^2$,
to solve for a velocity vector field $\u$ and pressure scalar field $p$ obeying
\bea
-\mu\Delta \u + \nabla p &=& 0 \qquad \mbox{ in } \Oc
\label{Spde1}
\\
\nabla \cdot \u &=& 0 \qquad \mbox{ in } \Oc
\label{Spde2}
\\
\u&=& \f \qquad \mbox{ on } \pO
\label{Sbc}
\\
\u(\x)&=&
\left\{\begin{array}{ll}
\bSig\log r + \bom + o(1)~, & d=2~, \\
o(1)~, & d=3~,
\end{array}
\right.
\qquad r:=\|\x\|\to\infty~.
\label{bSig}
\eea
This has a unique solution
for $\u$, and $p$ is unique up to an additive constant
\cite[p.~60]{Ladyzhenskaya} \cite[Sec.~2.3.2]{HW}.
In $d=2$ the constant term may be extracted from the solution
via $\bom = \lim_{r\to\infty} \u(\x) - \bSig \log \|\x\|$,
thus when $\f\equiv\mbf{0}$ the BVP defines a $2\times 2$ matrix
mapping $\bSig$ to $\bom$.
The (tensor-valued) fundamental solution for velocity is
\be
G(\x,\y) = 
\left\{\begin{array}{ll}
\frac{1}{4\pi\mu}\left(I\log\frac{1}{r} + \frac{\r\r^T}{r^2} \right),
& d=2, \\
\frac{1}{8\pi\mu}\left(I\frac{1}{r} + \frac{\r\r^T}{r^3} \right),
& d=3,
\end{array}
\right.
\qquad \r := \x-\y~, \quad r:=\|\r\|~.
\label{SG}
\ee
In contrast to the above scalar PDEs, the DLP kernel
is not
$\partial G(\x,\y)/\partial \n_\y$. The Stokes DLP kernel is
\be
D(\x,\y) = 
\left\{\begin{array}{ll}
\frac{1}{\pi}\frac{(\r\cdot\n_\y)\r\r^T}{r^4},
& d=2, \\
\frac{3}{4\pi}\frac{(\r\cdot\n_\y)\r\r^T}{r^5},
& d=3.
\end{array}
\right.
\label{SD}
\ee
Also, in $d=2$ the corresponding pressure kernels
\eqref{SP} will later be needed.

The goal is to evaluate velocities due to arbitrary densities $\btau$
in
$\u = (\alpha {\cal S} + \beta {\cal D})\btau$, the vector version of \eqref{rep}.
Theorem~\ref{t:qfssto} shows that
the ``completed'' mixture $(\ta,\tb)=(1,1)$ is robust,
for all sufficiently analytic $\btau$,
in $d=3$, or when the above $2\times 2$ matrix is nonsingular.
The latter condition is
analogous to the Laplace capacity condition; see Remark~\ref{r:2x2}.
(We also see numerically, and can prove, that a pure SLP $(\ta,\tb)=(1,0)$
is robust for any $\btau$ which creates zero {\em net fluid flux}
$\beta \int_\pO \btau \cdot \n = 0$,
as occurs in rigid-body flows.)
The SLP and DLP generate Stokes solutions $(\u,p)$
obeying \eqref{bSig}, with $\bom=\mbf{0}$ in $d=2$,
so that C2 holds for this zero-constant-term subspace.
Finally, the theorem then shows that C1 is satisfied.
We note that in $d=3$ a related MFS-based Stokes BIE method has been analysed
\cite{gonzalez09}.

\section{The method for smooth curves in two dimensions}
\label{s:2d}

Here we first describe QFS-B for the exterior
of a single boundary curve in 2D. We next give some theoretical
justifications for the source location algorithm.
We then show numerical tests of QFS-B,
and finally describe and test QFS-D.

\subsection{Basic 2D scheme using collocation on the boundary (QFS-B)}
\label{s:qfs-b}

The user of QFS defines the boundary $\pO$
by supplying a set of nodes $\x_j\in\pO$ and weights $w_j$, $j=1,\dots,N$, which are a good quadrature rule for boundary integrals, meaning that
\be
\int_\pO f(\x) ds_\x  \; \approx \; \sum_{j=1}^N f(\x_j) w_j
\label{rule}
\ee
holds for all smooth functions $f$ on $\pO$.  Specifically we
assume that the error (relative difference between left and right
sides) is no larger than the user-requested tolerance $\ep$ for all
``relevant'' functions $f$, such as BIE integrands with distant targets.
The user also
supplies their vector of density values $\tau_j := \tau(\x_j)$ at these
nodes.  The goal is then to evaluate a potential of the form \eqref{rep}
everywhere in the exterior, also with error $\bigO(\ep)$.

We now set up $P$ sources at locations $\y_j$, $j=1,\dots,P$.
For efficiency reasons we prefer that $P=N$,
although it will sometimes need to be slightly larger.
We assume that a smooth $2\pi$-periodic counterclockwise
parameterization of $\x:\R\to\R^2$ of $\pO$ is available, meaning that
$\x([0,2\pi))=\pO$, and $\x(2\pi)=\x(0)$.
Such a parameterization (and its derivatives) can in practice be extracted by spectral interpolation from user-supplied nodes.
Our recipe for source locations is then equispaced in parameter on an
interior curve $\gamma$ controlled by a separation parameter $\delta>0$,
\be
\gamma_\delta := \bigl\{ \x(t) - \delta \, \|\x'(t)\| \,\n(t)
+ \delta^2 \x''(t) : \; 0\le t < 2\pi  \bigr\}
~,
\label{curve}
\ee
where $\n(t):= R_{-\pi/2} \x'(t)/\|\x'(t)\|$ is the
parametrized outward unit normal, $R_\theta$ denoting counterclockwise rotation
by $\theta$.
We now propose to set $\delta$ and $P$,
and choose source locations $\y_j \in \gamma_\delta$
equispaced in parameter,
via Algorithm~\ref{g:src}.
To first order, 
this separates sources from $\pO$ by
a constant multiple of the local node spacing $h$ on $\pO$;
for an example see Fig.~\ref{f:setup}(c).

\begin{algorithm}[t] 
  \algrenewcommand\algorithmiccomment[2][\footnotesize]{{#1\hfill\(\triangleright\) #2}}  
  \caption{Choosing source locations in 2D}
  \algorithmicrequire{
    $C^2$-smooth parameterization $\x(t)$ of $\pO$, user number of nodes $N$, user tolerance $\ep$,
    source upsampling parameter $\upsilon\ge 1$
    (by default 1).}
  \begin{algorithmic}[1]
    \State assign $P \leftarrow N$
    \State 
    assign a separation $\delta$ appropriate for the user tolerance $\ep$, via
    \be
    \delta \; =\; \frac{1}{P}\log \frac{1}{\ep}~.
    \label{dalias}
    \ee
    
  \If{the curve $\gamma_\delta$ defined by \eqref{curve}
    self-intersects or falls outside of $\Omega$}
    \State estimate $\delta_0$ as the supremum of $\delta$ values such that $\gamma_\delta$ does not self-intersect nor fall outside of $\Omega$
    \State reassign $\delta\leftarrow \delta_0$
     \Comment{this brings source curve closer to $\pO$}
    \State reassign $P$ via \eqref{dalias} \Comment{this increases $P$}
  \EndIf
  
  \State
  reassign $P \leftarrow \lceil \upsilon P \rceil$
  \Comment{possibly upsample, round up}
  
  \State return $P$ source locations $\y_j$ on $\gamma_\delta$ via
  \be
  \y_j \;=\; \x(t_j) - \delta \, \|\x'(t_j)\| \,\n(t_j)
  + \delta^2 \x''(t_j)
  ~,  \qquad t_j=2\pi j/P~, \qquad j=1,\dots, P~.
  \label{src}
  \ee
  
  \end{algorithmic}
  \label{g:src}
\end{algorithm}   

As described early in Section~\ref{s:th},
one now fills the dense ``boundary from source'' matrix $B$ with entries as
in \eqref{qfssys},
\be
B_{ij} = \ta G(\x_i,\y_j) + \tb \frac{\partial G(\x_i,\y_j)}{\partial \n_{\y_j}}~,
\qquad i=1,\dots,N,\; j=1,\dots,P~,
\label{B}
\ee
where we state only the scalar case (in the vector case each
entry is a $2\times 2$ matrix).
The action of $B$ is sketched in Fig.~\ref{f:setup}(c).
Recall that, given the user-supplied density vector $\btau$,
Dirichlet matching data is evaluated via
\be
\u_c \;=\; A \btau~,
\label{ucA}
\ee
where $A$ 
is a user-supplied exterior limit Nystr\"om matrix
(as in \eqref{bie}--\eqref{linsys}).

Mathematically, one then solves for the vector $\bsig\in\C^P$ in the linear system
\be
B\bsig \;=\; \u_c~,
\label{Bsuc}
\ee
which abbreviates \eqref{qfssys}, and is interpreted as matching $\tilde u$ in \eqref{qfsu} to $u_c$ on $\pO$,
then applies \eqref{qfsu} as the QFS approximation to $u$ for all exterior
target points.

Since it
is inefficient to do the dense solve of
\eqref{ucA}--\eqref{Bsuc} anew for each density,
we propose the following variant, firstly in exact arithmetic.
One precomputes the $P\times N$ matrix solution $X$ to the matrix equation
\be
B X \; = \; A~.
\label{BXA}
\ee
The action of $X$ is sketched in Fig.~\ref{f:setup}(c).
Then, for each new $\btau$ vector, one takes the product
\be
\bsig = X \btau
\label{sXt}
\ee
to give the desired $\bsig$ in only $\bigO(N^2)$ time per vector.
It is easy to check that \eqref{BXA}--\eqref{sXt} solves
\eqref{ucA}--\eqref{Bsuc} in exact arithmetic.

However \eqref{BXA}--\eqref{sXt} is unstable in finite-precision
arithmetic because the entries of $X$ are large, due to the
ill-conditioning of $B$
(which is exponentially bad, as we will quantify in Proposition~\ref{p:circeig}).
Catastrophic cancellation in applying \eqref{sXt} typically loses
several digits of accuracy; this cannot be avoided if $X$ is formed.
Thus, instead, following \cite[Rmk.~5]{junlai} \cite{pvfmm} we propose
storing $X$ as two factors $X = Y Z$.
Taking the SVD of $B$ (which, since $P\ge N$, is either square or ``tall''),
\be
U\Sigma V^* = B~,
\label{svdB}
\ee
where $\Sigma = $ diag $\{s_j\}_{j=1}^{N}$,
the singular values being denoted by $s_j$,
one then fills
\be
Y = V \Sigma^{-1}, \hspace{.5in} Z = U^* A~,
\label{YZ}
\ee
where $\Sigma^{-1} := $ diag $\{s_j^{-1}\}_{j=1}^{N}$.
New user-supplied density vectors can then be converted to QFS source vectors in $\bigO(N^2)$ time
via
\be
\bsig = Y (Z \btau) ~,
\label{applyYZ}
\ee
where the {\em order of multiplication} implied by parenthesis is crucial
for numerical stability.
This concludes the basic QFS-B description in 2D;
we will now motivate some aspects via more analytic results.

\begin{rmk}
  In practice, in 2D, although $B$ is ill-conditioned, it is not
  sufficiently so that a regularized inverse is needed in \eqref{YZ}.
  This is because sources chosen using \eqref{dalias}, combined
  with the upcoming Proposition~\ref{p:circeig}, predicts a minimum
  eigenvalue, hence singular value, of $\bigO(\sqrt{\ep})$, safely
  above $\ep_\tbox{\rm mach}$.
  \label{r:cond}
\end{rmk}

\subsection{Discrete theory for source point choice and convergence rate in 2D}
\label{s:2dth}

So far our analytic results have been at the continuous (integral operator)
level. We now introduce analytical background for the {\em discrete}
problem, to justify Algorithm~\ref{g:src}
and to understand the convergence {\em rate} of QFS.

For this analysis, and later numerical tests,
we specialize to quadrature of an analytic curve $\pO$ deriving from
the periodic trapezoid rule (PTR) \cite{PTRtref}.
Recall that
for general $2\pi$-periodic functions $g(t)$ the latter is
\be
\int_0^{2\pi} f(t) dt \; \approx \; \frac{2\pi}{N}\sum_{j=1}^N f(2\pi j/N)
\qquad \mbox{(PTR)~,}
\label{ptr}
\ee
and this rule is high-order accurate for $g$ smooth.
Moreover we have exponential convergence for $g$ analytic.
\begin{thm}[Davis \cite{davis59}]
  Let $f$ be $2\pi$-periodic and analytic,
  and continue analytically
  to a function bounded uniformly
  in the closed strip $|\im t|\le d$.
  Then the error in the PTR quadrature (difference between left and right sides of \eqref{ptr}) is  $\bigO(e^{-dN})$ as $N\to\infty$.
  \label{t:ptr}
\end{thm}
Given a $2\pi$-periodic smooth parameterization $\x(t)$ of $\pO$,
a boundary quadrature rule \eqref{rule} follows by changing variable
from arclength to $t$ then applying the PTR,
$$
\int_\pO f(\x) ds_\x = \int_0^{2\pi} f(\x(t)) \|\x'(t)\| \, dt
\; \approx \; \frac{2\pi}{N} \sum_{j=1}^N f(\x(2\pi j/N)) \, \|\x'(2\pi j/N)\|
$$
implying that the nodes and weights in \eqref{rule} are
\be
\x_j = \x(2\pi j/ N)~, \qquad w_j = (2\pi/N) \|\x'(2\pi j/N)\|~.
\label{ptrpO}
\ee
Here $\x'(t) := d\x/dt$ is the parametric ``velocity'',
$\|\x'(t)\|$ its ``speed'', so the local node spacing ($h$) is $w_j$.

The source curve \eqref{curve} proposed above has its origin as
a 2nd-order Taylor approximation to the following
curve $\gim_\delta$ generated by
``imaginary parameter translation''.
Assume that the $\pO$ parameterization $\x(t)=[x_1(t),x_2(t)]$ is a pair of real analytic functions, and $\|\x(t)\|>0$, $t \in [0,2\pi)$.
  In that case it is possible to
  analytically extend the parameterization into an annulus around $\pO$
by identifying $\R^2$ with $\C$, as follows.
Let $Z(t) := x_1(t)+ix_2(t)$ for $t$ real, then $Z$ may be uniquely
analytically continued throughout some strip $I_0:=\{t\in\C: |\im t|<\delta_0\}$
about the real axis, defining an annular conformal map $Z:I_0\to\C$.
Then define
\be
\x(t,\delta):= [\re Z(t+i\delta), \im Z(t+i\delta)],  \qquad t,\delta \in \R~,
\label{imsh}
\ee
and note that $\x(t,0)=\x(t)$ for all $t$, which recovers $\pO$.
We call \eqref{imsh} an
{\em imaginary shift by $\delta$ in the complexified parameterization}.
Fixing $\delta$, it generates a curve
\be
\gim_\delta := \bigl\{ \x(t,\delta) : \; 0\le t < 2\pi  \bigr\}
\label{curveim}
\ee
which for $\delta>0$ lies inside $\pO$, and for $\delta<0$ lies outside $\pO$.
In practice we find that the 2nd-order approximation is good, ie,
$\gamma_\delta$ is very close to $\gim_\delta$, and
$\|\y_j - \x(2\pi j/P,\delta)\| \ll \delta$,
making the following theory relevant.

The imaginary translation idea has been studied in the setting of the MFS,
where it brings both practical and theoretical advantages
\cite{Ka96,Kark01,mfs,kangro2d,acper}.
It has led to a class of convergence results for various BVPs
which can be summarized, somewhat loosely, by the following.
\begin{thm}[MFS convergence rate \cite{Ka89,Ka96,mfs,kangro2d}]  
  Let the matching curve $\gamma_c=\pO$ be analytic, as above, with $Z$ analytic in a strip with half-width $\delta_0$.
  Let the data $u_c$ on $\pO$ be analytic, with parametric form $u_c(t)$
  continuing to an analytic function throughout some strip
  $|\im t|\le\delta_\ast$.
  Let $\delta>0$ be sufficiently small, generating
  source locations $\y_j = \x(2\pi j/P,\delta)$, $j=1,\dots,P$,
  and source mixture $(\ta,\tb)=(1,0)$.
  Then there is an algebraic order $a\ge 0$ such that
  the (MFS or QFS) solution method
  \eqref{qfssys}--\eqref{qfsu} in exact arithmetic has asymptotic error
  \be
  \|\tilde u - u_c\|_\pO = \left\{ \begin{array}{lll}
    \bigO (P^a e^{-\delta P}), & \delta<\delta_\ast/2 & \qquad \mbox{\rm (discrete source aliasing; smooth data)} \\
    \bigO (P^a e^{-\delta_\ast P/2}), & \delta\ge\delta_\ast/2 & \qquad \mbox{\rm (Nyquist frequency limit; rough data)}
  \end{array}\right.
  \label{mfsrates}
  \ee
  as $P\to\infty$, where $\|\cdot\|_\pO$ is some boundary norm.
  \label{t:mfs}
\end{thm}            
Such results state that the MFS has near-exponential convergence
in the number of source points, with two
regimes of rate interpreted as follows:
if the sources are {\em close} to $\pO$ then their discrete aliasing error
prevents each Fourier mode of $u_c(t)$ from being accurately represented;
whereas, if the sources are {\em far} from $\pO$ the roughness of $u_c$ dominates the error because they cannot represent Fourier modes on $\pO$
with index magnitude exceeding the Nyquist frequency $P/2$.
\begin{rmk}
  The above theorem was first proven for the interior Laplace BVP in the disk
  by Katsurada \cite{Ka89}, and generalized to interior Helmholtz in the disk
  with $L^2$-minimization on $\pO$ in \cite[Thm.~3]{mfs}.
  These results used Fourier series in $t$, and in almost all cases
  there is no algebraic prefactor ($a=0$).
  For general analytic boundaries with annular
  conformal maps,
  only the second case in \eqref{mfsrates} is known for interior Laplace \cite[Thm.~3.2]{Ka94} \cite[App.~A]{Ka96},
  while for exterior Helmholtz problems, Kangro \cite[Thm~4.1]{kangro2d}
  has proven the first case in \eqref{mfsrates}.
  These latter results use
  integral operator approximations in exponentially-weighted Sobolev spaces,
  and involve quite technical other conditions that we do not state.
  We know of no such results for Stokes BVPs, even on the disk.
  So the analysis of the MFS, even for 2D analytic domains,
  is still incomplete.
\end{rmk}

Our criterion \eqref{dalias} for $\delta$ and $P$ is
now understood as follows: one
sets the boundary error norm to the user-requested tolerance $\ep$,
drops the algebraic prefactor
in the upper (aliasing or smooth data) case of \eqref{mfsrates},
then uses this as an equality.
If the resulting curve $\gamma_\delta$ self-intersects or falls outside
of $\Omega$, then $\delta>\delta_0$ was too large and Theorem~\ref{t:mfs}
does not apply.
This provides an analytic foundation for Algorithm~\ref{g:src}.

The above MFS convergence theorems rely on the smoothness of the layer
operators between separated curves. For the special case of concentric
circles this smoothness
is simple to show via polar separation of variables, giving the following
general result for scalar 2nd-order elliptic PDE.
\begin{pro}[ill-conditioning of first-kind integral equation]
  Consider $\pO$ the unit circle with complexified parameterization $Z(t)=e^{it}$, and let $\delta \neq 0$ so that $\gim_\delta$ is the concentric circle of radius $e^{-\delta} \neq 1$.
  Then the single- and double-layer integral operators from source $\gim_\delta$ to target $\pO$ have as eigenvectors the Fourier modes $e^{int}$, $n\in\mathbb{Z}$, with corresponding eigenvalues decaying
  as $\bigO(e^{-|\delta n|})$
  asymptotically as $|n|\to\infty$,
  ignoring algebraic prefactors.
  \label{p:circeig}
\end{pro}
This is well known for Laplace and for fixed-$k$ Helmholtz 
(eg, see \cite{mfs} for the single-layer case).
For Stokes, numerically one sees similar upper (but not always lower) bounds;
the analysis is incomplete (but see \cite{hsiao85}).
For general curves obtained by imaginary translations
in the general annular conformal map case,
similar results feature
in the proofs of Theorem~\ref{t:mfs} in Laplace \cite{Ka96}
and Helmholtz \cite{kangro2d} settings.

\begin{rmk}[upper bound on separation]
  Given finite-precision arithmetic,
  Proposition~\ref{p:circeig} implies that the source curve $\gamma_\delta$
  should
  not be so far from the check curve (in this case $\pO$) that the influence
  on (eigenvalue of) the highest Fourier mode $|n|=N/2$ on the boundary
  drops below machine precision, $\ep_\tbox{mach}$.
  This gives the upper bound on separation
  $
  \delta \le \frac{2}{N}\log\emach^{-1}
  $
  We note that, since $P\approx N$ and $\ep>\emach$,
  the choice \eqref{dalias}
  is safely no more than about half this upper bound.
\end{rmk}


\bfi 
\ig{width=\textwidth}{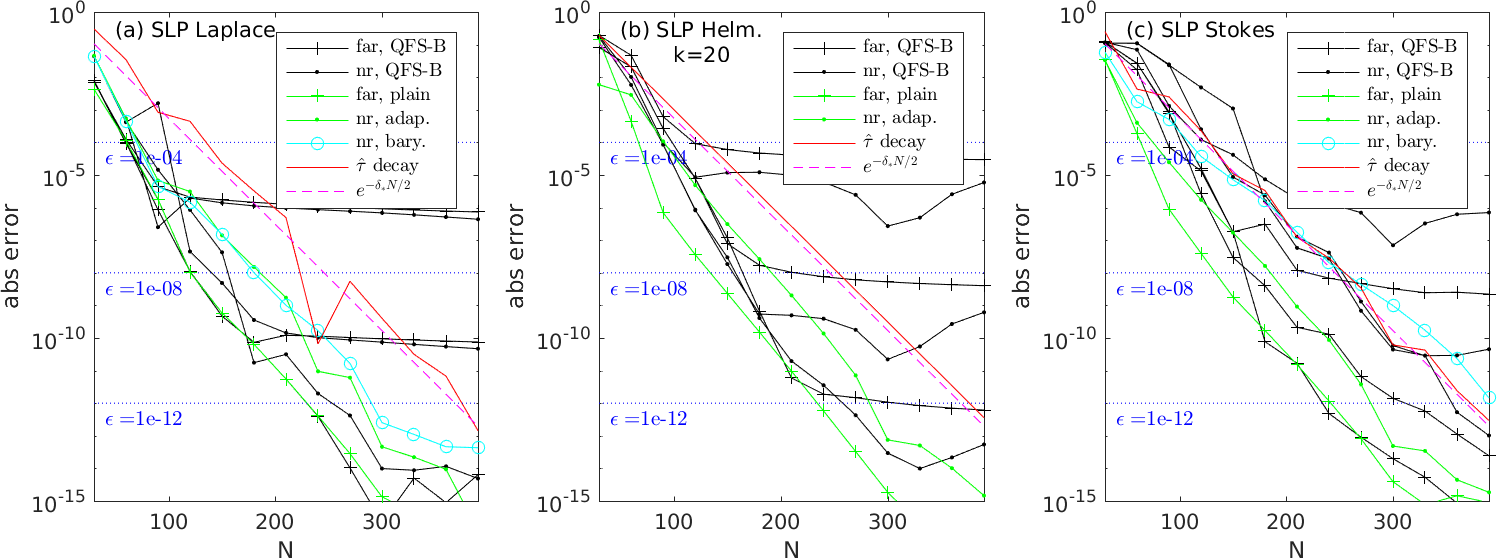}
\ca{Error convergence for 2D exterior evaluation
  of a given
  single-layer potential $\tau$ vs number of nodes $N$,
  for 3 PDEs (a,b,c); see Sec.~\ref{s:Btest}.
  The shape and (far, near) targets are as in Fig.~\ref{f:setup}.
  QFS-B from Sec.~\ref{s:qfs-b} is shown (black) for
  the three tolerances shown (blue).
  Green curves show the plain Nystr\"om rule for the far target, and
  adaptive Gaussian integration of the spectral interpolant for the near
  target.
  The barycentric method of \cite{lsc2d} is also shown for the near target
  (cyan circles).
  The Fourier decay \eqref{foud} of the density (red) and
  a predicted rate based on $\delta_\ast$ the parameter-plane singularity distance
  from the real axis (dashed) are compared.
}{f:BSLP}
\efi 

\bfi 
\ig{width=\textwidth}{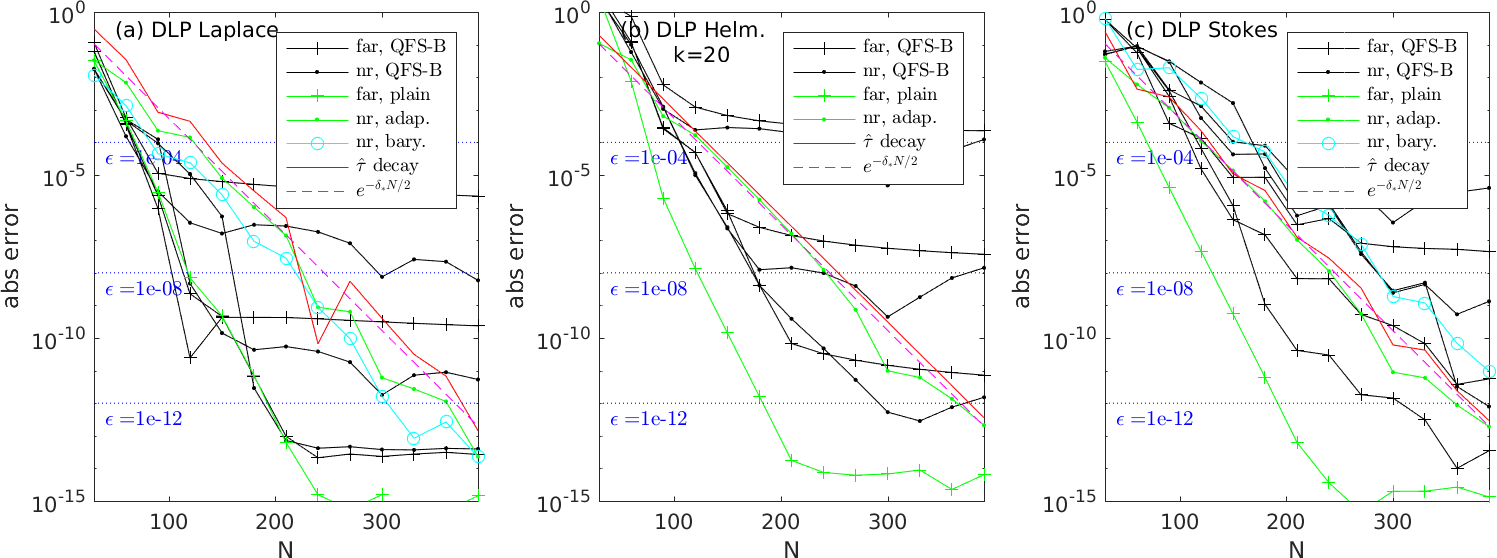}
\ca{Same as Fig.~\ref{f:BSLP} but testing evaluation of the
  double-layer potential.
}{f:BDLP}
\efi 


\subsection{Tests of QFS-B for Laplace, Helmholtz, and Stokes exterior evaluation}
\label{s:Btest}

Here we test, for three PDEs, and three tolerances $\ep$,
the error performance of QFS-B for evaluation of given single- or double-layer potentials (ie, $(\alpha,\beta) = (1,0)$ for SLP, or $(0,1)$
for DLP, in \eqref{rep}).
We use the same analytic starfish
domain as in Fig.~\ref{f:setup},
for which the maximum non-intersecting $\gamma$ separation distance
is found numerically to be $\delta_0 \approx 0.168$.
In each case we test far and near targets,
measuring errors relative to fully-converged plain or adaptive
integration, respectively. 
We work in MATLAB R2017 on a i7 CPU; calculations take only a few seconds.
We are forced to test a quite large near-boundary distance of $10^{-4}$
to retain all digits in adaptive integration (via MATLAB's {\tt integral} command), due to catastrophic cancellation in the integrand.
QFS, by contrast, can handle distances down to zero reliably.


{\bf Laplace.}
In QFS we use the pure SLP $(\ta,\tb)=(1,0)$, which is
robust by Theorem~\ref{t:qfslap}.
Recall the fundamental solution \eqref{LG}.
For each $N$ we use Algorithm~\ref{g:src}, with $\upsilon=1$
so that $P=N$ unless self-intersection triggered an increase in $P$.
At $\ep=10^{-12}$, in this domain, this was not triggered once $N\ge 150$.
Fig.~\ref{f:BSLP}(a) compares the convergence of QFS-B to other standard
methods, in a generic test case for evaluating a pure SLP,
and Fig.~\ref{f:BDLP}(a) shows the same for evaluating a pure DLP.
In both cases QFS-B shows very similar convergence
to the gold-standard plain Nystr\"om rule for the far target,
and adaptive integration of the trigonometric (spectral) interpolant
for a near target, down to the chosen tolerance $\ep$.
When errors hit $\ep$ (actually 1-2 digits below), they flatten out, as predicted by our separation choice \eqref{dalias}.
This, along with stability for larger $N$, indicates success.
A barycentric Cauchy method (see \cite{lsc2d} for SLP,
\cite{ioak,helsing_close} for DLP) is also compared for the near target,
and exhibits the same rate.

We now discuss some details about the density and convergence rates.
For the test we chose a
real-analytic density, in terms of the boundary parameter $t$,
\be
  \tau(t) \;=\; [0.5+\sin(3t+1)] \re\, \cot\frac{t-t_\ast}{2}~,\qquad   t\in[0,2\pi),
  \label{Ltau}
\ee
whose important feature is its
complex singularity location $t_\ast = 0.5 + i\delta_\ast$,
where $\delta_\ast = 0.15$ controls its distance and hence smoothness along the
real axis.
Note that $\delta_\ast$ is similar to the boundary's
$\delta_0$; this models low-frequency scattering problems, where data
smoothness is controlled by the geometry.
We have verified that the QFS performance is similar to gold-standard
methods down to $\ep$ also for other $\delta_\ast$ choices.

An estimate of the ability of $N$ boundary samples to capture any density
$\tau$ is the relative decay of its Fourier series $\tau(t) = \sum_{n\in\Z} \hat\tau_n e^{int}$ at the Nyquist frequency $n=N/2$,
which we measure by
\be
r_N[\tau] \;:=\; \frac{|\hat\tau_{N/2}|}{|\hat\tau_0|}~.
\label{foud}
\ee
This metric is included (in red) on the plots,
and compared against its asymptotic prediction $e^{-\delta_\ast N/2}$
(magenta) set by the
known singularity; the match is excellent.

\begin{rmk}[Laplace error decay rates]
  For the SLP in Fig.~\ref{f:BSLP}(a)
  we see that the Nyquist decay rate (red and pink) explains well
  the convergence of all the near-target methods. The far target
  rate is slightly faster.
  For the DLP in Fig.~\ref{f:BDLP}(a), in contrast,
  the plain Nystr\"om rule for the far target has about {\em twice}
  the Nyquist rate, and
  QFS-B achieves this faster rate for {\em both} near and far targets.
  The doubling of rate for the plain rule is
  believed to be due to the fact that for a
  distant target the kernel is as smooth as
  the geometry, so that the PTR rate $e^{-\delta_\ast N}$ of Theorem~\ref{t:ptr} is relevant, being twice the Nyquist rate.
  However, this does not explain why the rate for the SLP is less than doubled.
  We do not have an explanation for the doubling of the QFS-B DLP near rate.
  \label{r:double}
\end{rmk}

\begin{rmk}[robustness when the matching curve $\pO$ has logarithmic capacity near 1.]
  Lemma~\ref{l:0const}, in particular
  \eqref{Sigfail}, showed that numerical instability will occur when
  $C_\Omega\approx 1$: the Dirichlet matching on $\pO$ becomes
  unable to determine the total charge $\Sigma$, which is crucial
  to accurate evaluation of $u$.
  However, given a desired representation \eqref{rep}, the total charge
  is in fact known: $\Sigma = \alpha \int_\pO \tau$.
  Thus, stability is easily recovered by adding one row to each matrix
  $A$ and $B$ that enforces this condition in \eqref{ucA}--\eqref{Bsuc}.
  Specifically,
  we append to $A$ the row $\{\alpha w_j\}_{j=1}^N$ (where $w_j$ are the 
  weights in \eqref{ptrpO}), and to $B$ the row of all ones.
  We have verified that in the case of $\Omega$ the unit disk, where
  $C_\Omega=1$, this modification turns complete failure into successful
  error convergence similar to that shown above. We need not show the plots.
\label{r:logcap}
\end{rmk}

{\bf Helmholtz}.
Recall that the fundamental solution is \eqref{HG},
and that we use a combined-field QFS representation.
We again set source upsampling $\upsilon=1$.
For a complex-valued density $\tau$ we choose \eqref{Ltau} except with the Re operator removed,
and fix $k=20$ (around 8 wavelengths across the domain).
Fig.~\ref{f:BSLP}(b) compares the convergence of QFS-B
for evaluating the pure SLP to
the gold standard (as with Laplace, plain Nystr\"om for the far target, and adaptive integration of the spectral interpolant for the near target).
We again see success, meaning that the rates are similar and,
although the error at which
QFS-B saturates is 1-2 digits worse than for the Laplace case, it remains
consistent with the requested tolerance $\ep$.
We do not compare to a barycentric method, since we
know of no such published method for Helmholtz.
Fig.~\ref{f:BDLP}(b) shows similar performance for evaluating a pure DLP,
although, as in Remark~\ref{r:double}, there is a factor of two separating
the gold-standard far and near convergence rates; QFS-B falls somewhere
between the two rates.
One also sees saturation about 1 digit worse than $\ep$, indicating that
the user should set $\ep$ slightly below their desired tolerance.

{\bf Stokes}.
Recall that we use a completed QFS representation $S+D$,
and that the kernels are \eqref{SG}--\eqref{SD}.
We set viscosity to a generic near-unit value $\mu=0.7$, and choose
the vector-valued density function
\be
\btau(t) \;=\; \re \vt{e^{4i} \left(0.5+\sin(3t+1)\right) \cot\frac{t-t_\ast}{2}}{e^{5i} \left(0.5+\cos(2t-1)\right) \cot\frac{t-t_\ast}{2}}
~,
\qquad   t\in[0,2\pi),
\label{Stau}
\ee
Apart from its singularity distance $\delta_\ast=0.15$,
\eqref{Stau} is designed to be generic; eg, it has net flux
$\int_\pO \btau \cdot \n \neq 0$.
To achieve numerical and spectral stability
(see the upcoming Fig.~\ref{f:kappa}),
for this PDE we need to set source upsampling to $\upsilon=1.3$
Fig.~\ref{f:BSLP}(c) and Fig.~\ref{f:BDLP}(c) then compare QFS-B against
the same gold-standard methods used for Laplace, including
the barycentric Stokes methods introduced in \cite{lsc2d}.
With this choice, QFS-B again matches well the
gold-standard error convergence for SLP, and is close for DLP,
down to below $\ep$.
Together with stability at all $N$, this indicates success.

\begin{rmk}[Stokes breakdown for certain domains?]
  Unlike for Laplace where unit-capacity domains are easy to construct,
  we have not observed ``in the wild'' the Stokes QFS failure
  potentially allowed by Theorem~\ref{t:qfssto}.
  We also have not found literature about the possibility
of the $2\times 2$ matrix becoming singular.
For complete robustness, we have tested adding two extra rows to enforce
the known
$\bSig$ (analogous to Remark~\ref{r:logcap}): this is successful for
QFS-D, but limits near-target accuracy in QFS-B
to about $10^{-9}$, an issue that we leave for future study.
\label{r:2x2}
\end{rmk}

\subsection{Desingularized scheme using off-surface check points (QFS-D)}
\label{s:qfs-d}

We now show how only upsampled off-surface evaluations can be used to
evaluate layer potentials in the solution domain and on the boundary,
including the filling of the Nystr\"om $A$ matrix (task 2).
This makes QFS truly kernel-independent (apart from the choice
of QFS representation $(\ta,\tb)$ which may vary by PDE).

$P$ source points with separation parameter $\delta$
are chosen via Algorithm~\ref{g:src}, as in QFS-B.
Then recalling \eqref{dalias}, one chooses a check curve separation parameter
$\delta_c>0$ via
\be
\delta_c \;=\; \frac{1}{P}\log \frac{1}{\emach} - \delta
\;=\; \biggl(\frac{\log \emach}{\log \ep} - 1 \biggr) \delta
~,
\label{dc}
\ee
and if the resulting curve $\gamma_{-\delta_c}$ self-intersects
or hits $\Omega$,
$\delta_c$ is reduced to the supremum of values, $\delta_{c,0}$, for which this
no longer holds.
The number of check points is $\Nc = \lceil \upsilon_c N \rceil$,
with $\upsilon_c$ a small PDE-dependent upsampling parameter (by default 1).
The check points $\z_m$ are then
\be
\z_m \;=\; \x(t_m) + \delta_c \, \|\x'(t_m)\| \,\n(t_m)
+ \delta_c^2 \x''(t_m)
~,  \qquad t_m=2\pi m/\Nc~, \qquad m=1,\dots, \Nc~,
\label{chk}
\ee
noting the sign change which approximates an imaginary
translation by $-\delta_c$.
Fig.~\ref{f:setup}(d) shows source and check points when $\ep=10^{-12}$;
note that the check points are around three times closer to $\pO$ than
the source points.

\begin{rmk}[check point distance $\delta_c$]  
  The heuristic observation behind \eqref{dc} is that when collocation
  is performed to match potential values on $\gamma_c$ to $\bigO(\emach)$,
  there is an exponential deterioration of errors as one moves off this
  curve back towards $\pO$, as expected because numerical analytic
  continuation as a PDE solution is involved.
  Its rate is such that $\bigO(1)$ error is reached
  by $\gamma$, the source curve. Thus in order to insure the user-requested
  tolerance $\ep$ on $\pO$, the ratio condition
  \be
  \frac{\delta}{\delta+\delta_c} \;\ge\;
  \frac{\log \ep}{\log \emach}
  \label{ratio}
  \ee
  must hold.
  For example, when $\ep=10^{-12}$ and $\emach\approx 10^{-16}$, the right-hand
  side is $3/4$, leading to $\delta_c \le \delta/3$.
  Treating the condition as an equality leads to \eqref{dc}.
\end{rmk}

The $\Nc\times P$
``check from source'' matrix, which for QFS-D we now denote by $E$,
is then filled with elements
\be
E_{mj} = \ta G(\z_m,\y_j) + \tb \frac{\partial G(\z_m,\y_j)}{\partial \n_{\y_j}}~,
\qquad m=1,\dots,\Nc,\; j=1,\dots,P~,
\label{E}
\ee
where as before we state only the scalar case.
The action of $E$ is sketched in Fig.~\ref{f:setup}(d).
Since the separation parameter between source and check curves
is no more than $N^{-1}\log\emach^{-1}$, barring small upsampling factors,
Remark~\ref{r:cond} also applies to this $E$ matrix.

The final ingredient is to evaluate $u$ accurately on $\gamma_c$ to get $u_c$.
The user-supplied $N$ nodes are rarely adequate for this,
but plain quadrature from an {\em upsampled} set of nodes
can be very accurate. The following result, for the PTR case,
enables our choice of boundary upsampling factor.
\begin{thm}{\cite[Thm.~2.3, 2.9]{ce}}
  Let $\pO$ be analytic, and $Z$ be the analytic continuation of its
  complex parameterization, with $Z$ analytic and bijective in some strip $I_0=\{t\in\C: |\im t|<\delta_0\}$. Let $\tau(t)$, $t\in[0,2\pi)$, be an analytic
    density that continues analytically in $I_0$.
    Recalling \eqref{imsh}, let $\x = \x(t,\delta)$,
    for arbitrary $t\in\R$ and imaginary shift $\delta \neq 0$,
    be a target point.
    Then the error in applying the PTR \eqref{ptrpO} to
    evaluation of the Laplace layer potential \eqref{rep} at this target point
    is
    $\bigO(e^{-|\delta| N})$, as $N\to\infty$.
  \label{t:ce}
\end{thm}
Thus the exponential convergence rate with the number of boundary nodes is the {\em complexified parametric distance of the target from the source curve}.
The mechanism is the same as the upper case in \eqref{mfsrates}.
Its proof uses a modification of Theorem~\ref{t:ptr} to handle
integrands with a single singularity or branch cut in the strip (annulus).
The same rate is conjectured (and numerically verified)
for Helmholtz \cite{ce} and Stokes \cite{junwang}.

Applying Theorem~\ref{t:ce} (ignoring prefactors),
to reach full accuracy $\emach$ at the
check point imaginary translation of $\delta_c$, one needs
$\tilde{N} = \lceil \rho N \rceil$
boundary nodes, where the boundary upsampling factor $\rho$ is
\be
\rho = \max \left[ \frac{1}{\delta_c N} \log \frac{1}{\emach}, 1 \right]~.
\label{rho}
\ee
Here the max prevents downsampling, 
which would be wasteful.
For example, $\ep=10^{-4}$ results in $\rho\approx 1.3$, while
$\ep=10^{-12}$ results in $\rho\approx 4.3$.
The vector of check potentials is then evaluated by the upsampled plain rule,
\be
\u_c = C \btau~, \qquad \mbox{ where } \quad
C := \tilde{C} L_{\tilde{N} \times N}~,
\label{ucC}
\ee
where $\tilde{C}$ is an ``check from upsampled boundary'' matrix
with elements
\be
C_{mj} = \ta G(\z_m,\tilde\x_j) + \tb \frac{\partial G(\z_m,\tilde\x_j)}{\partial \n_{\tilde\x_j}}~,
\qquad m=1,\dots,\Nc,\; j=1,\dots,\tilde{N}~,
\label{C}
\ee
where $\tilde\x_j = \x(2\pi j/\tilde N)$, $j=1,\dots,\tilde N$,
are spectrally upsampled boundary nodes.
The matrix $L_{\tilde{N} \times N}$ in \eqref{ucC} is a standard
spectral upsampling matrix with elements
\be
(L_{\tilde{N} \times N})_{lj} = \frac{1}{N} \phi_N ( 2\pi [l / \tilde N - j/N ])~,
\qquad l=1,\dots,\tilde N,\; j=1,\dots,N~,
\label{L}
\ee
where $\phi_N(s):= 1 + 2 \sum_{k=1}^{N/2-1} \cos ks + \cos (N/2)s$ is a
slight variant of the Dirichlet kernel function,
and we took the case that $N$ and $\tilde N$ are both even.
The upsampled boundary nodes (blue dots) and action of the $\Nc\times N$
evaluation matrix $C$ in \eqref{ucC} are sketched in Fig.~\ref{f:setup}(d).

The dense factorization (precomputation stage) for QFS-D is very similar
to QFS-B. In the simplest version one takes the SVD
\be
U\Sigma V^\ast = E
\label{svdE}
\ee
then fills
\be
Y = V \Sigma^{-1}, \hspace{.5in} Z = U^* C~,
\label{YZC}
\ee
so that new density vectors may be converted
to source vectors as before by \eqref{applyYZ}.
Algorithm~\ref{g:qfs-d} summarizes these two steps.
The resulting sources may then be used to evaluate the potential
everywhere via \eqref{qfsu}.

\begin{algorithm}[t] 
  \algrenewcommand\algorithmiccomment[2][\footnotesize]{{#1\hfill\(\triangleright\) #2}}  
  \caption{QFS-D (desingularized) for single body, tasks 1 and 2}
\begin{algorithmic}[1]
  \Procedure{QFSDprecompute}{$N$ nodes $\x_j$ describing $\pO$,
    tolerance $\ep$,
    mixture $(\alpha,\beta)$,
    source and check upsampling parameters $\upsilon,\upsilon_c \ge 1$}
\Comment{set-up, $\bigO(N^3)$ work}
    \State Use the PDE type to choose robust QFS mixture $(\ta,\tb)$ as in Sections~\ref{s:thL}--\ref{s:thHS}.
    \State Choose $P$ source points $\y_j$ (if $d=2$ use Algorithm~\ref{g:src})
    \State Choose $\Nc$ check points $\z_m$ (if $d=2$ use \eqref{dc}--\eqref{chk})
    \State Choose boundary upsampling factor $\rho$ via \eqref{rho}
    \State Fill matrices: \,``check from source'' $E$ via \eqref{E}\\
     \makebox[.92in]{}``check from boundary'' $C$ via upsampling and matrix product \eqref{ucC}--\eqref{L}\\
     \makebox[.92in]{}``boundary from source'' $B$ via \eqref{B}
    \State Take SVD of $E$ then form $Y$ and $Z$ via \eqref{YZC}
    \Comment{(or LU as in Remark~\ref{r:LU})}
    \State Form Nystr\"om self-interaction matrix $\tilde A$ via \eqref{tA}
    \State \textbf{return} $Y$, $Z$, $\tilde A$
    \Comment{(or \textbf{return} $L$, $U$, $\bar{P}C$, $\tilde A$; Remark~\ref{r:LU})}
    
    \EndProcedure
    
    \Procedure{QFSDapply}{$Y$,$Z$,vector $\btau$ of $N$ density samples}
    \Comment{compute sources, $\bigO(N^2)$}
    \State \textbf{return} $\bsig = Y (Z \btau)$
\Comment{the layer potential $u(\x)$ may now be evaluated via \eqref{qfsu}}
    \EndProcedure

  \end{algorithmic}
  \label{g:qfs-d}
\end{algorithm}   

Finally, an approximation $\tilde A$ to the Nystr\"om self-interaction
matrix (including the exterior limit $\frac{\alpha}{2} I$ term due to the jump relation) is $\tilde A = BX$,
recalling \eqref{BXA}, with $B$ defined (in the scalar case) by \eqref{B},
and $X=YZ$.
While forming $X$ then $BX$ is adequate for a few digits of accuracy,
full accuracy requires reordering as
\be
\tilde A \; = \; (BY) Z~,
\label{tA}
\ee
to avoid catastrophic cancellations as in Sec.~\ref{s:qfs-b}.
After a couple of remarks, we proceed to numerical tests of the method.

\begin{rmk}
  The reader may wonder whether at the continuous level QFS-D has the
  same justification as QFS-B.
  Although, for simplicity, Theorem~\ref{t:qfslap}
  was phrased assuming $\gamma_c=\pO$, ie, for QFS-B, it
  easily generalizes to QFS-D by replacing $\pO$ by
  $\gamma_c$, and noting that in the proof \eqref{Lmfsrep} also
  applies throughout the exterior of $\gamma$, in particular throughout
  the exterior of $\Omega$ and on $\pO$.
  Theorems~\ref{t:qfshelm} and \ref{t:qfssto} similarly generalize.
\end{rmk}

A rigorous justification for the discrete convergence of QFS-D,
in particular the ratio \eqref{ratio} involving
extrapolation in finite-precision arithmetic,
we leave for future work.

\begin{rmk}[LU, and the case of rectangular $E$]
  So far, for simplicity, we described the use of the
  SVD \eqref{svdE}.
  However, 
  partially-pivoted
  LU is faster for precomputation, and is still stable by
  Remark~\ref{r:cond}.
  In the square case $P=\Nc$, one factorizes
  $\bar{P}E=LU$, where $\bar{P}$
  is a permutation matrix, and stores the factors and $C$.
  The apply step becomes $\bsig = U^{-1}(L^{-1}(\bar{P} C \btau))$
  where the parentheses and
  inverses indicate triangular back-substitutions done on the
  fly, taking $\bigO(N^2)$ work.
  The Nystr\"om matrix form \eqref{tA} becomes
  $\tilde A = (B U^{-1})(L^{-1}(\bar{P} C))$, where inverses
  require back-substitutions for full stability.
  If $P>\Nc$, as can occur when the source curve moves
  closer in Algorithm~\ref{g:src},
  we spectrally downsample to $\Nc$ source points, LU-factorize the
  resulting $\Nc \times \Nc$ matrix,
  then upsample $\bsig$ at  the end of the apply step.
  \label{r:LU}
\end{rmk}

\bfi 
\ig{width=\textwidth}{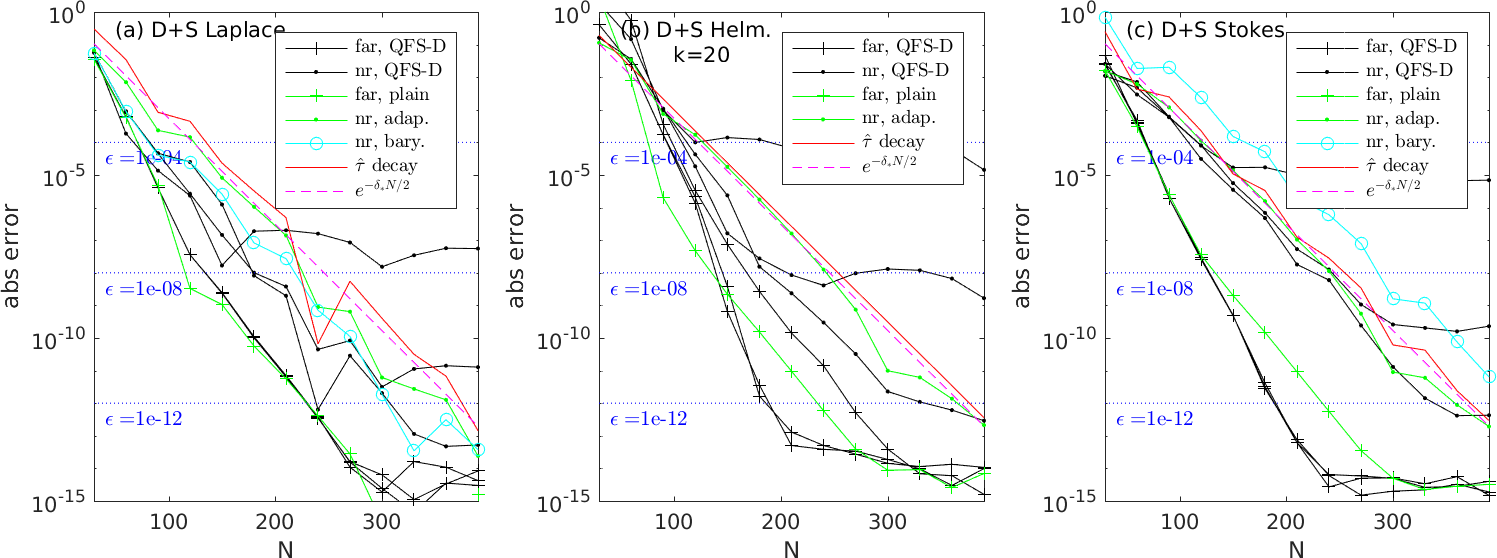}
\ca{Error performance of the desingularized
  method (QFS-D) in 2D, testing evaluation of the
  layer potential mixture ${\cal D} + {\cal S}$.
  All other details are as for Fig.~\ref{f:BSLP}.
}{f:DmLP}
\efi 

\bfi 
\ig{width=\textwidth}{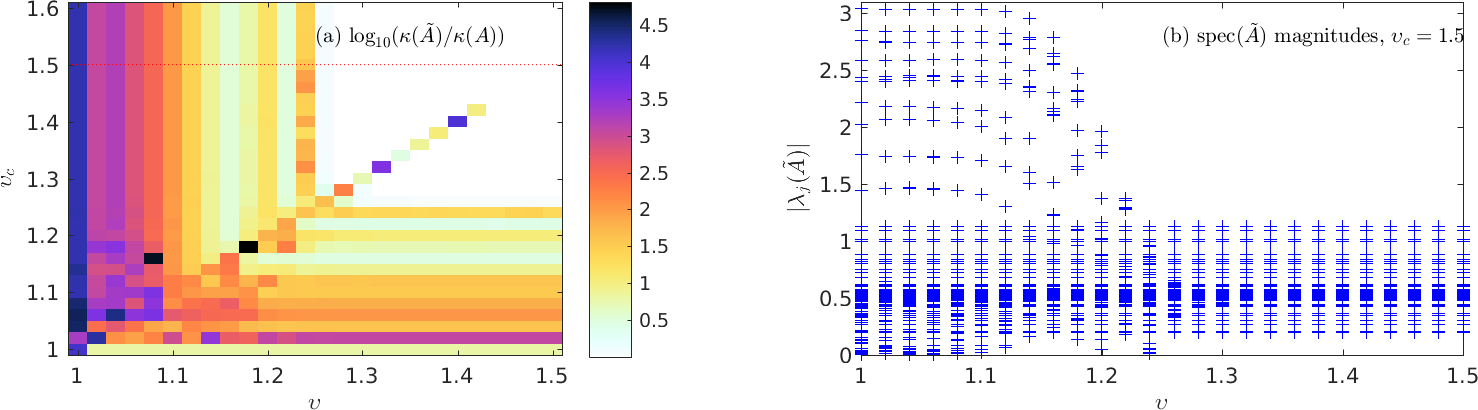}
\ca{Spectral properties of the QFS-D Nystr\"om matrix $\tilde A$
  compared to the gold-standard Kress matrix $A$, discretizing the
  ``completed'' Stokes operator $\half + D +S$.
  The number of nodes $N=200$ and tolerance $\ep=10^{-12}$ are fixed,
  for the star-shaded domain shown in Fig.~\ref{f:setup}.
  (a) shows a sweep of the condition number ratio over the source
  ($\upsilon$) and check ($\upsilon_c$) upsampling ratios.
  (b) shows the eigenvalue magnitudes as function of $\upsilon$ for fixed
  $\upsilon_c$.
  QFS-D also internally uses the representation $S+D$.
}{f:kappa}
\efi 

\subsection{Tests of QFS-D for Laplace, Helmholtz, and Stokes PDEs}
\label{s:Dtest}

Now we test the desingularized method for evaluation in 2D,
and study the conditioning of the resulting Nystr\"om matrix
$\tilde A$.

Our layer-potential evaluation tests are
shown in Fig.~\ref{f:DmLP}.
They use similar set-ups to those from Sec.~\ref{s:Btest},
with the following differences:
i) we combine the desired layer potentials to test
the mixture $(\alpha,\beta)=(1,1)$, which is expected to follow
the worse of either $S$ or $D$ alone;
ii) in the Stokes case we choose check upsampling
$\upsilon_c=1.5$ (recall that $\upsilon=1.3$ for source upsampling).
For this star-shaped domain
the supremum of acceptable
check point distances is $\delta_{c,0} \approx 0.09$,
controlled by its Schwarz singularities \cite{Da74}.

Fig.~\ref{f:DmLP} shows that for all three PDEs
the performance of QFS-D is better than QFS-B, sometimes exceeding
the convergence rate of the gold-standard quadrature methods.
A peculiar behavior for the far target is that convergence continues
down to at least $10^{-13}$ regardless of the tolerance $\ep$;
we believe this is due to the larger source-check distance
$\delta+\delta_c$.
The near target errors saturate around $\ep$, as expected.
The Stokes case clearly shows that the
Nyquist Fourier decay of the density controls
the near-target rate for all methods,
and that there is also a common rate
for the far target, about twice the near rate.
For Stokes, QFS-D is 1-2 digits better than the barycentric
method of \cite{lsc2d} at the near target.

The convergence of iterative methods for the linear system \eqref{linsys} is
sensitive to the spectrum, so it is
crucial that any Nystr\"om quadrature scheme
well approximate the spectral properties of the operator \eqref{bie}.
For Laplace and Helmholtz the condition numbers are very
close to those from the gold-standard Kress scheme,
regardless of upsampling factors $\upsilon\ge 1$ and $\upsilon_c\ge 1$, thus we recommend that both remain at 1.

However, Stokes demands upsampling, as we now show.
The ``completed''
operator $A=\half + D + S$, commonly used for exterior no-slip BVPs
\cite{hsiao85,hebeker,biros04,gonzalez09,wu2019}
is well conditioned (eg, see proof of Lemma~\ref{l:stoimp}).
In Fig.~\ref{f:kappa}(a) we compare the condition number
$\kappa(\tilde A)$ of its Nystro\"om matrix obtained by QFS-D
to the
condition number $\kappa(A)\approx 7.2$ for $A$ filled by
Kress quadrature.
The convergence (the white-colored region) of the QFS-D
condition number in the $\upsilon$ and $\upsilon_c$ plane justifies our earlier choice of $\upsilon=1.3$ and $\upsilon_c=1.5$.
Curiously, there is a ill-conditioned ``diagonal'' region $\upsilon_c \approx \upsilon$ up to about 1.42.
All such behavior is believed to be due to eigenvalues of $E$
passing through zero.
Panel (b) verifies convergence of the spectrum of $\tilde A$
(note the expected clustering at $1/2$),
showing a good, and stable, spectrum for $\upsilon\ge 1.3$.


To verify the correctness of $\tilde A$,
we now turn to multi-body BVP applications.

\section{Application to large-scale 2D boundary-value problems}
\label{s:big2d}

In this section we measure the accuracy, convergence, and speed of the proposed method, in the context of FMM-accelerated algorithms for solving
larger-scale BVPs.
We first solve exterior Helmholtz Dirichlet (sound-hard) scattering problems with both a moderate (100) and large (1000) number of bodies.
We then solve a forced Stokes flow past a moderate number of inclusions in a confined geometry.
We emphasize that the same QFS method,
with minimal changes, is used for both PDEs.

\subsection{Geometry generation}
\label{section:large_scale:geometry}

For all problems presented in this section, the
solution domain is the exterior of
$\Omega:=\bigcup_{i=1}^K \Omega_i$, the union of many smooth bodies.
For the Stokes case only, in order to drive nontrivial flows,
the solution domain will also be bounded by an enclosing circle.
It is critical for testing the near-boundary layer-potential evaluations
that many of the bodies are {\em nearly touching}, being approximately separated
by a controllable distance $\ddist$.
Thus in Appendix~\ref{a:geometry} we present a method
which produces $K$ bodies in a prescribed but random layout while insuring that:
\begin{enumerate}
  \item all bodies are disjoint ($\Omega_i \cap \Omega_j=\{\emptyset\}$ for $i\neq j$),
  \item at least some
    pairs of bodies have a minimum separation distance in the interval $(\ddist,1.1\ddist]$, and
  \item there is a controllable amount of polydispersity (variation of body size).
\end{enumerate}

\subsection{Exterior Dirichlet Helmholtz (sound-hard) scattering BVP}
\label{section:large_scale:helmholtz}

In this section we solve the BVP \eqref{Hpde}--\eqref{SRC} in the exterior of
$\Omega$, a collection of $K$ smooth objects, each with boundary $\pO_i$.
The data $f=-u_\tbox{inc}|_\pO$ derives from an incident
plane wave $u_\tbox{inc}(\x) = e^{ik\mbf{d}\cdot\x}$ with direction
$\mbf{d} = (1,0)$.
Once the BVP solution $u$ is solved for,
we plot the physical solution $u_\tbox{inc} + u$,
as in the single-body example Fig.~\ref{f:setup}(a).

We use the standard indirect CFIE \cite{kress85}
representation \eqref{rep} with mixture
$(\alpha,\beta) = (-ik,1)$, where $k$ is the wavenumber,
ie the 2nd-kind BIE
$$
(\half + D - ikS)\tau \; = \; f~.
$$
This has the $K\times K$ block form
\be
\begin{bmatrix}
    A^{(1,1)} & A^{(1,2)}
    & \dots\\
    A^{(2,1)} & A^{(2,2)} & \dots\\
    \vdots & \vdots & \ddots
  \end{bmatrix}
  \begin{bmatrix}
    \tau^{(1)} \\
    \tau^{(2)} \\
    \vdots
  \end{bmatrix}
  =
  \begin{bmatrix}
    f^{(1)} \\
    f^{(2)} \\
    \vdots
  \end{bmatrix}~,
  \label{helmblkie}
\ee
where the interaction operator to body $i$ from body $j$ is
$A^{(i,j)} = \delta_{i,j}/2 + D_{\pO_i,\pO_j} - ik S_{\pO_i,\pO_j}$,
the subscripts on operators indicating their target, source curves.
We now discretize this BIE using an $N_i$-node periodic trapezoid rule
on the $i$th body (postponing for now the choice of $N_i$),
giving $N=\sum_{i=1}^K N_i$ total unknowns.
For convenience, and without ambiguity, we also use the
above notation for the discretized system.

Quadrature precomputation is as follows.
We use Algorithm~\ref{g:src} for source and check points,
but for more accuracy we add to
lines 3--4 the condition that the local ``speed''
(magnitude of $t$-derivative of the function in \eqref{curve})
be no less than half its corresponding value $\|\x'(t)\|$ on $\pO$.
%
For each body $\Omega_i$ we independently run {\tt QFSDprecompute} from Algorithm~\ref{g:qfs-d} to fill its self-interaction Nystr\"om matrix $A^{(i,i)}$ (from now on we drop the tilde
notation).
Following Remark~\ref{r:LU} we store the LU-factors $L^{(i)}$, $U^{(i)}$ and $\bar{P}^{(i)}C_i$, so that on-the-fly back-substitution
is used whenever {\tt QFSDapply} is called.

Since each body is simple and acoustically small,
each $A^{(i,i)}$ is relatively well conditioned, so a plain iterative
solution of \eqref{helmblkie} is possible.
However, we find that the following
standard ``one-body'' block-diagonal preconditioning
can halve the iteration count.
One solves, via non-restarted GMRES \cite{gmres} with prescribed tolerance, the stacked preconditioned density vector $\tilde\tau := \{\tilde\tau^{(i)}\}_{i=1}^K$, in
\be
  \begin{bmatrix}
    I_{N_1} &  A^{(1,2)} (A^{(2,2)})^{-1}
    & \dots\\
    A^{(2,1)} (A^{(1,1)})^{-1} & I_{N_2} & \dots\\
    \vdots & \vdots & \ddots
  \end{bmatrix}
  \begin{bmatrix}
    \tilde \tau^{(1)} \\
    \tilde \tau^{(2)} \\
    \vdots
  \end{bmatrix}
  =
  \begin{bmatrix}
    f^{(1)} \\
    f^{(2)} \\
    \vdots
  \end{bmatrix}~,
  \label{helmblkiep}
  \ee
  where $I_{N_i}$ indicates the $N_i\times N_i$ identity matrix.
Then the density vectors on each body are recovered by
$\tau^{(i)} = (A^{(i,i)})^{-1} \tilde\tau^{(i)}$.
This corresponds to preconditioning \eqref{helmblkie} from the right
by a matrix containing only the diagonal blocks $(A^{(i,i)})^{-1}$.
Here, block inverses are dense and stored for later use.

What remains is to describe the FMM-accelerated
{\em matrix-vector multiply} performed in each
GMRES iteration.
This applies the $N\times N$ matrix in
\eqref{helmblkiep} to a vector $\tilde\tau$, as follows:
\ben
\item
  Split the preconditioned density $\tilde\tau$ into vectors $\tilde\tau^{(i)}$.
\item
  Recover actual densities $\tau^{(i)} = (A^{(i,i)})^{-1} \tilde\tau^{(i)}$
  for each body $i=1,\dots,K$.
\item
  Compute QFS strength vectors $\bsig^{(i)}$ from $\tilde\tau^{(i)}$
  via {\tt QFSDapply},
  for each body $i=1,\dots,K$.
\item Send the stack of strengths $\{\bsig^{(i)}\}_{i=1}^K$
  with corresponding
  QFS source locations $\{\{\y^{(i)}_j\}_{j=1}^{N_i}\}_{i=1}^K$
  into a single point FMM call
  with all $N$ boundary nodes as targets. 
\een
Note that the Helmholtz FMM must include monopoles and dipoles scaled as in the QFS representation $D-ikS$.

With the iterative solution $\tilde\tau$ complete,
evaluation of the solution $u$ at desired target points
proceeds by doing exactly the above steps 1--4, except with
the desired FMM targets instead of the boundary nodes in step 4.
Complicated bookkeeping is absent (by comparison, in \cite{dpls,junwang},
on-boundary, near-boundary, and far targets had to be handled separately).

We use a Julia implementation of Algorithm~\ref{g:qfs-d}
that makes efficient use of a multi-core shared-memory machine, plus a
custom Julia interface to the multithreaded library FMMLIB2D \cite{HFMM2D}.


\begin{rmk}[Sparse matrix storage of quadrature corrections to the FMM]
  It is possible to fill a sparse matrix whose action on
  $\tau$ applies all self- and close-quadrature corrections to a
  point FMM between boundary nodes alone \cite{fmmbie3d}.
  One advantage of our on-the-fly approach is that storage does not grow even with many target points in the near-field. 
  We show below that the cost it adds to the FMM is usually minor.
  \label{r:sparse}
\end{rmk}

We find that
when there are many wavelengths across the entire system, or boundaries are close, or the number $K$ of bodies grows, the GMRES convergence
rate for \eqref{helmblkiep} becomes progressively poorer.
This motivates two test cases: a moderate problem
(which allows detailed comparison with the Kress scheme) with $K=100$ quite near-to-touching inclusions solved to a high tolerance ($10^{-12}$),
then a larger problem with $K=1000$ with larger $\ddist$ solved to lower tolerance ($10^{-10}$).
In both cases, body centers lie near
two entwined spiral curves, each generated by the function
$\r:\R \to \R^2$,
\be
\r(s) := (as + b)^p(\cos(s+\xi),\sin(s+\xi))~,
\label{eq:spiral_definition}
\ee
for various values $a$, $b$, and $p$.
The first spiral has $\xi=0$, the second $\xi=\pi$.
Each ``arm'' is leaky for waves, allowing partial
resonance, hence keeping the iteration count tolerable.

\subsubsection{Computers}
\label{section:computers}
Throughout these examples, timing benchmarks will be measured on three computers:
\begin{enumerate}
  \item A Macbook Pro with 16GB of RAM and a single quad-core Intel(R) Core(TM) i7-8569U CPU @ 2.80GHz,
  \item A workstation with 128GB of RAM and two six-core Intel(R) Xeon(R) CPU E5-2643 v3 @ 3.40GHz,
  \item A single compute node with 1TB of RAM and two AMD EPYC 7742 64-Core Processors @ 3.34GHz.
\end{enumerate}
We will henceforth refer to these machines as the \emph{Macbook}, \emph{Workstation}, and \emph{AMD Node}, respectively. For some of the benchmarks, we will force the computer to run all computations serially; this is done by setting the environment variables {\tt OMP\_NUM\_THREADS}, {\tt MKL\_NUM\_THREADS}, and
{\tt BLAS\_NUM\_THREADS} to 1.

\subsubsection{Moderately sized Helmholtz problem}
\label{section:large_scale:helmholtz:medium}  

We pick $K=100$ bodies with centers $\mbf{c} = \r(s) + \cnoise$,
where
$\r(s)$ is the spiral \eqref{eq:spiral_definition} with parameters $a=3$ and $b=p=1$,
$s$ is uniform random in $[\pi, 5\pi/2]$, and
$\cnoise$ is uniform random in $[-1,1]^2$.
(Specifically, this formula for $\c$ acts as the
 {\tt randomcenter} function in Appendix~\ref{a:geometry}.)
The base radius is $r_0=1$, and $\ddist=0.02$,
so that the ratio of perimeter to $\ddist$
(called $f_\tbox{clup}$ in \cite{helsing_close,dpls}) is about 300.
The wavenumber is $k=10$; the geometry
is about 86 wavelengths across.
Both the GMRES tolerance and QFS tolerance $\ep$ were $10^{-12}$.

The baseline number $N_i$ of quadrature nodes on the $i$th body
is chosen such that $2\pi R_i/N_i \approx \sqrt{\ddist}$,
where $R_i$ is the maximum body radius;
this is motivated by asymptotics that the smoothness scale of the
density varies as the square-root of the distance between curves
\cite{sanganimo} \cite[Ex.~1]{junwang}.
For convergence studies, larger $N_i$ are generated simply as integer
multiples of this,
and the {\em average $N$ per body}, $N/K$, is reported.
The geometry and physical solution $u_\tbox{inc}+u$
is shown in \Cref{figure:helmholtz_medium:solution}(a).

We compare QFS to the use of
Helmholtz Kress quadratures \cite{kress91} on the $N_i$ nodes,
combined with plain quadrature from upsampled boundary nodes
to handle nearby targets. Here boundaries are upsampled aggressively
so that their target error (recalling Theorem~\ref{t:ce}) is around $10^{-16}$.
This makes the Kress scheme very expensive, so we do not even
report CPU times for it.

\begin{figure}  
  \centering
  \begin{subfigure}[c]{0.45\textwidth}
    \centering
    \includegraphics[width=\textwidth]{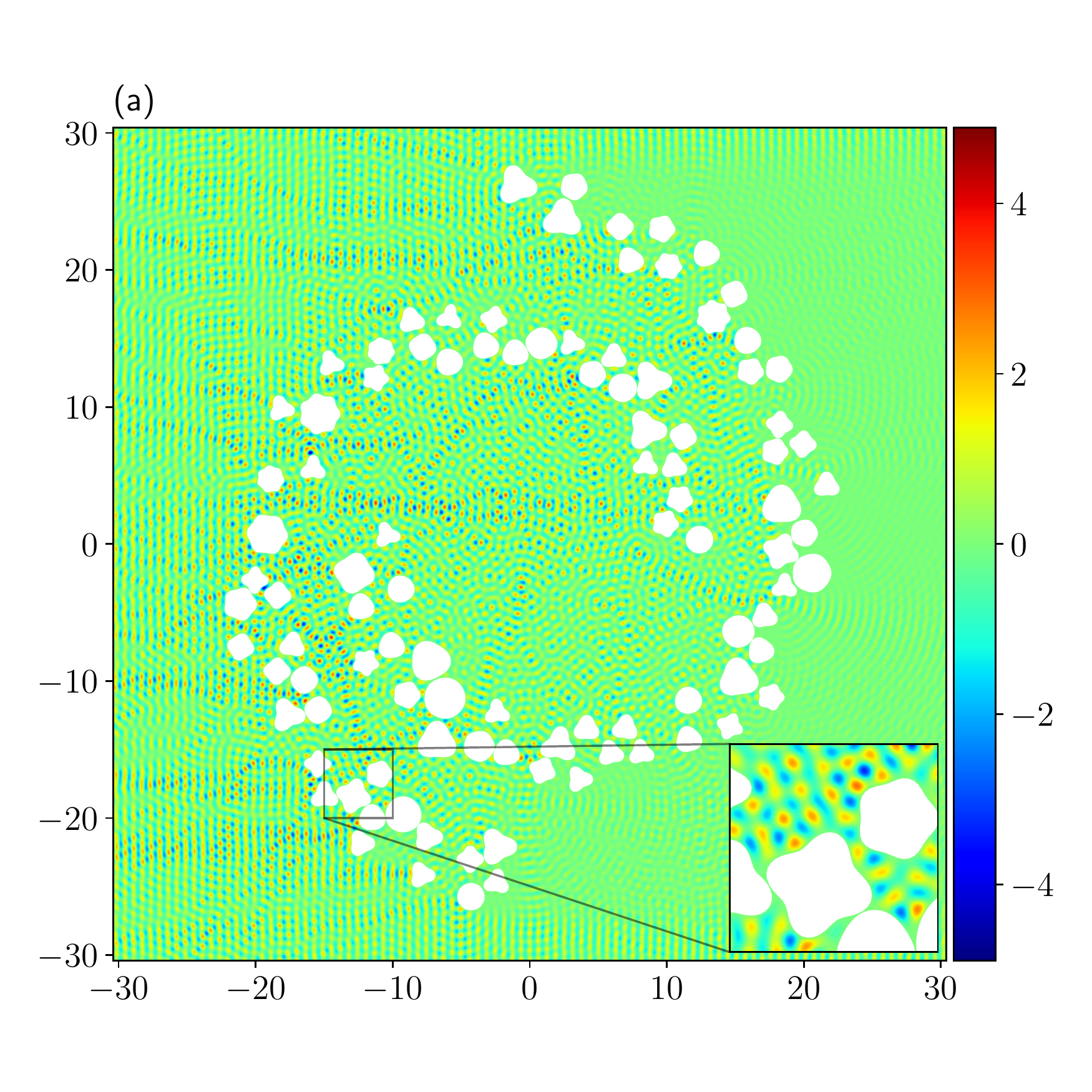}
  \end{subfigure}
  \begin{subfigure}[c]{0.45\textwidth}
    \centering
    \includegraphics[width=\textwidth]{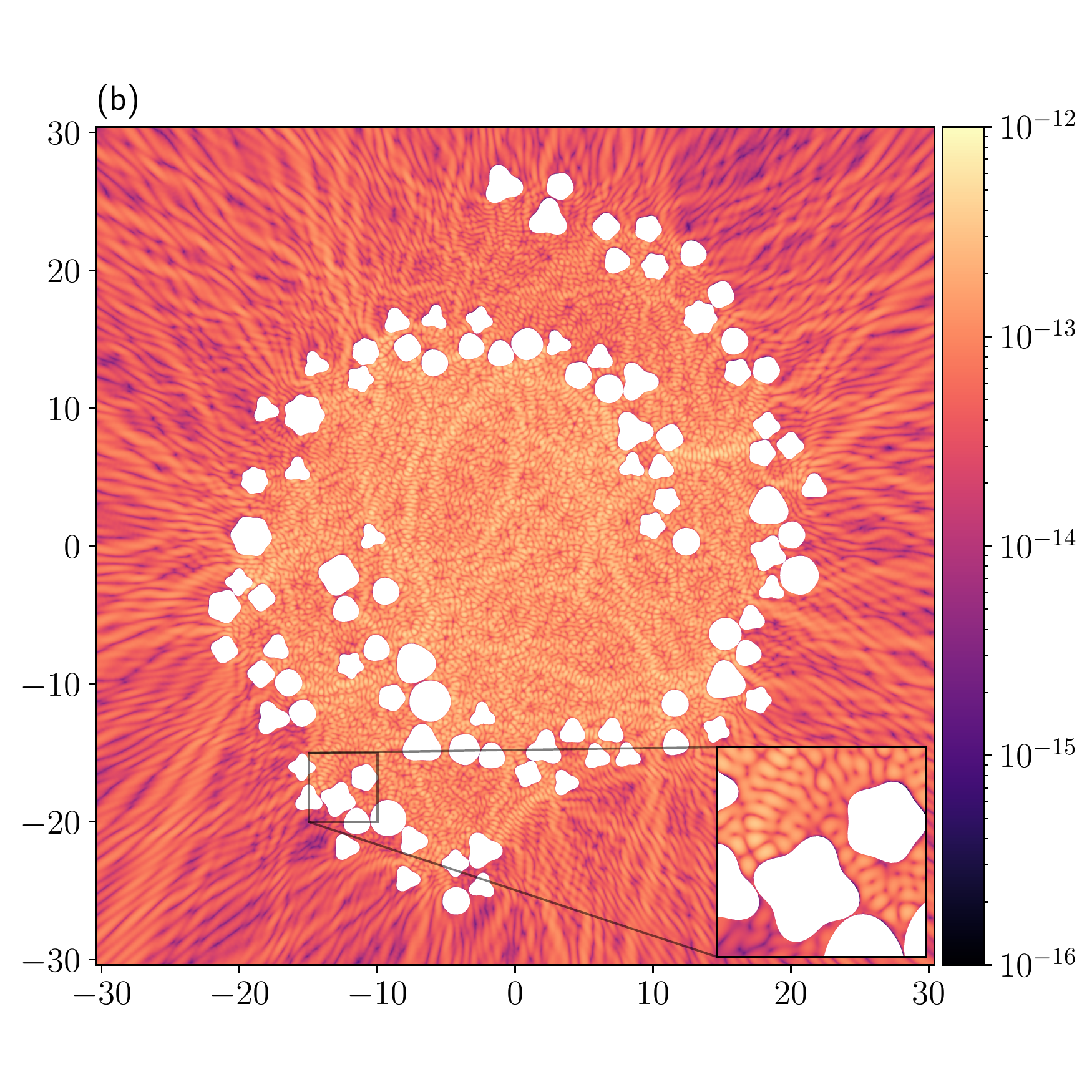}
  \end{subfigure}
  \caption{Exterior Dirichlet Helmholtz scattering from 100 inclusions.
    (a) shows the physical solution $u_\tbox{inc}+u$, and (b) the
    absolute difference between the two finest discretizations.}
  \label{figure:helmholtz_medium:solution}
\end{figure}

The absolute difference between the solution computed by the two finest discretizations is shown in \Cref{figure:helmholtz_medium:solution}(b). Its maximum is $6.70\times 10^{-13}$, less than the GMRES tolerance.
A more detailed view of convergence is shown in \Cref{figure:helmholtz_medium:refinement}(a), which also compares the QFS solution
to that with Kress with upsampling.
The black QFS ``self-convergence'' curve shows the maximum ($L^\infty$)
difference over a $2000\times2000$ grid
between solutions computed at successive levels of refinement.
Clearly, the convergence appears to be spectral, down to 12-digit accuracy.
A much higher spectral rate is observed for the self-convergence
at the (distant) target $(0,0)$, both for QFS and for Kress.
Their rates are indistinguishable,
and furthermore converge to the same answer to 12 digits (red curve).
(We note that Kress, even with upsampling, is unable to accurately
evaluate on all of the grid points used for $L^\infty$-norm testing of QFS.)

\begin{figure}  
  \centering
  \begin{subfigure}[c]{0.4\textwidth}
    \centering
    \includegraphics[width=\textwidth]{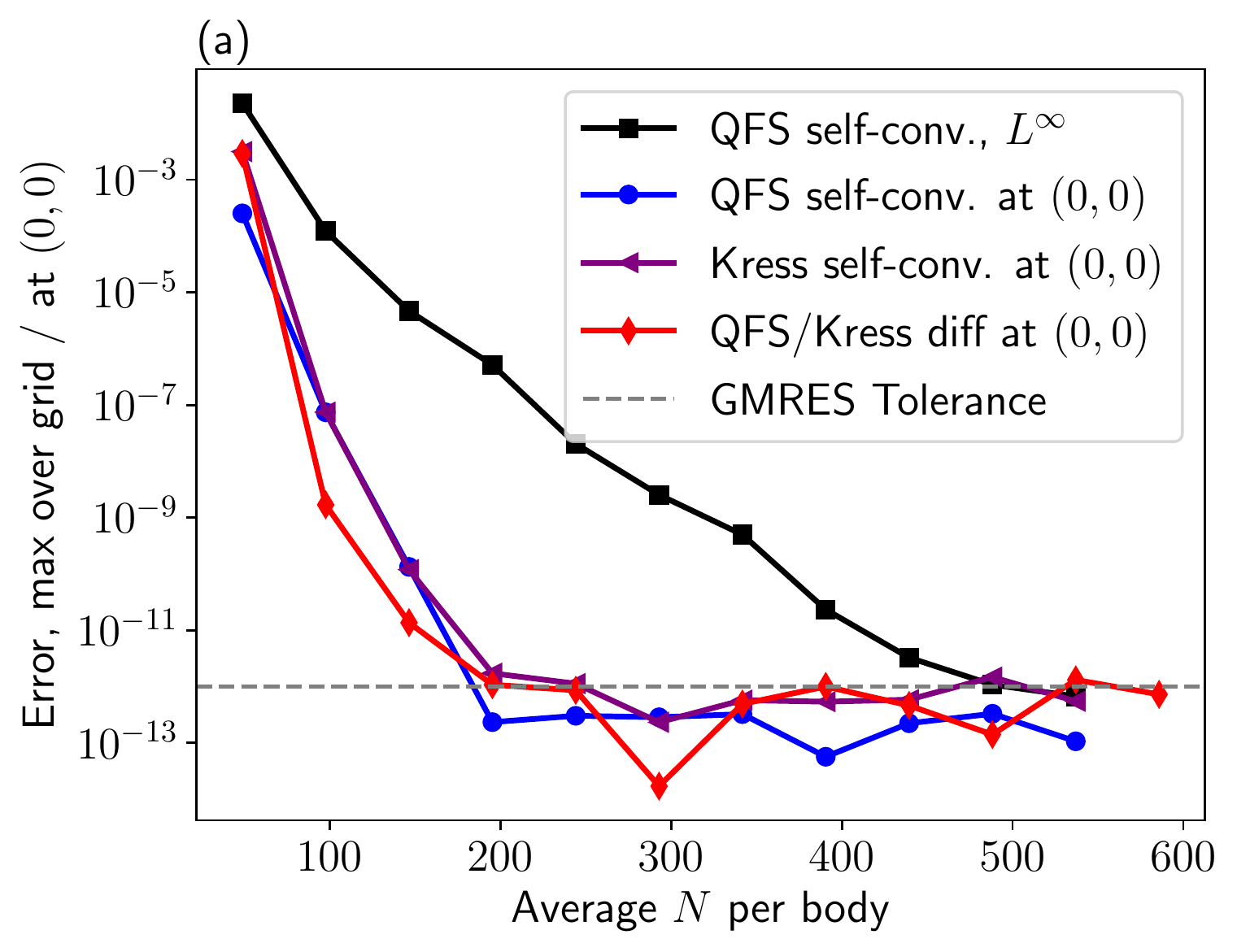}
  \end{subfigure}
  \begin{subfigure}[c]{0.29\textwidth}
    \centering
    \includegraphics[width=\textwidth]{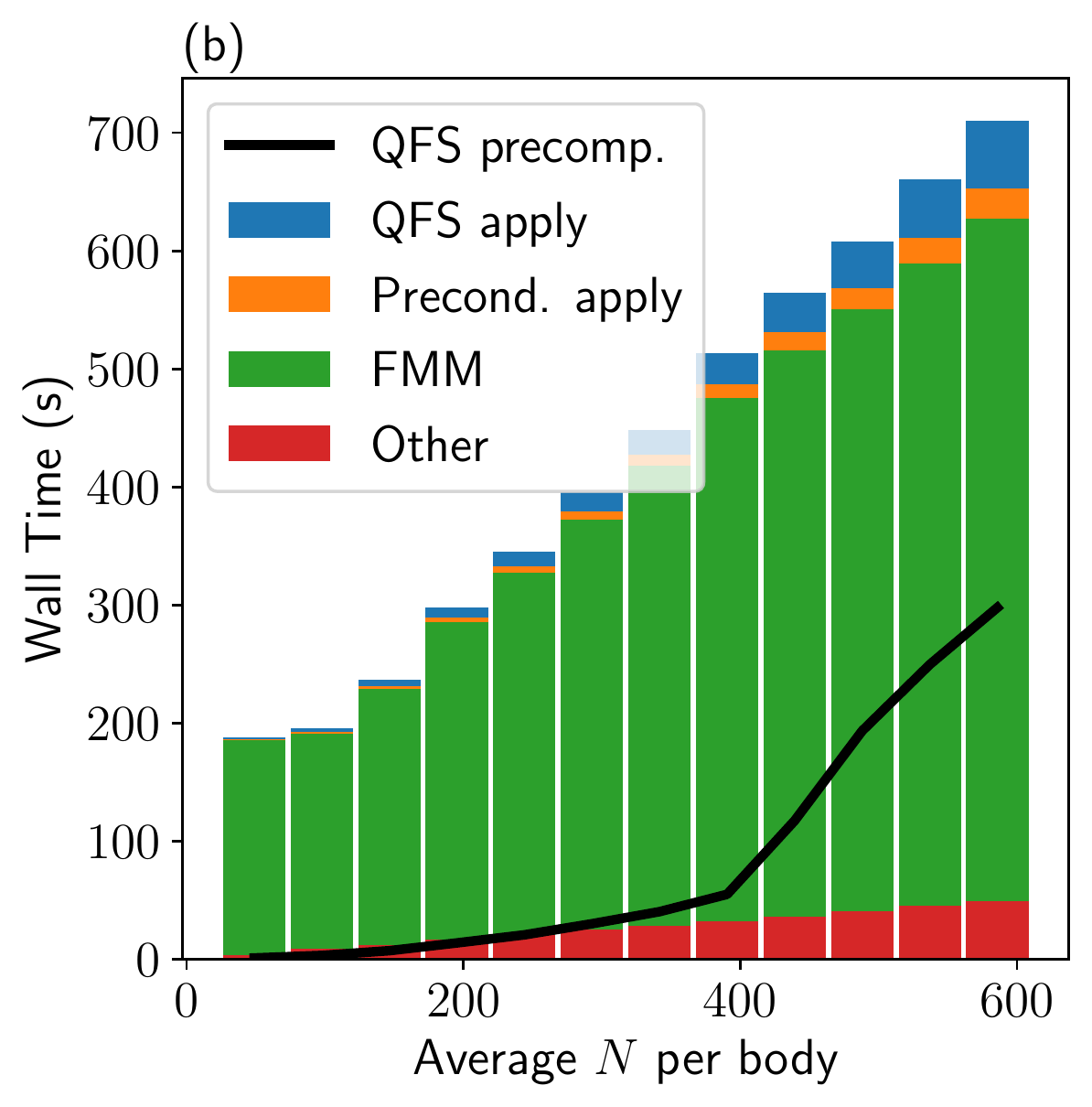}
  \end{subfigure}
  \begin{subfigure}[c]{0.29\textwidth}
    \centering
    \includegraphics[width=\textwidth]{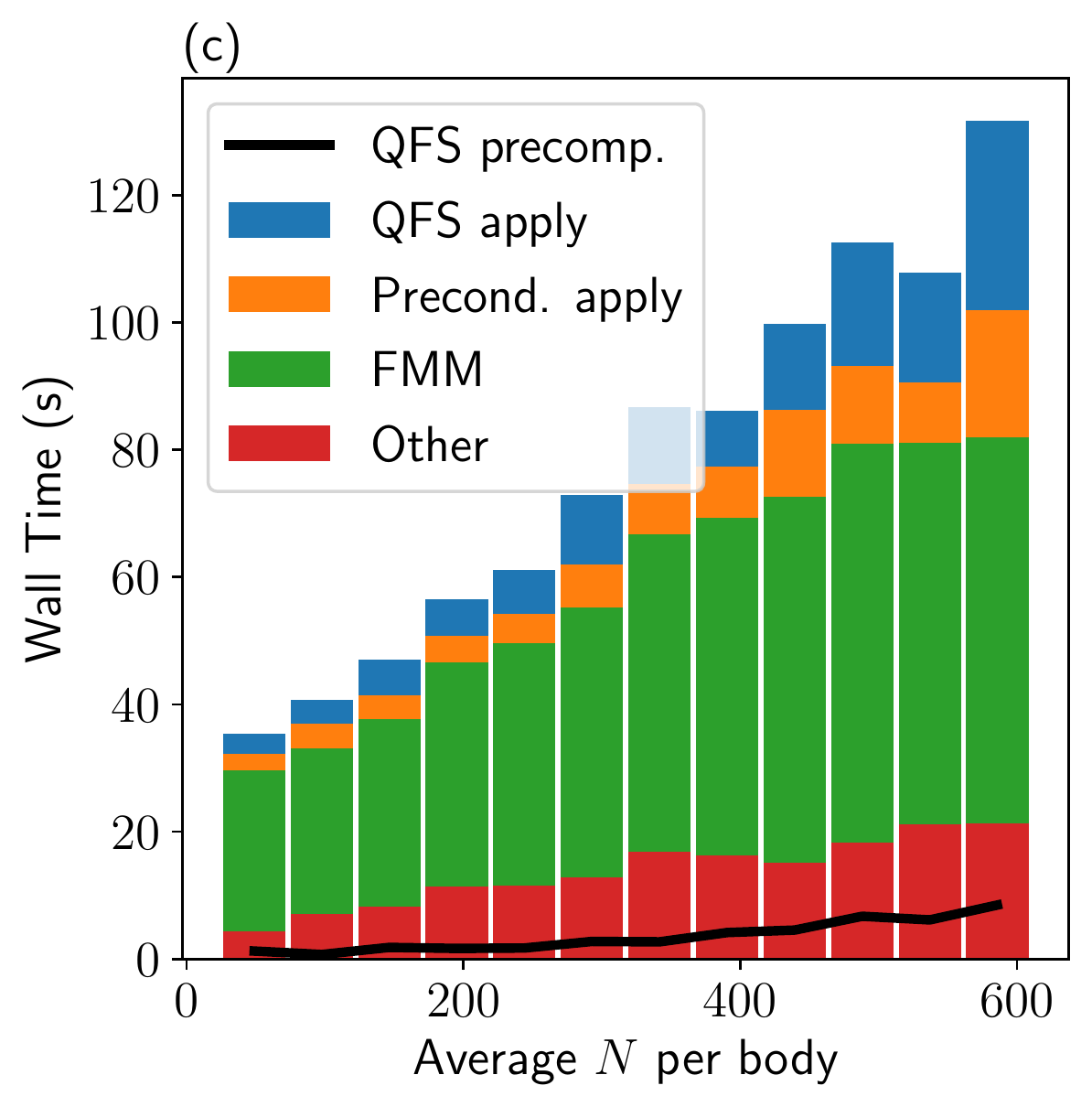}
  \end{subfigure}
  \caption{ Dirichlet Helmholtz scattering problem with 100 inclusions.
  (a) shows convergence of the proposed QFS scheme at a similar rate to that of gold-standard Kress quadratures plus expensive upsampling.
  CPU timings are shown in (b) and (c) on the Macbook and AMD Node, respectively (see Section~\ref{section:large_scale:helmholtz:medium}).}
  \label{figure:helmholtz_medium:refinement}
\end{figure}

Finally, in \Cref{figure:helmholtz_medium:refinement}(b) and (c), we show wall-clock times for these simulations on the Macbook and AMD Node, respectively.
The black line shows the QFS precomputation time, which includes the time required to form and invert the Nystrom matrices $A^{(i,i)}$.
Bars show the total time (accumulated over all GMRES iterations) for {\tt QFSDapply} routine (blue), the block-diagonal preconditioner (orange), and the point FMM (green).
The red block shows all time not accounted for by these processes, which is dominated by internals of GMRES. Although the percentage of the total solve time consumed by {\tt QFSDapply} grows as $N$ does, it never exceeds 8\% or 25\% on the Macbook and the AMD Node, respectively. In fact, for this high-iteration count problem, the time used by orthogonalization within GMRES is similar to that used by QFS for quadrature, especially as the thread-count over which the FMM and the {\tt QFSDapply} routines can be split increases. A tabulation of results is presented in \Cref{table:helmholtz:medium}. For all discretizations for both methods, the iteration count for GMRES to converge to $10^{-12}$ is exactly 854.

\begin{rmk}[Optimizations for repeated objects]
  The timing results presented here are actually a worst-case scenario: all of the scattering surfaces are unique so need their own QFS precomputation. If objects repeat, two optimizations appear: (1) the {\tt QFSDprecompute} stage needs to be done only once for each unique object, and (2), because the matrices appearing in the {\tt QFSDapply} algorithm are the same, all steps in \Cref{g:qfs-d} can be packed together into highly-optimized BLAS3/LAPACK
  calls. These optimizations are possible for the QFS-B algorithm, as well.
\end{rmk}

\begin{table}
  \centering
  \begin{tabular}{l|rrrrrr}
    \toprule
    Average $N$ per body    & 49     & 147     & 244     & 342     & 439     & 537     \\
    \midrule
    GMRES Iter., QFS        & 854    & 854     & 854     & 854     & 854     & 854     \\
    GMRES Iter., Kress      & 854    & 854     & 854     & 854     & 854     & 854     \\
    Self-conv., $L^\infty$  & 2.3e-2 & 4.7e-6  & 2.0e-8  & 5.0e-10 & 3.3e-12 & 6.7e-13 \\
    Self-conv., $(0,0)$     & 2.5e-4 & 1.3e-10 & 3.0e-13 & 3.3e-13 & 2.2e-13 & 1.0e-13 \\
    Kress/QFS Diff          & 2.9e-3 & 1.3e-11 & 1.1e-12 & 4.6e-13 & 5.2e-13 & 8.3e-13 \\
    \midrule
    \multicolumn{7}{l}{Timing (in seconds, Macbook)}                                   \\ 
    \midrule
    QFS Precomp.           & 0.9    & 7.1     & 20.5    & 40.1    & 117.2   & 249.5   \\
    QFS Apply              & 1.5    & 5.4     & 12.0    & 21.3    & 33.3    & 48.9    \\
    Precond. apply         & 0.6    & 2.2     & 5.5     & 9.3     & 15.0    & 22.1    \\
    FMM                    & 181.9  & 216.8   & 307.1   & 389.8   & 479.7   & 544.3   \\
    Other                  & 6.7    & 10.1    & 16.1    & 16.7    & 20.3    & 23.5    \\
    Total Solve            & 187.7  & 236.4   & 344.8   & 448.4   & 564.2   & 660.2   \\
    QFS \% of Total        & 0.8    & 2.3     & 3.5     & 4.8     & 5.9     & 7.4     \\
    \bottomrule
  \end{tabular}
  \caption{Tabulation of the results from the moderately size Helmholtz problem from \Cref{section:large_scale:helmholtz:medium}. See discussion and \Cref{figure:helmholtz_medium:solution,figure:helmholtz_medium:refinement} for further analysis.}
  \label{table:helmholtz:medium}
\end{table}

\subsubsection{Large Helmholtz problem}
\label{section:large_scale:helmholtz:large}

We repeat the tests from the previous section, with $K=1000$ obstacles
and spiral parameters $a=4$, $b=7$ and $p=1.4$,
with $s$ in \eqref{eq:spiral_definition}
uniform random in $[\pi, 7\pi/2]$,
and $\c=\r(s)+\cnoise$,
but now $\cnoise$ is
uniform random in $[-2,2]^2$.
We set $\delta=0.05$, two and a half times larger than before.
The GMRES and QFS tolerances are $10^{-10}$.
The wavenumber is $k=1$, giving about 76 wavelengths across;
see \Cref{figure:helmholtz_large:solution}(a).

\begin{figure}   
  \centering
  \begin{subfigure}[c]{0.45\textwidth}
    \centering
    \includegraphics[width=\textwidth]{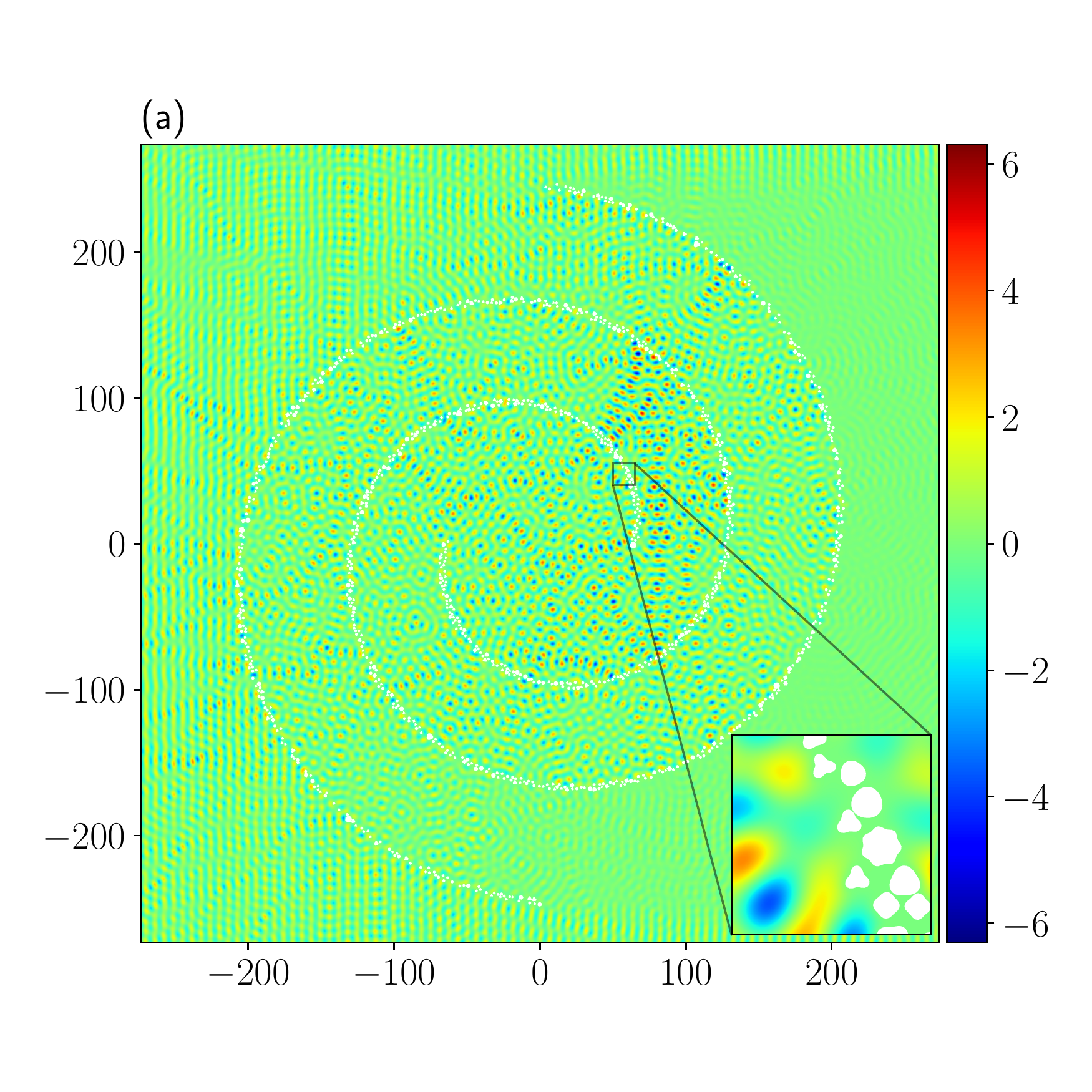}
  \end{subfigure}
  \begin{subfigure}[c]{0.45\textwidth}
    \centering
    \includegraphics[width=\textwidth]{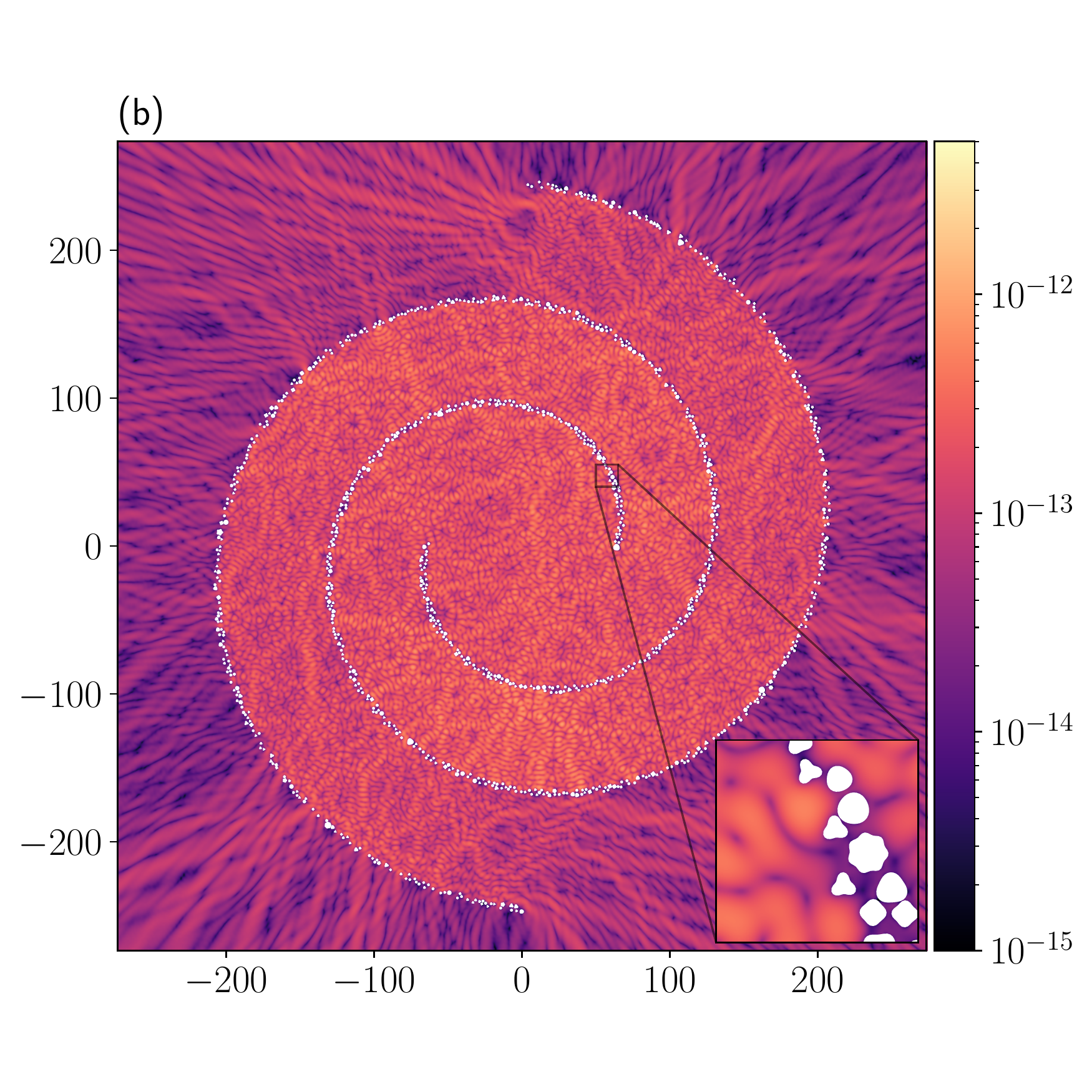}
  \end{subfigure}
  \caption{Exterior Dirichlet Helmholtz scattering from 1000 inclusions.
    (a) shows the physical solution $u_\tbox{inc}+u$, and (b) the
    absolute difference between the two finest discretizations.}
  \label{figure:helmholtz_large:solution}
\end{figure}

As in \Cref{section:large_scale:helmholtz:medium}, we measure self-convergence in $L^\infty$ and at the far field point $(0,0)$, but do not compare to a Kress-based solver due to computational cost.
A self-convergence study is shown in \Cref{figure:helmholtz_large:refinement}(a), consistent with a spectral rate, with convergence stagnating below the GMRES tolerance for both near and far targets.
Timings on the AMD Node are shown in \Cref{figure:helmholtz_large:refinement}(b).
As before, the {\tt QFSDapply} stage consumes a small portion of the total solve time ($<22\%$).
A tabulation of all results is presented in \Cref{table:helmholtz:large}.

\begin{figure}
  \centering
  \begin{subfigure}[c]{0.45\textwidth}
    \centering
    \includegraphics[width=\textwidth]{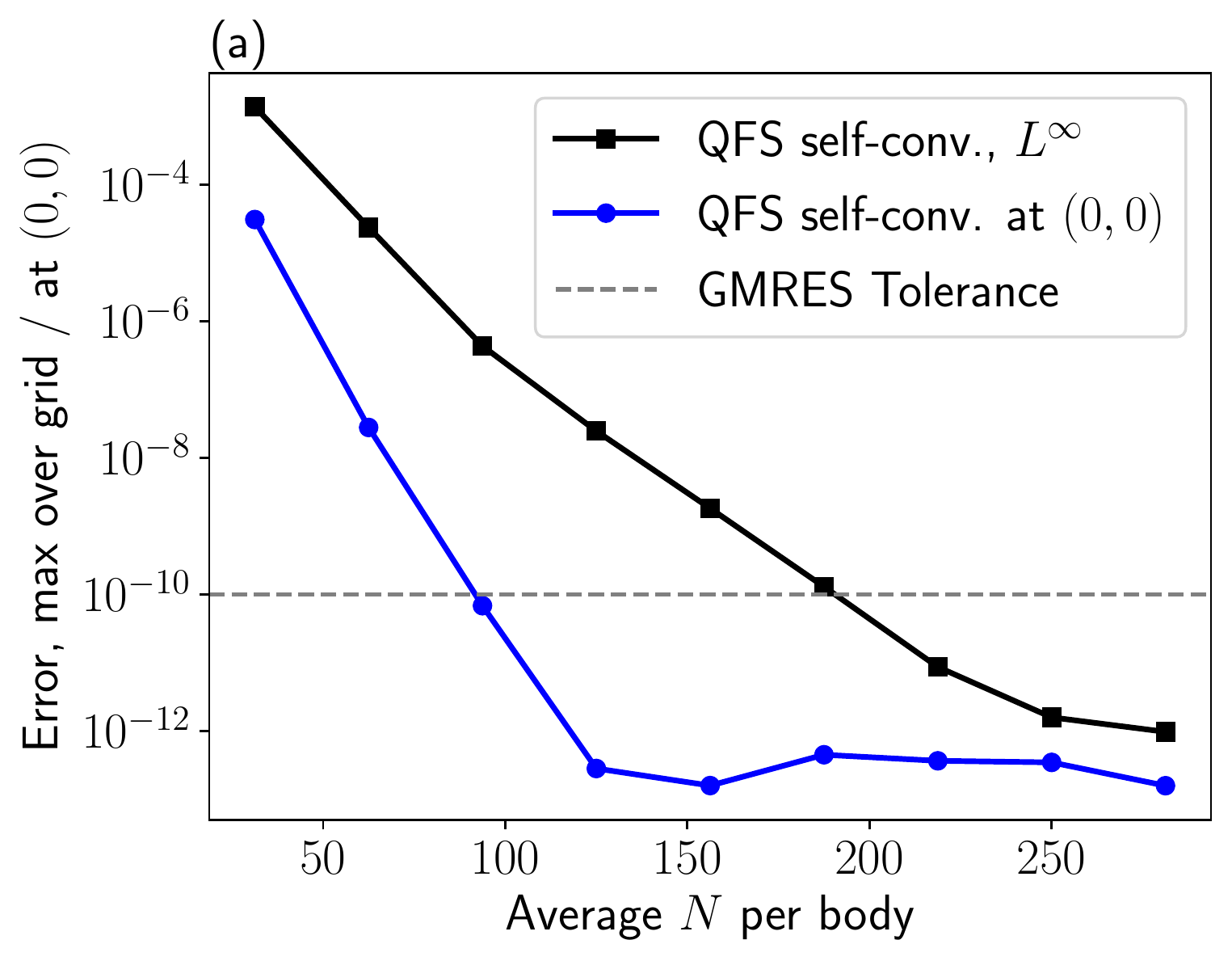}
  \end{subfigure}
  \begin{subfigure}[c]{0.45\textwidth}
    \centering
    \includegraphics[width=\textwidth]{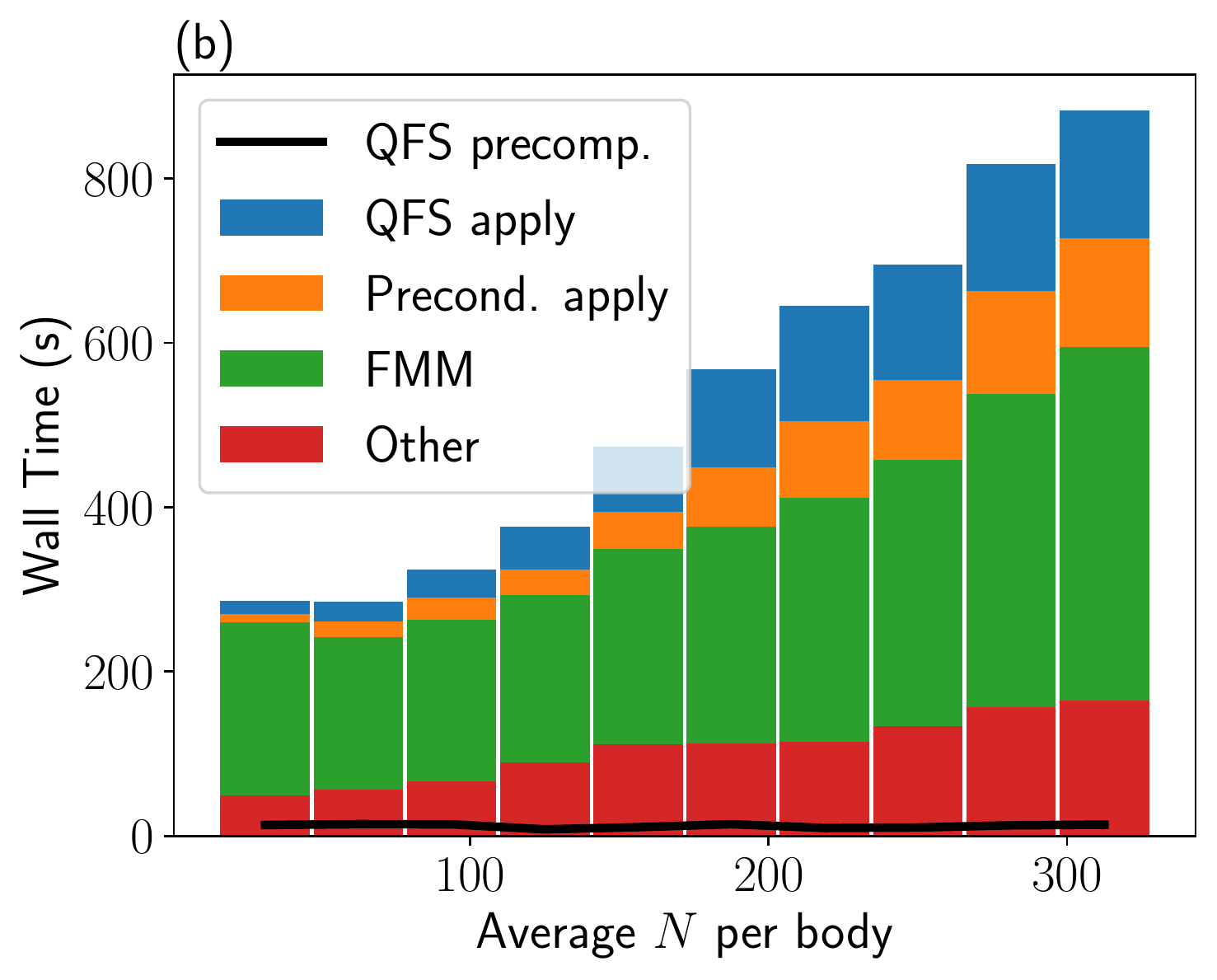}
  \end{subfigure}
  \caption{Exterior Dirichlet Helmholtz scattering problem with 1000 inclusions. Panel (a) shows self-convergence results, both maximum over the grid (black) and at $(0,0)$ (blue). Panel (b) shows computational timings.}
  \label{figure:helmholtz_large:refinement}
\end{figure}

\begin{table}
  \centering
  \begin{tabular}{l|rrrrr}
    \toprule
    Average $N$ per body   & 31     & 94      & 156     & 219     & 281     \\
    \midrule
    GMRES Iterations       & 1331   & 1331    & 1331    & 1331    & 1331    \\
    Self-conv., $L^\infty$ & 1.4e-3 & 4.3e-7  & 1.8e-9  & 8.6e-12 & 9.7e-13 \\
    Self-conv., $(0,0)$    & 3.1e-5 & 6.8e-11 & 1.6e-13 & 3.6e-13 & 1.6e-13 \\
    \midrule
    \multicolumn{6}{l}{Timing (in seconds, AMD Node) }        \\ 
    \midrule
    QFS Precomp.           & 13.2   & 14.1    & 10.6    & 9.6     & 12.9    \\
    QFS Apply              & 16.0   & 33.8    & 79.3    & 140.1   & 155.3   \\
    Precond. apply         & 9.5    & 27.5    & 45.7    & 93.2    & 125.2   \\
    FMM                    & 210.1  & 196.7   & 237.6   & 296.5   & 381.6   \\
    Other                  & 49.5   & 66.4    & 111.4   & 114.9   & 156.2   \\
    Total Solve            & 285.8  & 324.4   & 474.0   & 645.3   & 818.2   \\
    QFS \% of Total        & 5.6    & 10.4    & 16.7    & 21.8    & 18.9    \\
    \bottomrule
  \end{tabular}
  \caption{Tabulation of the results from the large Helmholtz problem from \Cref{section:large_scale:helmholtz:large}.}
  \label{table:helmholtz:large}
\end{table}

\subsection{Driven Stokes flow in 2D}
\label{section:large_scale:stokes}

In this section we consider pressure-driven Stokes flow past a fixed array of
obstacles. The centers for the individual obstacles are drawn by rejection sampling $\mathbf{c}$ uniformly in $[-15,15]^2$, discarding those with $|\mathbf{c}|>14-\ddist$, with $\ddist=0.05$, and geometry generation otherwise proceeds as described in \Cref{a:geometry}, with the additional stipulation that no part of any obstacle curve can lie within $\ddist$ of an outer confining ring of radius $15$. The flow is forced by fixing constant $\mathbf{u}=(1,0)$
velocity (Dirichlet) data
on the outer boundary; a no slip condition ($\mathbf{u}=\boldsymbol{0}$) is enforced at all inner obstacles.
This differs from \eqref{Spde1}--\eqref{Sbc} only in that the domain
is bounded, removing the decay condition.
The solution $(\u,p)$ is then unique up to a constant in $p$.
\Cref{figure:stokes:solution} shows the geometry.

We use the completed formulation ${\cal D} + {\cal S}$
on each body, which is proven to remove their nullspaces \cite{hsiao85}
\cite[p.50]{manasthesis}.
On the outer circle, denoted by $\pO_0$, we use a plain
DLP, but perturb the self-interaction of that block
with a rank-1 operator $\n\n^\intercal$, where $\n$ in the unit normal.
This removes the 1D nullspace associated with the interior BVP
\cite[Table~2.3.4]{HW} \cite{biros04}.
With $j=1,\dots,K$ indexing the interior obstacles,
the BIE takes the $(K+1)\times (K+1)$ block form
\be
\begin{bmatrix}
    -\half+D_{\pO_0,\pO_0} + \n\n^\intercal & D_{\pO_0,\pO_1}+S_{\pO_0,\pO_1}
    & \dots\\
D_{\pO_1,\pO_0} & \half +D_{\pO_1,\pO_1} + S_{\pO_1,\pO_1} & \dots\\
    \vdots & \vdots & \ddots
  \end{bmatrix}
  \begin{bmatrix}
    \btau^{(0)} \\
    \btau^{(1)} \\
    \vdots
  \end{bmatrix}
  =
  \begin{bmatrix}
    \mbf{f}^{(0)} \\
    \mbf{f}^{(1)} \\
    \vdots
  \end{bmatrix}~.
  \label{stokesblkie}
\ee
Each boundary is discretized with the PTR as for Helmholtz in \Cref{section:large_scale:helmholtz}, giving $N=2\left(N_\text{circ}+\sum_{i=1}^K N_i\right)$ total unknowns.
The discretization of $\n\n^\intercal$ is as in Remark~\ref{rmk:pressure_null}.
The QFS mixture becomes $(\ta,\tb)=(1,1)$, proven
to be robust for Stokes in Theorem~\ref{t:qfssto}.
Other than this, the use of QFS-D is almost identical to the Helmholtz case,
a key advantage of the scheme.

A couple of remarks are in order, mostly relating to the new enclosing boundary and pressure evaluation.

\begin{rmk}[Circular boundaries]
  In the special case of a circular boundary, the ``check from source'', ``check from boundary'', and ``boundary from source'' matrices $E$, $C$, and $B$ are
  circulant (immediately for scalar PDEs, and so long as expressed in
  a polar
  coordinate system for vector PDEs). Thus for $N_\text{circ}$ large,
  the circulant property may be exploited via FFTs to stably apply and invert
  $D_{\pO_0,\pO_0}$.
\end{rmk}

\begin{rmk}[Rank-deficiency removal for interior Stokes QFS]
  We use QFS on the confining boundary $\pO_0$
  by negating $\delta$ and $\delta_c$ for interior evaluation.
  The ``check from source'' matrix $E$ now must have a 1D
  null-space associated with the pressure constant ambiguity.
  Although this can be dealt with by regularizing the SVD
  \eqref{svdE},
  we prefer to use a low-rank perturbation to remove rank deficiency.
  We replace $E$ by
  \begin{equation}
    E_\text{augmented}=E + \n_\text{check} (w_\text{source}\n_\text{source})^\intercal,
  \end{equation}
  where $w_\text{source}$ elementwise-multiplies by arc-length quadrature weights for the source curve $\gamma$, and $\n$ indicates a column-vector of normals at
  source or check points.
  \label{rmk:pressure_null}
\end{rmk}

\begin{rmk}[QFS for pressure evaluation]
  The kernels \eqref{SG}--\eqref{SD} are for velocity $\u$ evaluation.
  With $\btau$ solved for in \eqref{stokesblkie}, the pressure solution $p$ may also
  be evaluated via the associated 2D pressure kernels
  \cite{Ladyzhenskaya} \cite[Sec.~2.3]{HW}
  \be
  G_p(\x,\y) = \frac{1}{2\pi}\frac{\r}{r^2},
  \quad
  D_p(\x,\y) = \frac{\mu}{\pi}\biggl(-\frac{\n_\y}{r^2} + 2(\r\cdot \n_\y) \frac{\r}{r^4}\biggr),
  \quad \r := \x-\y, \quad r:=\|\r\|
  \label{SP}
  \ee
  with the QFS source strengths $\bsig$, analogous to \eqref{qfsu}.
  Since the interior BVP solution is only defined up to a pressure constant, we are done.

  However, a more general requirement is that QFS evaluate $p$ with the correct
  constant. This is guaranteed in exterior QFS, because both the desired
  representation \eqref{rep} and QFS source representation have pressures
  vanishing as $r\to\infty$, simply because the kernels \eqref{SP} do.
  For interior QFS evaluation the collocation of $\u$ on $\gamma_c$ leaves
  $p$ generally off by a constant.
  To fix this we compute the $p$ difference relative to plain quadrature
  at a far-field point, then add this difference as a multiple
  of $\n_\text{source}$ to the QFS source strengths. This corrects $p$ but leaves $\u$ unaffected.
\end{rmk}

The density $\btau$ is found by solving \Cref{stokesblkie} using one-body right-preconditioned GMRES to a tolerance of $10^{-9}$, as in \Cref{section:large_scale:helmholtz}.
\Cref{figure:stokes:solution} shows the resulting solution $(\u,p)$ evaluated
by QFS.
Forcing the flow through these tightly packed obstacles requires an enormous pressure gradient, and large flows avoid the narrowest constrictions, instead rushing through the widest contiguous paths. The pointwise difference in $u/\|\mathbf{u}\|_{L^\infty}$ between the two finest resolutions is shown in panel (c); the largest difference is $8.7\times 10^{-11}$.

\begin{figure}
  \centering
  \begin{subfigure}[c]{0.32\textwidth}
    \centering
    \includegraphics[width=\textwidth]{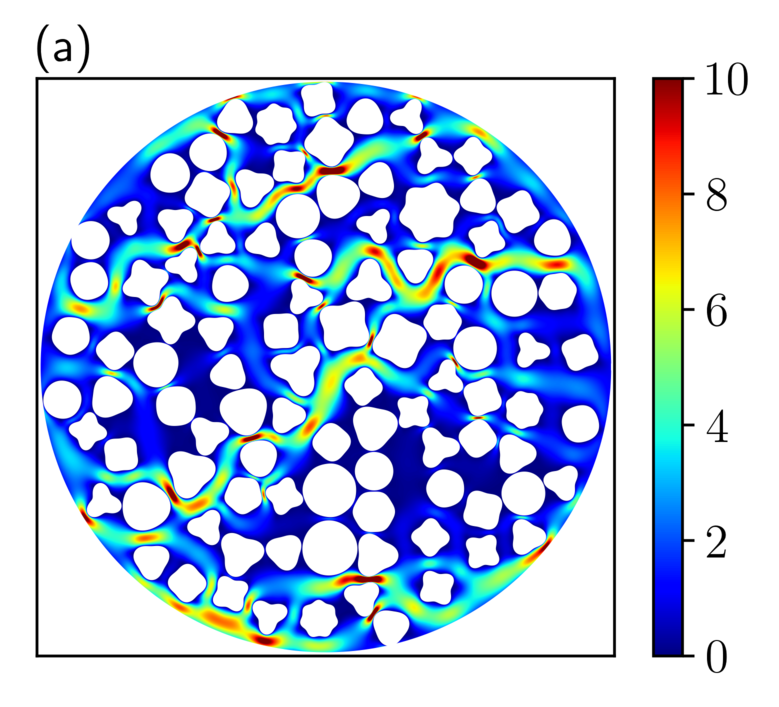}
  \end{subfigure}
  \begin{subfigure}[c]{0.32\textwidth}
    \centering
    \includegraphics[width=\textwidth]{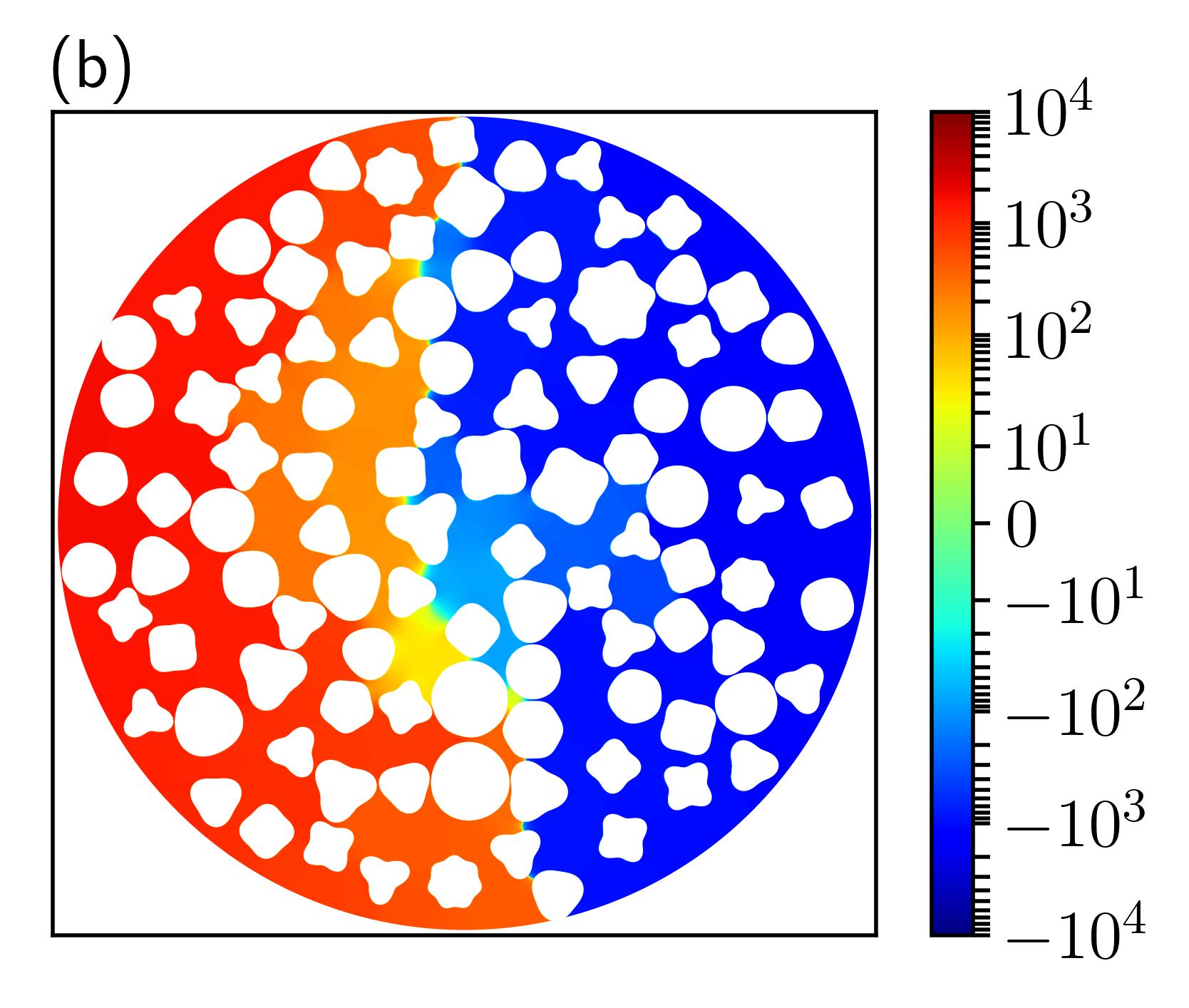}
  \end{subfigure}
  \begin{subfigure}[c]{0.32\textwidth}
    \centering
    \includegraphics[width=\textwidth]{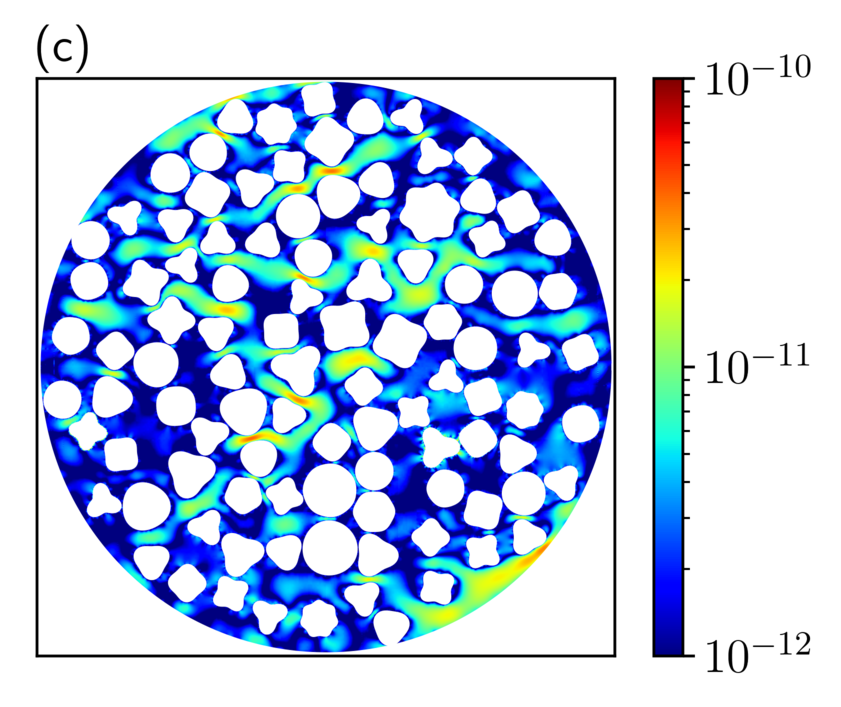}
  \end{subfigure}
  \caption{Driven Stokes flow around 100 no-slip inclusions. Panels (a) and (b) show the speed $|\u|$ and pressure $p$, respectively, computed at the finest discretization. Panel (c) shows the estimated error in $u$, normalized by $\|\mathbf{u}\|_{L^\infty}$.}
  \label{figure:stokes:solution}
\end{figure}

The convergence with respect to the mean value of $N_i$ appears spectral
in \Cref{figure:stokes:refinement}(a),
with a faster far-field rate than in the near-field.
We compare QFS again to the Kress quadratures (which apply to the log-singular
$S$ operator): for both $\u$ and $p$,
they converge rapidly to one another at a far-field point.

\Cref{figure:stokes:refinement}(b) shows timings on the Workstation, using all 12 cores. Unlike for our Helmholtz implementation, the {\tt QFSDapply} routine often uses a large percentage of the compute time. The Stokes code used here is based on an older
Python implementation that uses an unthreaded loop over the bodies, relying on BLAS for parallelization, which gives poor scaling when solving many small-sized problems. In contrast, we use a wrapper to a Fortran biharmonic 2D FMM
\cite{manasthesis} that
scales well across processors, leaving the {\tt QFSDapply} stage to dominate the computation.
Note that since $\upsilon=1.3$, the FMM involves about $1.3 N$ sources
and $N$ targets.
Panel (c) shows timings on the same computer, forced to run in serial; now the {\tt QFSDapply} stage takes no more than $23\%$ of the total time for any discretization. These serial results are collected in \Cref{table:stokes}. For very sparse discretizations, QFS uses slightly more iterations than Kress (592 vs. 571), but the differences disappear upon refinement.

\begin{figure}
  \centering
  \begin{subfigure}[c]{0.4\textwidth}
    \centering
    \includegraphics[width=\textwidth]{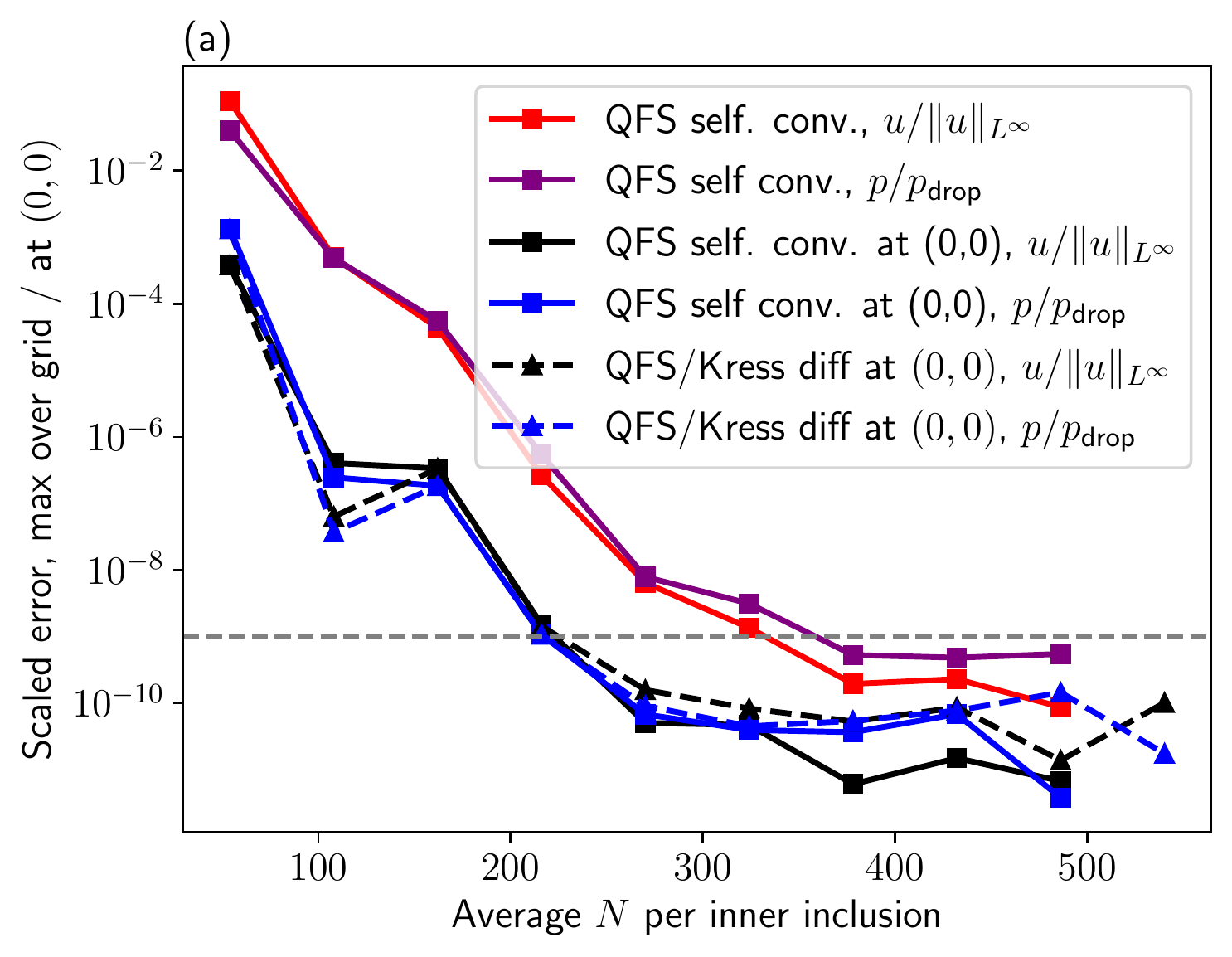}
    \caption{\centering}
  \end{subfigure}
  \begin{subfigure}[c]{0.29\textwidth}
    \centering
    \includegraphics[width=\textwidth]{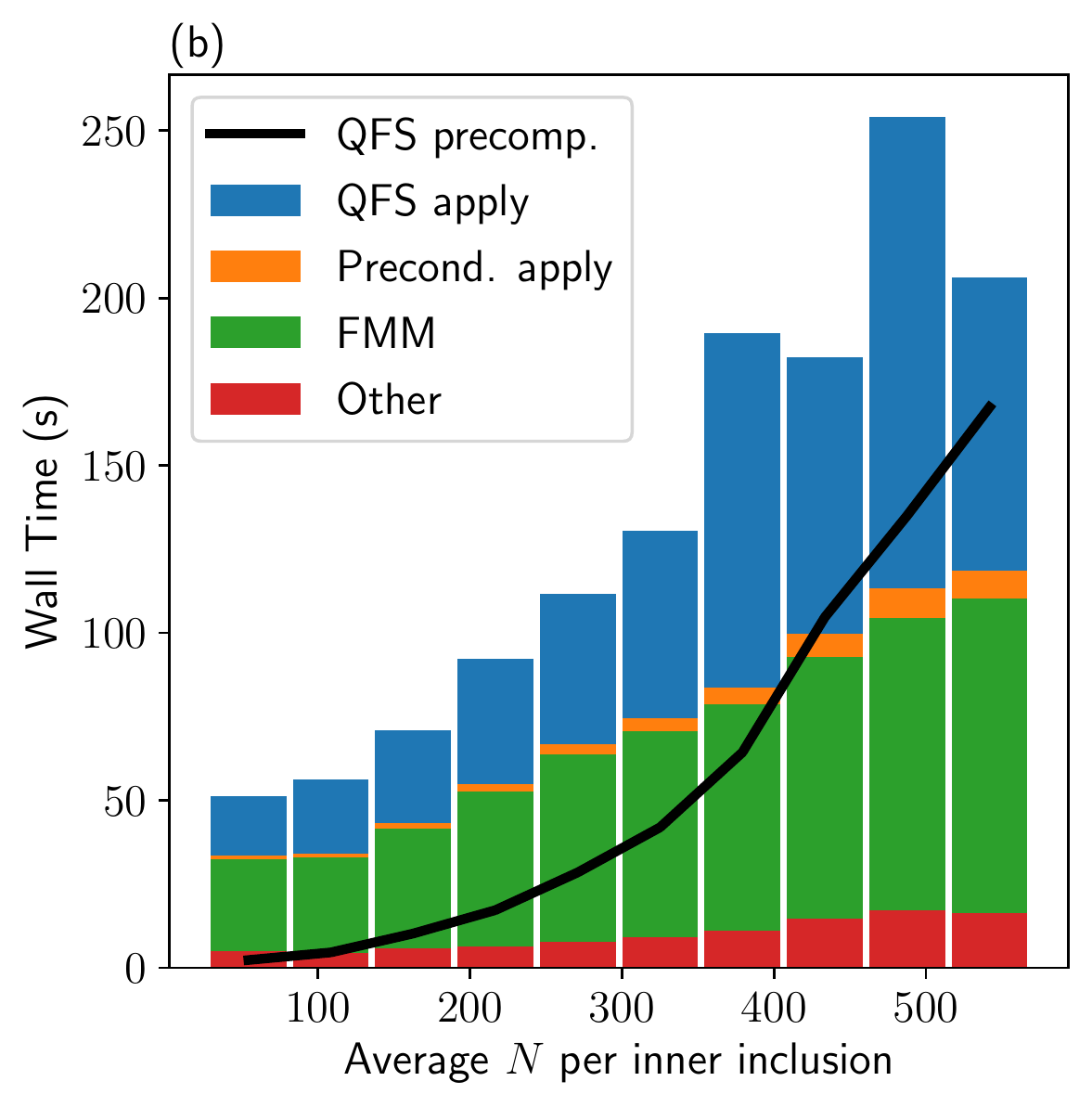}
    \caption{\centering}
  \end{subfigure}
  \begin{subfigure}[c]{0.29\textwidth}
    \centering
    \includegraphics[width=\textwidth]{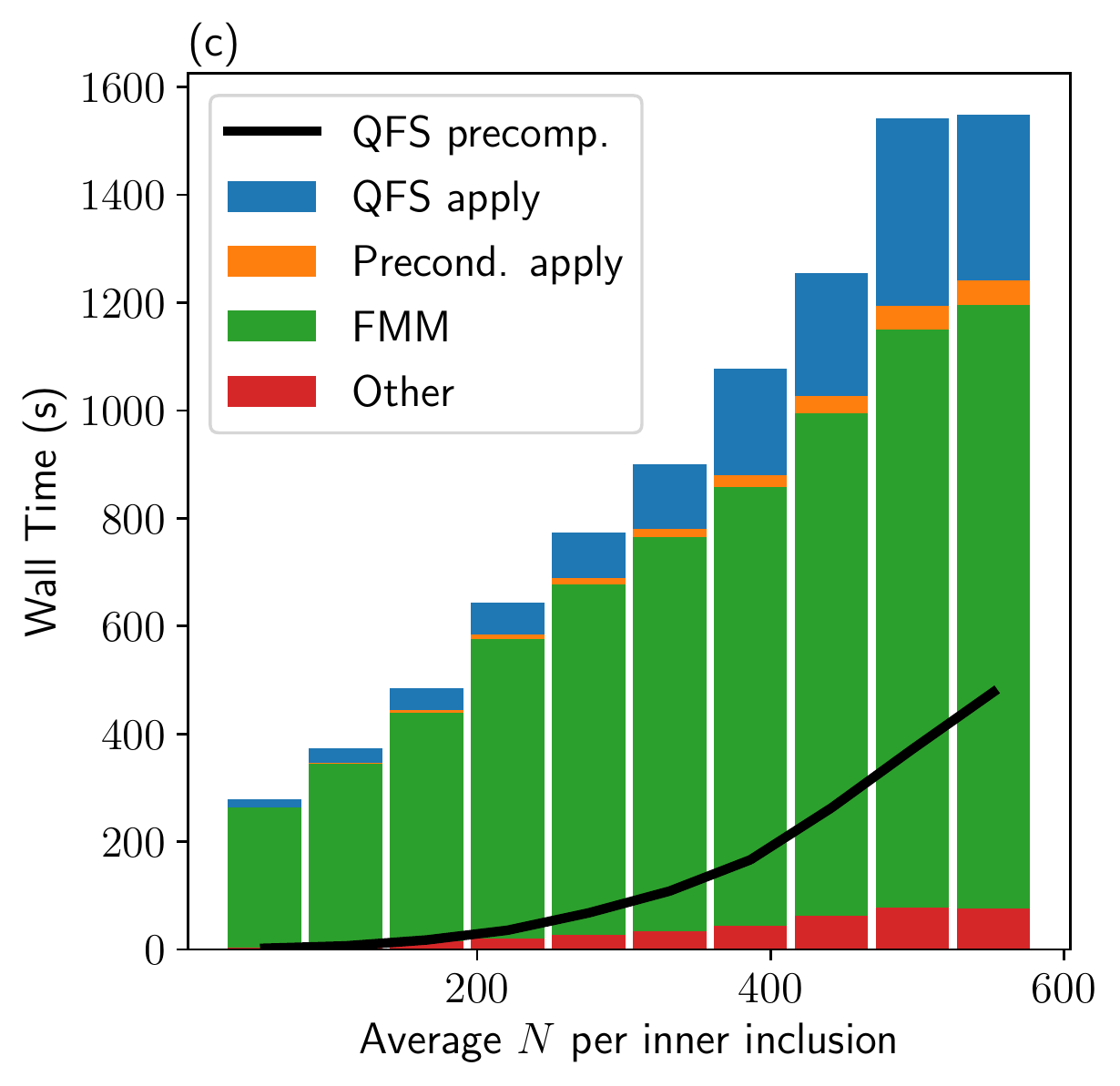}
    \caption{\centering}
  \end{subfigure}
  \caption{Driven 2D Stokes flow past 100 inclusions. Panel (a) shows the self-convergence of QFS (measured both over a grid and at only a far-field point) and the difference between the QFS solution and gold-standard Kress solution at a far-field point. Errors in $u$ are normalized by $\|\mathbf{u}\|_{L^\infty}$; errors in $p$ are normalized by $p_\text{drop}$, the maximal pressure drop over the domain. Panels (b) and (c) show computational timings on the Workstation; using all cores (b) or in serial (c).}
  \label{figure:stokes:refinement}
\end{figure}

\begin{table}
  \centering
  \begin{tabular}{l|rrrrr}
    \toprule
    Average $N$ per body      & 54     & 162     & 270    & 378    & 486     \\
    \midrule
    GMRES Iterations, QFS    & 592    & 571     & 570    & 573     & 570     \\
    GMRES Iterations, Kress  & 571    & 578     & 570    & 570     & 571     \\
    \midrule
    \multicolumn{6}{l}{Convergence for $u/\|u\|_{L^\infty}$}                 \\
    \midrule
    Self-conv., $L^\infty$   & 1.1e-1 & 4.4e-5 & 6.4e-9  & 2.0e-10 & 8.7e-11 \\
    Self-conv., $(0,0)$      & 3.8e-4 & 3.4e-7 & 5.0e-11 & 6.1e-12 & 6.8e-12 \\
    QFS/Kress diff, $(0,0)$  & 3.8e-4 & 3.4e-7 & 1.6e-10 & 5.3e-11 & 1.4e-11 \\
    \midrule
    \multicolumn{6}{l}{Convergence for $p/p_\text{drop}$}                    \\
    \midrule
    Self-conv., $L^\infty$   & 4.9e-2 & 554e-5 & 7.9e-9  & 5.3e-10 & 5.5e-10 \\
    Self-conv., $(0,0)$      & 1.3e-3 & 1.8e-7 & 6.8e-11 & 3.7e-11 & 3.8e-12 \\
    QFS/Kress diff, $(0,0)$  & 1.3e-4 & 1.8e-7 & 9.1e-11 & 5.4e-11 & 1.4e-10 \\
    \midrule
    \multicolumn{6}{l}{Timing (in seconds, Workstation, serial)}             \\ 
    \midrule
    QFS Precomp.             & 2.4    & 17.7   & 67.9    & 166.5   & 371.6   \\
    QFS Apply                & 14.9   & 40.1   & 84.2    & 196.3   & 348.7   \\
    Precond. apply           & 0.9    & 5.4    & 11.4    & 22.3    & 43.4    \\
    FMM                      & 259.2  & 424.0  & 651.4   & 814.4   & 1072.3  \\
    Other                    & 4.2    & 14.9   & 26.6    & 44.0    & 78.3    \\
    Total Solve              & 279.2  & 484.4  & 773.6   & 1077.0  & 1542.3  \\
    QFS \% of Total          & 5.3    & 8.3    & 10.9    & 18.2    & 22.6    \\
    \bottomrule
  \end{tabular}
  \caption{Tabulation of the results from the Stokes problem from \Cref{section:large_scale:stokes}. See discussion and \Cref{figure:stokes:solution,figure:stokes:refinement} for further analysis.}
  \label{table:stokes}
\end{table}


\bfi  
\raisebox{.2in}{\ig{height=1.5in}{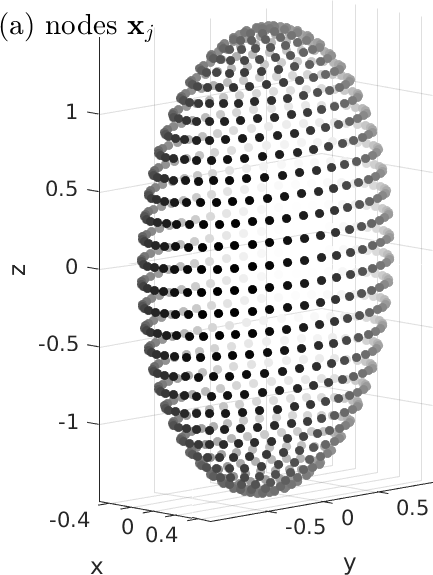}}
\;
\ig{height=1.7in}{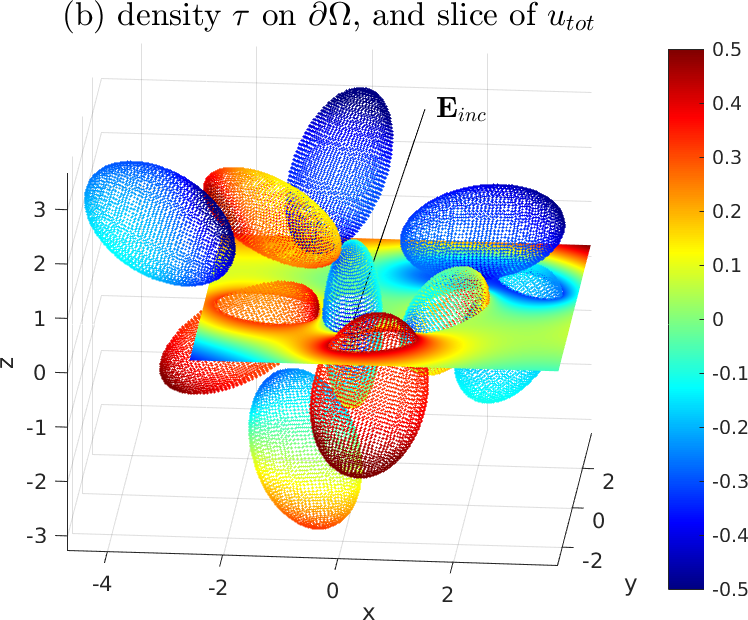}
\;
\ig{height=1.7in}{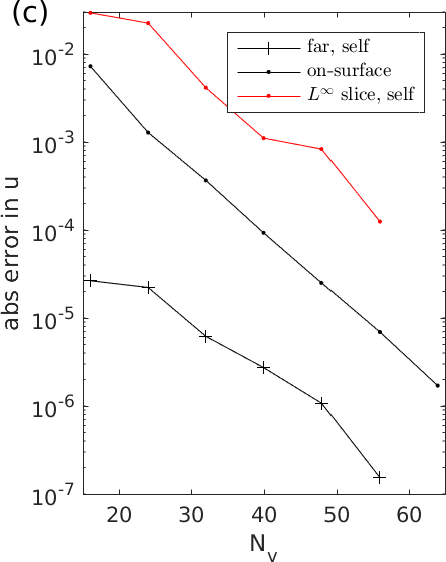}
\ca{3D Exterior Laplace BVP example for $K=10$ triaxial ellipsoids, with
  minimum separations $\ddist=0.1$, solved and evaluated with QFS-D.
  (a) shows the $N_0=1408$ surface nodes for the resolution $N_v=32$.
  (b) shows the physical potential
  $u_\tbox{tot}$ (using the colorscale)
  on a slice $z\approx 0.192$,
  and density $\tau$ (multiplied by 0.3 to fit the same colorscale),
  for $N_v=64$.
  (c) shows convergence of the solution potential $u$ at various points.
  At the highest $N_v=64$, there are $N=36040$ total degrees of freedom.
}{f:3d}
\efi

\section{Laplace 3D implementation and test}
\label{s:3d}

We now describe a preliminary dense (non-accelerated) test of QFS-D
in the exterior of several
identical ellipsoids of semiaxes $(1/2,1,3/2)$, ie, aspect ratio 3.
Since the ellipsoid is triaxial,
quadrature techniques for bodies of revolution
(eg \cite{klintporous,laiaxi}) do not apply.

We ``grew'' a cluster $\Omega$ of $K=10$ such ellipsoids, each a distance $\ddist=0.1$
from at least one other, as follows: for each new body, after choosing a
random orientation in SO(3), we translate it along a line with random orientation pointing towards $\mbf{0}$ until the minimum distance to any other body approximates $\ddist$ to $10^{-6}$.\footnote{The distance between any two ellipsoids is found by alternating
projection, with each projection using a Newton iteration for a Lagrange multiplier.}
We model an electrostatics problem where the $j$th conducting body has constant
voltage $V_j$ (chosen at random in $[-\half,\half]$), plus there is
an applied electric field $\mbf{E}_\tbox{inc}$ (with strength 0.3
in the direction shown in Fig.~\ref{f:3d}(b),
imposing a voltage drop of 2.25 across the cluster).
The physical potential is $u_\tbox{tot} = u_\tbox{inc} + u$,
where $u_\tbox{inc}(\x) = \mbf{E}_\tbox{inc} \cdot \x$,
whereas $u$ solves the $d=3$ Laplace BVP \eqref{Lpde}--\eqref{Sig}
with data on the $j$th boundary
$f_j(\x) = V_j - u_\tbox{inc}(\x)$, $\x\in \pO_j$.

We use the completed representation $u=({\cal D}+{\cal S})\tau$
in $\R^3\backslash\overline{\Omega}$,
leading to the well-conditioned BIE
$$
(\half + D + S)\tau \; = \; f~,
$$
which we discretize via QFS-D as follows.
The ellipsoid (axis-aligned at the origin) is
parameterized
$\r(u,v) = (\sqrt{1-v^2}(\cos u)/2,\sqrt{1-v^2}\sin u,3v/2)$, for
$(u,v)\in[0,2\pi]\times[-1,1]$.
Our spectrally-accurate global surface quadrature
is controlled by $N_v$, the number of Gauss-Legendre nodes $v_j$ covering
$v\in[-1,1]$. On each loop $v=v_j$ we place
an $n_j$-node periodic trapezoid rule in $u$, where
$n_j$ is the smallest even number larger than $\min\bigl[(4N_v/3)(1-v_j^2)^{1/2},8\bigr]$.
This scales $n_j$ by the loop circumference to uniformize the node density.
Weights are the products of the 1D rule weights and the Jacobian
$\|\r_u \times \r_v\|$.
There are about $N_0=0.9 N_v^2$ nodes on the ellipsoid;
see Fig.~\ref{f:3d}(a).

QFS sources are located by simple constant normal displacement
without upsampling: $\y_j = \x_j- \delta\n_j$,
where $\delta=0.08$ (somewhat below the smallest radius of curvature)
was chosen by experiment.
Similarly, the exterior check points are $\z_j = \x_j+\delta_c\n_j$ with
$\delta_c = 0.01$. We chose an upsampling factor $\rho=3$,
sufficient for around 9 digits of accuracy at check points.

The only missing ingredient is a spectral upsampling matrix $L_{\tilde{N}\times N}$ that maps values at nodes from a $N$-node rule controlled by $N_v$ to an $\tilde N$-node rule controlled by $\tilde{N}_v = \rho N_v$.
In brief this applies length-$n_j$ 1D discrete Fourier transforms (DFTs)
on the $j$th loop, zero-pads all modes up to the frequency $\max_j n_j/2$,
applies barycentric Lagrange interpolation onto new nodes in the
$v$ direction (separately for each Fourier mode),
then applies zero-padded length-$\tilde n_j$
inverse 1D DFTs to recover values on the $j$th output loop.
In practice this is a chain of matrix-matrix or Kronecker products.
A subtlety is that for odd modes only, the Lagrange interpolation must be
performed on the function divided by $\sqrt{1-v^2}$,
recalling that associated Legendre functions $P^{m}_n(v)$ with odd $m$
are polynomials in $v$ multiplied by this factor
\cite[Sec.~12.5]{arfken}.
Our code to fill $L_{\tilde{N}\times N}$ is about 30 lines of MATLAB.

We now have all the QFS-D pieces, so
use Algorithm~\ref{g:qfs-d} (LU variant as in Remark~\ref{r:LU})
to precompute $L$, $U$ and $\bar{P}C$,
using the $d=3$ Laplace kernel \eqref{LG} and the pure SLP $(\ta,\tb)=(1,0)$
QFS mixture.
Since we do not seek accuracies near $\emach$,
we simply store $X = U^{-1}(L^{-1}(\bar P C))$ then get the
$N_0\times N_0$ 1-body Nystr\"om matrix $A_0 = BX$.
\begin{rmk}
  The spectrum of $A_0$ may be improved by
  two-sided averaging \cite{qbx,hedgehog}: $A_0$ becomes the average
  of interior and exterior QFS discretizations, with the $I/2$ jump
  term then added explicitly.
  This gives $\kappa(A_0)\in[2.5,2.9]$ for all resolutions tested.
\end{rmk}
The dense $KN_0\times KN_0$ Nystr\"om matrix $A$ is now filled with $A_0$
as diagonal blocks, and offdiagonal blocks $A^{(i,j)} = B^{(i,j)} X$,
where $B^{(i,j)}$ is a matrix evaluating the SLP kernel $G$
from the QFS source locations for body $j$ to the nodes of body $i$.

GMRES with tolerance $10^{-8}$ is used with dense matrix-vector multiplication,
requiring exactly 24 iterations
for all resolutions tested, apart from the smallest $N_v=16$.
Thus we do not use one-body preconditioning.
The density $\tau$ solving the linear system is shown in Fig.~\ref{f:3d}(b).

The convergence of various errors in $u$
with $N_v$ is shown by Fig.~\ref{f:3d}(c), and appears to be spectral.
At a ``far'' target ($\x=(1,-1,2)$, a distance 1.15 from the nearest
body), 7-digit accuracy is reached by $N_v=56$,
or $N_0=2792$ per body, estimated by self-convergence.
At a generic on-surface target on body $j=1$
the value of $u_\tbox{tot}-V_1$ (and hence the error, shown with black dots),
reaches close to 6-digit accuracy at $N_v=64$.
A tougher test is the $L^\infty$ error over a 2D
slice of $54246$ exterior targets
(see Fig.~\ref{f:3d}(b)) with grid spacing 0.025,
and passing through at least one nearest-touching point,
and including a target $2\times 10^{-5}$ from one of the bodies.
Shown by the red curve, this reaches only 4-digit accuracy at $N_v=56$,
although the rate seems the same.
As expected, the worst errors occur at near-touching regions
and are oscillatory at the node scale.

\begin{rmk}
  Out of curiosity we have made this BVP challenging by imposing
  $\bigO(1)$ voltage differences
  between bodies, resulting in large fields $\|\grad u\|\approx10$,
  and density near-singularities at close-touching points.
  This may be analogous to velocity differences that occur in Stokes
  with rigid bodies \cite{mobility,corona3dmob,junwang}.
  With the same $u_\tbox{inc}$, if we set all $V_j=0$, the BVP
  becomes easier, the densities nonsingular, and all
  $L^\infty$ and on-surface errors improve by at least 1 digit.
\end{rmk}

Even though our implementation was naive, using a dense $A$,
timings were reasonable.
We worked in MATLAB
on a laptop with a quad-core Intel i7-7700HQ CPU and 32 GB RAM.
At $N_v=40$ (giving uniform 3-digit accuracy), the entire calculation
is done in 11 seconds.
At the largest $N_v=64$, two-sided QFS-D took 30 s to fill $A_0$
(dominated by wielding $L_{\tilde N\times N}$),
70 s to fill $A$, and 14 s for its GMRES solution.
Storing this $A$ needs 10 GB; obviously
an FMM-accelerated version would not have this limitation.
New densities can be converted into QFS source strengths at a rate
of at least $10^6$ points/s; this would enable a 3D FMM to perform
accurate evaluations close to or on surfaces with little extra cost.




\section{Conclusions}
\label{s:conc}

We have explored in depth, analytically and numerically, a proposal
to use an {\em effective source representation} for
the efficient spectrally-accurate
evaluation of layer potentials living on simple curves and surfaces.
The map to source strengths is precomputed by {\em collocation} either on
the boundary (QFS-B) or on a nearby ``check boundary'' (QFS-D).
The latter needs only a family of {\em smooth} quadratures on $\pO$,
and a high-order upsampling (interpolation) rule between members of the family.
We show that, with 2D periodic trapezoid nodes,
error performance is similar to the best-known
schemes: Kress for on-surface and expensive
adaptive quadrature for off-surface.
We expect it to add to the toolkit for large-scale simulations in
complex media, including viscous flows, wave scattering, electrostatics
(three cases we study here), as well as
sedimentation, vesicle dynamics, Maxwell, elastostatics, and elastodynamics.

The counterintuitive underlying idea---solving an ill-conditioned
1st-kind integral equation
to give a new global evaluator for a well-conditioned 2nd-kind integral
equation---brings several advantages:
distant, near, and on-surface targets all use the {\em same}
accurate representation
(making acceleration almost trivial, given a point-FMM code),
singular quadratures are replaced with
an upsampled smooth rule (as in QBX \cite{qbx,ce} or hedgehog
\cite{hedgehog}),
and the method is kernel-independent, allowing
easy implementations for various scalar and vector PDEs.

Our scheme is efficient when there are many simple bodies.
This hangs on the philosophy that {\em it is worth spending
  a lot of effort to create a good layer-potential representation
  that will be reused a huge number of times
  (GMRES iterations, simulation time-steps, etc).}
Since the idea is essentially a precomputed solution operator for
the method of fundamental solutions (MFS),
it comes with the same caveats
about the shape as the MFS.
While even 2D corners
can be handled by the MFS \cite{hochmancorner,larrythesis,lightning},
in 3D there is probably a limitation to simple smooth bodies.

On the theory side, we proved robustness
(assuming potential value collocation on the check curve),
showing that general conditions (C1-C2) hold for three common PDEs,
in 2D and 3D, for sufficiently analytic data.
Our discrete analysis in 2D invoked MFS
and BIE literature, but aspects such as the ratio condition
\eqref{ratio} seem more difficult to analyze.


\begin{rmk}[Why not just use MFS?]
  Given the success of 1st-kind representations
  ``under the hood'' of QFS, the reader may wonder
  whether one should just instead use (one-body preconditioned) MFS
  to solve the entire multi-body BVP. Such a method has utility
  (eg \cite{acper}).
  However, this would not fit within our goal of providing a
  general black-box layer-potential evaluator tool.
\end{rmk}




Future work suggested by this study includes i) application
to Neumann and other boundary conditions,
ii) FMM-accelerated 3D mobility solvers,
and iii) clustered MFS source locations to handle corner domains
\cite{hochmancorner,larrythesis,lightning}.
There also remain interesting analysis questions such
as understanding upsampling factors for 2D Stokes on the disk.






  \begin{acknowledgements}
    We are grateful for discussions with Manas Rachh, and the use of his
    2D biharmonic FMM code.
We thank Ralf Hiptmair for asking a question
(along the lines of ``why can't interior multipoles be used to precompute a quadrature for a rigid object?'') at an ICOSAHOM 2018 talk that helped inspire this work.
The Flatiron Institute is a division of the Simons Foundation.
\end{acknowledgements}

\appendix
\section{Robustness of continuous QFS representations for analytic data
  in three PDEs}
\label{a:mfs}

Here we prove Theorem~\ref{t:qfslap}, then state and prove
versions for Helmholtz and Stokes, which need adjustments.
We use ideas from Doicu--Eremin--Wriedt \cite[Ch.~IV, Thms.~2.1-2]{doicu}, who considered the only the pure SLP for Helmholtz.
These theorems are thus also useful for any MFS (first-kind IE) method
for exterior BVPs.

We consider $\Omega\subset\R^d$ with smooth boundary $\pO$,
and $\gamma\subset\Omega$ a smooth simple closed source curve (in $d=2$)
or source surface ($d=3$),
enclosing a domain $\Omega_-\subset\Omega$.
We abbreviate $u_\n := \partial u / \partial \n$.

\begin{proof}[Proof of Theorem~\ref{t:qfslap}.]
  Using $u$ to also denote the continuation of the solution,
  one may read off its data on $\gamma$,
  and the exterior Green's
  representation formula \cite[(1.4.5)]{HW} holds,
  \be
  u(\x) = \int_\gamma \biggl[ -G(\x,\y) u_\n(\y) + \frac{\partial G(\x,\y)}{\partial \n_\y} u(\y) \biggr] ds_\y~, \qquad \x \in \R^d\backslash\overline{\Omega_-}~.
  \label{GRFe}
  \ee
  Note that, by the decay condition, no constant term is needed.
  Let $v$ be the unique solution to
  the Laplace BVP interior to $\gamma$ with Dirichlet
  data $v=u$ on $\gamma$, then let $v_\n^-$ be its normal derivative,
  then the exterior extinction GRF holds
  \be
  0 = \int_\gamma \biggl[ G(\x,\y) v_\n^-(\y) - \frac{\partial G(\x,\y)}{\partial \n_\y} v(\y) \biggr] ds_\y~, \qquad \x \in \R^d\backslash\overline{\Omega_-}~.
  \label{GRFn}
  \ee
  Adding the last two equations cancels the DLP terms, leaving \eqref{Lmfsrep}
  with
  $$
  \sigma(\y) := v_\n^-(\y) - u_\n(\y)~,\qquad \y\in\gamma~,
  $$
  a density solving \eqref{Lmfsie}.
  Since all data on $\gamma$ was analytic, $\sigma$ is certainly smooth.
  
  For uniqueness,
  instead let $\sigma$ solve \eqref{Lmfsie} with zero RHS.
  Construct $u$ via \eqref{Lmfsrep} from this $\sigma$, then $u$ vanishes
  on $\pO$ by the uniqueness of the exterior Dirichlet BVP in $d=3$,
  or by Lemma~\ref{l:0const} in $d=2$, $u\equiv0$ in $\R^d\backslash\Omega$.
  By unique continuation from Cauchy data $u\equiv u_n\equiv 0$ on $\pO$,
  $u$ vanishes also in $\R^d\backslash\overline{\Omega_-}$.
  Since the potential is continuous across a single-layer \cite[Thm.~6.14]{LIE},
  and $u$ is harmonic in $\Omega_-$, then $u$ solves the interior
  Dirichlet BVP in $\Omega_-$ with vanishing data. By uniqueness of this
  BVP, $u$ vanishes in $\Omega_-$, thus both limits of $u_\n$ either side
  of $\gamma$ vanish, so by the jump relation \cite[Thm.~6.18]{LIE}, $\sigma\equiv0$.
\end{proof}


We now state and prove variants for the other two PDEs tested in this work.

\begin{thm}[QFS robustness for exterior Helmholtz]
  Let $u$ solve \eqref{Hpde} and \eqref{SRC} for $k>0$, with $u=u_c$ on $\pO$.
  Let $u$ also continue as a Helmholtz
  solution throughout the closed
  annulus (or shell) between $\pO$ and a simple smooth interior surface $\gamma\subset\Omega$. Let $\eta\in\R$, $\eta \neq 0$.
  Let $G$ be the fundamental solution \eqref{HG}.
  Then the first kind combined-field integral equation
  \be
  \int_\gamma \left[ \frac{\partial G(\x,\y)}{\partial \n_\y} - i \eta G(\x,\y) \right] \sigma(\y) ds_\y = u_c(\x), \qquad \x\in\pO
  \label{Hmfsie}
  \ee
  has a unique solution $\sigma\in C^\infty(\gamma)$, and
  \be
  u(\x) = \int_\gamma \left[ \frac{\partial G(\x,\y)}{\partial \n_\y} - i \eta G(\x,\y) \right] \sigma(\y) ds_\y~, \qquad \x\in\Omega\backslash \R^d~.
  \label{Hmfsrep}
  \ee
  \label{t:qfshelm}.
  \end{thm}
\begin{proof}
  We use $u$ to denote the
  continuation of $u$ as a Helmholtz solution onto $\gamma$.
  Let $v$ solve the homogeneous
  Helmholtz impedance BVP in the interior of $\gamma$,
  with boundary data
  \be
  v_\n-i\eta v \;=\; u_\n-i\eta u \qquad \mbox{ on } \gamma~.
  \label{helmimp}
  \ee
  It is standard that
  this BVP has a unique solution for any real $k$ \cite[Sec.~8.8]{SBH19}
  (or \cite[Prop~2.1]{iti}).
  Adding the Helmholtz versions of the GRFs \eqref{GRFe}
  (which applies since $u$ is radiative \cite[Sec.~2.2]{coltonkress})
  and \eqref{GRFn},
  and adding and subtracting $i \eta (v-u)$, we get
  $$
  u(\x) = \int_\gamma \left[ G(\x,\y)[v_\n-i \eta v - (u_\n-i \eta u)
    + i\eta (v-u)] - \frac{\partial G(\x,\y)}{\partial \n_\y} (v-u) \right] ds_\y
  $$
  which, by \eqref{helmimp} simplifies to give \eqref{Hmfsrep}
  with $\sigma:= (u-v)|_\gamma$. Choosing $\x\in\pO$ shows that
  $\sigma$ solves \eqref{Hmfsie}.
  
  Uniqueness follows by similar arguments as Laplace:
  instead let $\sigma$ solve \eqref{Hmfsie} with zero RHS,
  then let $u$ be given by \eqref{Hmfsrep}. Then $u$ vanishes on
  $\pO$, so by the uniqueness of the exterior Dirichlet BVP \eqref{Hpde}--\eqref{SRC}, $u$ also vanishes throughout $\R^d\backslash\Omega$.
  Since $u$ is analytic \cite[Thm.~2.2]{coltonkress},
  by unique continuation $u$ also vanishes in $\R^d\backslash\overline{\Omega_-}$.
  By the jump relations its interior limits of $u$ on $\gamma$ are
  $u^- = -\sigma$ and $u^-_\n = -i\eta\sigma$.
  Thus $u_\n^--i\eta u^- = 0$ on $\gamma$, and by construction
  $u$ is also a Helmholtz solution in $\Omega_-$.
  By the uniqueness of the impedance BVP in $\Omega_-$, then
  $u\equiv 0$ in $\Omega_-$,
  so, again by either jump relation, $\sigma\equiv 0$.
\end{proof}

\begin{thm}[QFS robustness for exterior Stokes velocity evaluation]
  Let $(\u,p)$ solve \eqref{Spde1}--\eqref{Spde2} in $\Omega\backslash \R^d$,
  with $\u=\u_c$ on $\pO$,
  and decay condition at infinity
  $\u(\x) = \bSig \log\|\x\| + o(1)$ in $d=2$
  or $\u(\x) = o(1)$ in $d=3$
  (ie, zero constant term).
  Let $(\u,p)$ continue analytically as a Stokes solution throughout the closed
  annulus (or shell) between $\pO$ and a simple smooth interior surface $\gamma\subset\Omega$.
  Let $d=3$, or for $d=2$ let the $2\times 2$ matrix mapping $\bSig$ to $\bom$
  in \eqref{bSig} be nonsingular for both $\pO$ and $\gamma$.
  Let $G$ and $D$ be as in \eqref{SG}--\eqref{SD}.
  Then the (``completed'' $S+D$ representation) first kind integral equation
  \be
  \int_\gamma [G(\x,\y) + D(\x,\y)] \bsig(\y) ds_\y = \u_c(\x), \qquad \x\in\pO
  \label{Smfsie}
  \ee
  has a unique solution $\bsig\in C^\infty(\gamma)^d$, and
  \be
  \u(\x) = \int_\gamma [G(\x,\y) + D(\x,\y)] \bsig(\y) ds_\y~, \qquad \x\in\Omega\backslash \R^d~.
  \label{Smfsrep}
  \ee
  \label{t:qfssto}
\end{thm}
\begin{proof}
  The proof is as for Helmholtz but with $-i\eta$ replaced by 1.
  The Green's representation formulae \eqref{GRFe}--\eqref{GRFn}
  apply for the Stokes velocity field, with
  traction data $\T(\u,p)$
  (defined, eg, in \cite[Sec.~2.3.1]{HW})
  in place of normal derivative data,
  and
  $D(\x,\y)$ from \eqref{SD} in place of the scalar kernel
  $\partial G(\x,\y)/\partial\n_\y$.
  Then let $(\v,q)$ solve the homogeneous Stokes BVP
  interior to $\gamma$, with Robin (``impedance'') data
  \be
  \T(\v,q) + \v = \T(\u,p) + \u \qquad \mbox{ on }\gamma~.
  \label{stoimp}
  \ee
  A solution exists by Lemma~\ref{l:stoimp} below.
  Adding \eqref{GRFe} (which applies since $\u$ has a zero constant term),
  and \eqref{GRFn}, and adding and subtracting $\v-\u$, we get
  $$
  \u(\x) = \int_\gamma \left[ G(\x,\y)[\T(\v,q)+\v - (\T(\u,p)+\u) - (\v-\u)]
    - D(\x,\y) (\v-\u) \right] ds_\y
  $$
  which, by \eqref{stoimp} simplifies to give \eqref{Smfsrep}
  with $\bsig:= (\u-\v)|_\gamma$. Choosing $\x\in\pO$ shows that
  $\bsig$ solves \eqref{Smfsie}.
  This completes existence.
  The uniqueness proof is similar to Laplace, apart from the following.
  One needs uniqueness for the exterior Stokes Dirichlet BVP with zero
  constant term: in $d=2$ Lemma~\ref{l:0const} (logarithmic
  capacity condition) is replaced by the nonsingularity hypothesis for $\pO$ in the theorem statement.
  The unique continuation argument relies on each component of $\u$ being analytic \cite[p.~60]{Ladyzhenskaya}.
  The rest of the proof is as for Helmholtz, replacing $-i\eta$ by 1,
  with the uniqueness
  of the interior Robin BVP assured by Lemma~\ref{l:stoimp} below.
\end{proof}


\begin{lem}[Existence and uniqueness for Stokes interior Robin BVP]
  Let $\Omega\subset\R^d$, be bounded with smooth boundary $\pO$.
  Let $\f: \pO\to\R^d$ be given smooth data. 
  Let the vector field $\v$ and scalar function $q$ solve
  in $\Omega$ the Stokes equations $-\mu \Delta \v + \nabla q = \mbf{0}$ and
  $\nabla\cdot\v = 0$, with Robin data
  $\T(\v,q) + \v = \f$.
  Then this problem has at most one solution.
  In addition, let $d=3$, or $d=2$ and let the $2\times 2$ matrix mapping
  $\bSig$ to $\bom$ in \eqref{bSig} be nonsingular,
  then it has exactly one solution.
  \label{l:stoimp}
\end{lem}
\begin{proof}
  Uniqueness follows
  easily as in \cite[p.~83]{hsiao85} \cite[p.51]{manasthesis}.
  One uses $(\v,q)$ for both the solution pairs in
  Green's 1st identity \cite[p.~53]{Ladyzhenskaya} to get
  $$
  \frac{\mu}{2}\int_\Omega \| \nabla \v + \nabla \v^T \|_F^2 = \int_\pO \T(\v,q)\cdot \v\, ds~,
  $$
  where $F$ indicates the Frobenius norm of the $d\times d$ tensor.
  Applying the Robin condition with $\f\equiv\mbf{0}$ shows that
  the right-hand side is non-positive, so that both vanish, so that
  $\v\equiv\mbf{0}$.
  For existence, suppose that $\bphi \in C(\pO)^d$ solves the BIE
  \be
  (S + D^T + 1/2)\bphi \;=\; \f
  \label{stoimpbie}
  \ee
  where $D^T$ is the adjoint double-layer
  operator. Then $\v = {\cal S}\bphi$ solves the Stokes equations in
  $\Omega$ with the correct Robin data following from the jump relations,
  thus is a solution.
  By the Fredholm alternative, to prove existence for \eqref{stoimpbie},
  one may prove uniqueness for the adjoint BIE
  $(S + D + 1/2)\bpsi = \mbf{0}$.
  This is already known in $d=3$ \cite[Thm~2.1]{hebeker}.
  In $d=2$ the constant term again rears its ugly head \cite{hsiao85}, but
  given the hypothesis we prove uniqueness as follows.
  Let $\bpsi$ solve the homogeneous adjoint BIE, then construct
  $\w = ({\cal S} + {\cal D})\bpsi$
 and $r$ the corresponding pressure representation,
  which solve the modified exterior
  BVP \eqref{Spde1}--\eqref{bSig} with given $\bom=\mbf{0}$, but
  $\bSig$ arbitrary.
  By the $2\times 2$
  matrix nonsingularity hypothesis the exterior
  solution is unique, hence trivial.
  By the jump relations on $\pO$, the interior limits are $\w^{-} = \bpsi$
  and $\T(\w,r)^{-} = -\bpsi$, so that $(\w,r)$ solves the interior
  Robin BVP with zero data $\T(\w,r)^{-}+\w^- = \mbf{0}$. By uniqueness
  proved above, the solution is identically zero, so again by the
  jump relations, $\bpsi\equiv\mbf{0}$.
\end{proof}

We suspect that there is a way to remove the
above Stokes domain nonsingularity condition
in $d=2$, perhaps following \cite{hsiao85}.

\section{Geometry generation for large-scale 2D examples}
\label{a:geometry}

Here we present an algorithm to generate $K$ simple polar-Fourier shapes
$\pO_i$, $i=1,\dots,K$, located at randomly-generated centers, that obey
the distance and variation criteria of Sec.~\ref{section:large_scale:geometry}.
Its inputs are $\ddist$, a body radius scale $r_0$,
and a routine {\tt randomcenter} that returns fresh centers $\c$.

First we make a list of centers $\{\c_1,\dots,\c_K\}$ that are far enough apart. Starting with the empty list,
\ben
\item Generate a new candidate center $\c$ via {\tt randomcenter},
\item Append $\c$ to the list if $\c$ has distance at least $2r_0$ from
  all $\c_i$ in the list,
\item Repeat 1-2 until the list has $K$ centers.
  \een
  The body $\pO_i$ is now chosen from the star-shaped family defined
  about the center $\c_i$ by the polar parameterization
  $r(t)=r_0(1+a\cos(ft + \phi))$, where $a$ is the ``wobble'' amplitude, $f$ its frequency, and $\phi$ its rotation.
For example \Cref{f:setup}(a) shows $\c=\mbf{0}$,
$r_0=1$, $a=0.3$, $f=5$, $\phi=0.2$.
$f$ is drawn randomly from $\{3,4,\dots,7\}$ with
a distribution function $\{16/31, 8/31, 4/31, 2/31, 1/31\}$,
to include higher frequencies less often.
The amplitude $a$ is uniform random in $[0,0.3(3/f)^{3/2}]$.
Thus higher frequencies will tend to have smaller amplitudes, in order to prevent any single boundary from dominating the resolution requirements.
$\phi$ is uniform random in $[0,2\pi)$.
  Since the maximum radius of a body is currently $(1+a)r_0$,
  intersections are possible,
  and there may not be any bodies that are $\approx\ddist$ apart.
  Thus we use the following to adjust all body radii:
\begin{enumerate}
  \item All geometries are rescaled so that the farthest distance from center to boundary is $r_0+\ddist/2$. At this point, no boundaries can intersect and all boundaries must be separated by at least $\ddist$.
  \item 10\% of the geometries are chosen at random, and for each chosen geometry:
  \begin{enumerate}
    \item Denote the current maximum radius of the geometry by $R$.
    \item (\emph{expansion}) The geometry is rescaled to increase its radius by $\ddist$, and minimal separation distances between the geometry and all its nearest neighbors are computed.
    \item Step (b) is repeated until either the geometries current radius is $>1.5R$, or the geometry is separated from a nearest neighbor by $<\ddist$.
    \item If the prior step is terminated because the geometry is $<\ddist$ from a nearest neighbor, proceed to the next step; otherwise handling for this geometry is finished.
    \item (\emph{rescue}) The geometry is rescaled to decrease its radius by $\ddist/10$, and separation distances between the geometry and all its nearest neighbors are computed.
    \item Step (e) is repeated until the radius of the geometry is between $(\ddist, 1.1\ddist]$.
  \end{enumerate}
  \item Step 2 is repeated three times.
  \item For speed, the prior items are computed using approximate methods (distances are computed pointwise over barely-resolved boundaries), and in rare instances boundaries may be closer together than $\ddist$. A final \emph{rescue} step is performed for every boundary with upsampled boundaries and using full Newton iterations to compute the minimal distances.
\end{enumerate}
Although elaborate, this process allows us to efficiently place $K$ polydisperse boundaries in a specified manner throughout a domain, with a separation no less than $\ddist$ between the individual boundaries. When the initial set of boundaries are packed sufficiently tightly, the \emph{expansion} steps always produce at least some boundaries whose expansion is terminated because they are too close to others;
thus, due to how the {\em rescue} stage is implemented, there will always be some close pairs of boundaries separated by between $\ddist$ and $1.1\ddist$.

\bibliography{refs}
\bibliographystyle{abbrv}

\end{document}